\documentclass{amsart}
\usepackage{xypic}
\usepackage[all,cmtip]{xy}
\usepackage{tikz}
\usetikzlibrary{calc,decorations.markings,decorations.pathmorphing,arrows}
\usepackage{amsmath, amssymb, amsthm, amsfonts, mathrsfs, amsfonts}
\usepackage[a4paper,left = 2.5cm, right= 2.5cm, top = 2.5cm, bottom = 2.5cm]{geometry}
\usepackage{cases}
\usepackage{float} 
\usepackage{caption}

\newcommand{\N}{\ensuremath{\mathbb{N}}}
\newcommand{\Z}{\ensuremath{\mathbb{Z}}}

\newcommand{\C}{\ensuremath{\mathbb{C}}}
\newtheorem{thm}{Theorem}[section]

\newtheorem{cor}[thm]{Corollary}
\newtheorem{lemma}[thm]{Lemma}
\newtheorem{prop}[thm]{Proposition}
\theoremstyle{remark}
\newtheorem{rmk}[thm]{Remark}

\theoremstyle{definition}
\newtheorem{defin}[thm]{Definition}

\newcommand{\on}{\operatorname}
\newcommand{\s}{\mathfrak s}
\newcommand{\p}{\mathfrak p}
\newcommand{\uu}{\mathfrak{u'}}
\newcommand{\vv}{\mathfrak{v}}

\begin{document}

\title{Self-injective Jacobian algebras from Postnikov diagrams}
\author{Andrea Pasquali}
\address{Dept.~of Mathematics, Uppsala University, P.O.~Box 480, 751 06 Uppsala, Sweden}
\email{andrea.pasquali@math.uu.se}
\maketitle

\begin{abstract}
  We study a finite-dimensional algebra $\Lambda$ constructed from a Postnikov diagram $D$ in a disk, obtained from the dimer algebra of Baur-King-Marsh 
  by factoring out the ideal generated by the boundary idempotent.
  Thus, $\Lambda$ is isomorphic to the stable endomorphism algebra of a cluster tilting module $T\in\on{CM}(B)$ introduced by Jensen-King-Su in order to categorify the cluster algebra structure 
  of $\C[\on{Gr}_k(\C^n)]$. We show that $\Lambda$ is self-injective if and only if $D$ has a certain rotational symmetry. 
  In this case, $\Lambda$ is the Jacobian algebra of a self-injective quiver with potential, which implies that its truncated Jacobian algebras in the sense of Herschend-Iyama are 2-representation finite.
  We study cuts and mutations of such quivers with potential leading to some new 2-representation finite algebras.
\end{abstract}

\section*{Introduction}
In this article we study algebras constructed from $(k,n)$-Postnikov diagrams. These are configurations of oriented curves in the disk satisfying some axioms, and were defined in \cite{Pos06}
to study total positivity of the Grassmannian $\on{Gr}_k(\mathbb C^n)$. 
The combinatorial data of such a diagram has been shown in \cite{OPS15} to be equivalent to the data of a maximal noncrossing collection of $k$-element subsets of $\left\{ 1, \dots, n \right\}$.

To a Postnikov diagram $D$ one can associate (see \cite{BKM16}) a planar ice quiver with potential $(Q, W, F) = (Q, W, F)(D)$, and consider the frozen Jacobian algebra $A = A(D)$ (which is infinite dimensional).
If one then quotients out the idempotent $e$ corresponding to the boundary (frozen) vertices, one gets a quiver with potential $(\underline Q,\underline W)$ whose Jacobian algebra $\Lambda = \Lambda(D)$ is
finite dimensional, and this is the main object of our study.

One can also label the vertices of $Q(D)$ by the $k$-element subsets appearing in the maximal noncrossing collection corresponding to $D$.
Postnikov diagrams were used in \cite{Sco06} to show that the homogeneous coordinate ring of $\on{Gr}_k(\mathbb C^n)$ is a cluster algebra: the $k$-element subsets are labels for the Pl\"{u}cker coordinates, which are cluster variables.
The maximal noncrossing collections correspond precisely to clusters, and indeed the quiver $Q$ corresponds to the quiver of the cluster given by its collection.
By \cite{OPS15}, every cluster consisting of Pl\"ucker coordinates appears in this way (since all maximal noncrossing collections appear as the labels of such a quiver $Q(D)$).

There is an algebra $B$, depending only on $k$ and $n$, which was used in \cite{JKS16} to categorify the cluster algebra structure of the homogeneous coordinate ring of $\on{Gr}_k(\mathbb C^n)$, building on
the categorification of the coordinate ring of an affine open cell constructed in \cite{GLS08}.
The categorification takes place in $\on{CM}(B)$, where Jensen-King-Su define a Cohen-Macaulay $B$-module $L_I$ of rank 1 for every $k$-element subset $I$ of $\left\{ 1, \dots n\right\}$.
Given a maximal noncrossing collection $\mathbb I$, one can define a module 
$$
T = \bigoplus_{I\in \mathbb I} L_I
$$
and this module is shown in \cite{JKS16} to be cluster tilting in $\on{CM}(B)$. One of the main results in \cite{BKM16} is that there is an isomorphism 
\begin{align*}
  A\cong \on{End}_B(T)
\end{align*}
where $A$ is the frozen Jacobian algebra corresponding to the Postnikov diagram associated to $\mathbb I$.
The frozen vertices correspond to projective-injective $B$-modules, so there is an isomorphism 
\begin{align*}
  \Lambda \cong \underline{\on{End}}_B(T).
\end{align*}

It turns out that $\Lambda$ is the same if we take the completed algebra $\hat B$ instead of $B$, so we can use results about the completed case. 
The stable category $\underline{\on{CM}}(\hat B)$ is a 2-Calabi-Yau triangulated category with a cluster tilting object $T$, and we 
can use the machinery of \cite{IO13} to prove results about the algebra $\underline{\on{End}}_{\hat B}(T)$ (which is, as we said, isomorphic to $\Lambda$).

Our main result is that $\Lambda$ is self-injective if and only if $D$ is symmetric under rotation by $2k\pi/n$ (which corresponds to 
$\mathbb I$ being invariant with respect to adding $k$ to all elements).

Thus Postnikov diagrams turn out to be a new source of planar self-injective quivers with potential in the sense of \cite{HI11b}.
Previously, the only known planar self-injective quivers with potential were mutation equivalent to so called ``squares'', ``triangles'', or ``$n$-gons'', in the 
terminology of \cite[\S9]{HI11b}. The algebras coming in this way from ``$n$-gons'' are precisely the self-injective cluster tilted algebras classified by Ringel (\cite{Rin08}). 
We construct some new examples not belonging to the above families, thus answering \cite[Question 10.1(1)]{HI11b} in the negative. 
In fact, two counterexamples had already been found (and we recover them), but they were not published.
In particular, we construct an infinite family of algebras for which the Nakayama permutation has arbitrarily large order.
Previously, the only known self-injective planar quivers with potential with Nakayama permutation of order at least 6 were mutation equivalent to ``$n$-gons'', and our examples are not 
of this type.

Self-injective Jacobian algebras are precisely the 3-preprojective algebras of 2-representation finite algebras (\cite{HI11b}).
The latter can be constructed by choosing an appropriate set of arrows $C$ (called a cut) in the quiver.
We exploit the results of Herschend-Iyama to prove that for a given symmetric Postnikov diagram, all such 2-representation finite algebras $\Lambda_C$ are iterated 2-APR tilts of each other, so in particular they are derived equivalent.
Moreover, it is interesting to know when a cut is invariant under the Nakayama automorphism, since in this case $\Lambda_C$ is twisted $2\frac{l-1}{l}$-Calabi-Yau for some $l$ (\cite{HI11}).
In our setting, the Nakayama automorphism is simply given by rotation in the plane, so this condition is easily accessible.

One can study mutations of cluster tilting objects, of quivers with potential, or of Postnikov diagrams (the latter is called geometric exchange). These all correspond to each other, with the caveat that only certain vertices of the quiver
become mutable (since geometric exchange only works for some regions of the disk).
In \cite{Pos06} it is proved that geometric exchange is transitive on the set of $(k,n)$-Postnikov diagrams, and we deduce that mutation is transitive on cluster tilting objects in $\on{CM}(\hat B)$ whose 
indecomposable summands have rank 1.
We also give a direct proof of a special case of a theorem which appeared in \cite{HI11b} about mutations along a Nakayama orbit.

It should be noted that many of the statements we present are combinations of published results and probably known to experts, even though they cannot be found in the literature as we state them.
The original contributions of this paper are in the results of Section \ref{sec:main}, Section \ref{sec:cuts} and in the new examples of Sections~\ref{sec:ex} and \ref{sec:figs}.

The structure of this article is as follows. In Section \ref{sec:notation} we set up some notation and conventions. 
In Section \ref{sec:jacobian} we recall the definitions we need about ice quivers with potential and frozen Jacobian algebras.
In Section \ref{sec:post} we define Postnikov diagrams, explain their combinatorics and use them to construct ice quivers with potential.
In Section \ref{sec:selfinj} we collect some results about cluster tilting objects with self-injective endomorphism algebras.
In Section \ref{sec:B} we define the algebra $B$ and the modules $L_I$, as well as
compute the action of the Serre functor of $\underline{\on{CM}}(\hat B)$ on the modules $L_I$.
In Section \ref{sec:tilting} we define the module $T$ and study cluster tilting objects and their mutations in $\on{CM}(\hat B)$. We interpret those
mutations in terms of mutations of quivers with potential and geometric exchange.
In Section \ref{sec:main} we consider Postnikov diagrams which are rotation symmetric, and prove our main result.
In Section \ref{sec:cuts} we study cuts for self-injective quivers with potential arising from symmetric Postnikov diagrams.
In Sections \ref{sec:ex} and \ref{sec:figs} we present some examples of self-injective quivers with potential constructed in this way.
We recover an infinite family found in \cite{HI11b}, as well as some members of another infinite family. We construct a new infinite family, and 
finally some sporadic cases.

\subsection*{Acknowledgement}
I am thankful to my advisor Martin Herschend for the many helpful discussions and comments. I would like to thank Jakob Zimmermann for his help with the computational aspects of 
determining self-injectivity, and both him and Laertis Vaso for suggestions about the manuscript. 
I also thank the anonymous referees for spotting issues and suggesting improvements to previous versions of the paper.
Finally, my thanks go to Karin Baur, Alastair King and Robert Marsh, for their helpful comments.

\section{Notation and conventions}\label{sec:notation}
By an algebra $\Lambda$ we mean a unital, associative and basic $\mathbb C$-algebra unless otherwise specified. We write $\Lambda\on{mod}$ ($\on{mod}\Lambda$) for the category of finitely generated 
left (right) $\Lambda$-modules. If $\Lambda$ is graded by an abelian group $G$, we write $\Lambda\on{mod}^G$ ($\on{mod}^G\Lambda$) for the category of $G$-graded finitely generated left (right) $\Lambda$-modules.
In various contexts we will denote by $D$ the functor $\on{Hom}_\C(-, \C)$.
Unless otherwise specified, ``module'' means object of $\Lambda\on{mod}$.
If $\varphi:\Lambda\to \Lambda$ is a ring automorphism and $M\in \Lambda\on{mod}$, we define $\,_{\varphi}M\in \Lambda\on{mod}$ to be $M$ as an abelian group,
with $a*_{\varphi}m = \varphi(a)m$. Similarly we define $M_\varphi$ by $m*_\varphi a = m\varphi(a)$, for $M\in \on{mod}\Lambda$.
The composition $g\circ f$ means that $f$ is applied first and $g$ second.

Throughout this article, we will fix two positive integers $k\leq n$.
We denote by $[n]$ the set $\Z/n\Z$, usually equipped with the cyclic ordering. We write $\binom{[n]}{k}$ for the set of $k$-element subsets of $[n]$. For a subset $I$ of $[n]$, we write
\begin{align*}
  I+k := \left\{ i+k \ |\ i\in I \right\} \subseteq [n]
\end{align*}
and for a subset $\mathbb I$ of $\binom{[n]}{k}$, we write
\begin{align*}
  \mathbb I+k := \left\{ I+k\ |\ I\in \mathbb I \right\}\subseteq \binom{[n]}{k}.
\end{align*}

\section{Ice quivers with potential}\label{sec:jacobian}

In this section we recall some definitions, notation and facts about (ice) quivers with potential (see \cite{BIRS11} for a reference).
Let $Q = (Q_0, Q_1)$ be a finite quiver without loops and 2-cycles. We can complete the path algebra $\C Q$ with respect to the $\left<Q_1\right>$-adic topology, and denote the completion by $\widehat{\C Q}$. 
A \emph{potential} on $Q$ is an element 
\begin{align*}
  W\in \widehat{\C Q}\left/\overline{\left[\widehat{\C Q}, \widehat{\C Q}\right]}\right.,
\end{align*}
where $\left[\widehat{\C Q}, \widehat{\C Q}\right]$ is the vector space spanned by commutators in $\widehat{\C Q}$, and $\overline{\phantom{xx}}$ denotes closure in the $\left<Q_1\right>$-adic topology.
In other words, $W$ is a (possibly infinite) linear combination of cycles in $Q$, where we identify cycles up to cyclic permutation of their arrows.
We say that $W$ is finite if it can be written as a finite such linear combination.
For $a\in Q_1$, we can define the \emph{cyclic derivative} $\partial_a:\widehat {\C Q}\to \widehat {\C Q}$ by $$\partial_a(a_1\cdots a_l) = \sum_{a_i = a}a_{i+1}\cdots a_la_1\cdots a_{i-1}$$
and extended by linearity and continuity on $\widehat{\C Q}$. We also get an induced map $\partial_a:\widehat {\C Q}
\left/\overline{\left[\widehat{\C Q}, \widehat{\C Q}\right]}\to\widehat {\C Q}\right.$. 
\begin{defin}
  A \emph{quiver with potential} is a pair $(Q,W)$ where $Q$ is a quiver without loops and 2-cycles and $W$ is a potential on $Q$.
  The \emph{Jacobian algebra} $\hat\wp(Q, W)$ is the algebra
\begin{align*}
  \hat\wp(Q, W) = \widehat{\C Q}\left/\overline{\left<\partial_aW \ |\ a\in Q_1\right>}.\right.
\end{align*}
\end{defin}
We can generalise this definition slightly by allowing frozen vertices. 
\begin{defin}
  An \emph{ice quiver with potential} is a triple $(Q, W, F)$ where 
  $(Q, W)$ is a quiver with potential, and $F$ is a subset of $Q_0$ (the elements of $F$ are called \emph{frozen vertices}).
  Call $Q_F$ the set of arrows of $Q$ that start and end at a frozen vertex.
  The \emph{frozen Jacobian algebra} $\hat\wp(Q, W, F)$ is the algebra
  \begin{align*}
    \hat\wp(Q, W, F) = \widehat{\C Q}\left/\overline{\left<\partial_aW \ |\ a\in Q_1\setminus Q_F\right>}.\right.
  \end{align*}
  In other words, we do not take derivatives with respect to arrows between the frozen vertices.
\end{defin}
Given an ice quiver with potential $(Q, W,F)$, one can construct a quiver with potential $(\underline Q, \underline W)$ as follows. Set $\underline Q$ to be the quiver obtained from $Q$ by removing the frozen vertices and all adjacent arrows, 
and define $\underline W$ to be the image of $W$ under the quotient map $\widehat{\C Q}\to \widehat{\C\underline Q}$. 
Then we have $\hat\wp(Q, W, F)/\left< F\right>\cong \hat\wp(\underline Q, \underline W)$, where $\langle F\rangle$ is the ideal generated by the sum of the idempotents corresponding to vertices in $F$.

If $W$ is finite, we can also define a non-completed Jacobian algebra $\wp(Q, W)$ by the same construction without all the completions.
In this article, the quivers with potential $(\underline Q, \underline W)$ which appear have the property that the completed and non-completed Jacobian algebras are isomorphic. In the rest of this section, we lay the ground 
for proving this.
Let $(Q, W)$ be a quiver with finite potential. There is a canonical map $\wp(Q, W)\to \hat\wp(Q, W)$, but this map is in general neither injective nor surjective.
\begin{prop}\label{prop:admissible}
  If $(Q,W)$ is a quiver with finite potential such that $\left<\partial_aW \ |\ a\in Q_1\right>$ is an admissible ideal of $\C Q$, then the canonical map $\wp(Q, W)\to \hat\wp(Q, W)$ is an isomorphism.
\end{prop}

\begin{proof}
  Call $I = \left<\partial_aW \ |\ a\in Q_1\right>\subseteq \C Q$ and $\hat I = \overline{\left<\partial_aW \ |\ a\in Q_1\right>}\subseteq \widehat{\C Q}$.
  Call $J$ and $\hat J$ the arrow ideals of $\C Q$ and $\widehat{\C Q}$ respectively. 
  By assumption we have that there exists $N\gg0$ such that $J^N\subseteq I$ and then $\hat J^N\subseteq \hat I$.
  Observe that we have that $\widehat{\C Q} = \C Q + \hat J^N$, and that $ J^N = \C Q\cap \hat J^N$. 
  There is a commutative diagram
  \[
    \xymatrix{
      J^N\ar@{^{(}->}[r]\ar@{^{(}->}[d] & I \ar@{^{(}->}[r] \ar@{^{(}->}[d] & \C Q \ar@{->>}[r]\ar@{^{(}->}[d] & \wp(Q, W) \ar[d]\\
      \hat J^N\ar@{^{(}->}[r] & \hat I \ar@{^{(}->}[r] & \widehat{\C Q} \ar@{->>}[r]& \hat\wp(Q, W).
    }
  \]
  We get an induced commutative diagram
  \[
    \xymatrix{
      I/J^N \ar@{^{(}->}[r]\ar[d] & \C Q/J^N \ar@{->>}[r]\ar^{\cong}[d] & \wp(Q, W) \ar[d]\\
      \hat I/ \hat J^N \ar@{^{(}->}[r] & \widehat{\C Q}/\hat J^N \ar@{->>}[r] & \hat\wp(Q, W),
    }
  \]
  so it is enough to show that the map $I/J^N\to  \hat I/ \hat J^N $ is an isomorphism.
  This map is injective since $ J^N = \C Q\cap \hat J^N$. Moreover, $\hat I\subseteq I+\hat J^N$ since $I+\hat J^N$ is closed, so the map is surjective.
\end{proof}

\begin{cor}\label{cor:admissible}
  Let $(Q, W, F)$ be an ice quiver with potential. Suppose that $W$ is finite and that every sufficiently long path is equal in $\wp(Q, W, F)$ to a path that goes through a 
  frozen vertex. Then $\wp(\underline Q, \underline W)\cong \hat\wp(\underline Q, \underline W)$.
\end{cor}

\begin{proof}
  The assumption means exactly that the ideal $\left<\partial_a\underline W \ |\ a\in \underline Q_1\right>$ is admissible.
\end{proof}

\section{Postnikov diagrams}\label{sec:post}
Let us recall the definition of a $(k,n)$-Postnikov diagram (\cite[\S 14]{Pos06}, \cite[Definition 2.1]{BKM16}).
\begin{defin}
  A \emph{$(k,n)$-Postnikov diagram} $D$ consists of $n$ directed smooth curves (strands), in a disk with $n$ marked points on the boundary, clockwise labelled $1, 2, \dots, n$.
  The strands are also labelled, with strand $i$ starting at $i$ and ending at $i+k$. The following axioms must hold:
  \begin{enumerate}
    \item All crossings are transverse crossings between two distinct strands.
    \item There are finitely many crossings.
    \item Proceeding along a given strand, the other strands crossing it alternate between crossing it from the right and from the left.
    \item If two strands cross at distinct points $P_1$ and $P_2$, then one strand is oriented from $P_1$ to $P_2$ and the other from $P_2$ to $P_1$.
  \end{enumerate}
  For axiom $(3)$, we consider that strands cross at the boundary vertices in the obvious way.
  A Postnikov diagram is defined up to isotopy that fixes the boundary.
  Two Postnikov diagrams are \emph{equivalent} if they are related by a sequence of twisting and untwisting moves as shown in Figure \ref{fig:twisting}. The same moves with opposite orientations are also allowed.
  The moves have to be executed inside a disk with no other strand involved.
  A Postnikov diagram is \emph{reduced} if no untwisting moves can be applied to it. 
\end{defin}
\begin{figure}[h]
\[
\begin{tikzpicture}[baseline=(bb.base),scale = 0.6,
  doublearrow/.style={black, to-to,thick, line join=round,
decorate, decoration={
    zigzag,
    segment length=4,
    amplitude=.9,post=lineto, pre= lineto, pre length = 2pt,
  post length=2pt}}]
\newcommand{\goodarrow}{\arrow{angle 60}}
\newcommand{\dotrad}{0.1cm} 
\path (0,0) node (bb) {}; 


\draw  plot[smooth]
coordinates {(-3,1)  (0,-1)  (3,1)}
[ postaction=decorate, decoration={markings,
  mark= at position 0.5 with \goodarrow}];


\draw  plot[smooth]
coordinates { (3,-1) (0,1) (-3, -1)  }
[ postaction=decorate, decoration={markings,
 mark= at position 0.5 with \goodarrow}];


\draw  plot
coordinates {(7,1) (12,1)}
[ postaction=decorate, decoration={markings,
  mark= at position 0.5 with \goodarrow}];


\draw  plot
coordinates {(12,-1) (7,-1)}
[ postaction=decorate, decoration={markings,
  mark= at position 0.5 with \goodarrow}];

\draw [doublearrow] (4.5,0) -- (5.5,0);

\begin{scope}[shift={(0,-4)}]


\draw (-1.5,0) circle(\dotrad) [fill=black];
\draw (8.5,0) circle(\dotrad) [fill=black];


\draw  plot[smooth]
coordinates {(-1.5,0) (0,1)  (3,-1)}
[ postaction=decorate, decoration={markings,
  mark= at position 0.6 with \goodarrow}];


\draw  plot[smooth]
coordinates {(3,1)(0,-1) (-1.5,0)}
[ postaction=decorate, decoration={markings,
  mark= at position 0.8 with \goodarrow}];


\draw  plot[smooth]
coordinates {(8.5,0) (10, 1) (12,1)}
[ postaction=decorate, decoration={markings,
  mark= at position 0.5 with \goodarrow}];


\draw  plot[smooth]
coordinates {(12,-1)(10,-1) (8.5,0)}
[ postaction=decorate, decoration={markings,
  mark= at position 0.6 with \goodarrow}];

\draw [doublearrow] (4.5,0) -- (5.5,0);

\end{scope}

\end{tikzpicture}
\]
\caption{Twisting and untwisting moves in a Postnikov diagram.}
\label{fig:twisting}
\end{figure}
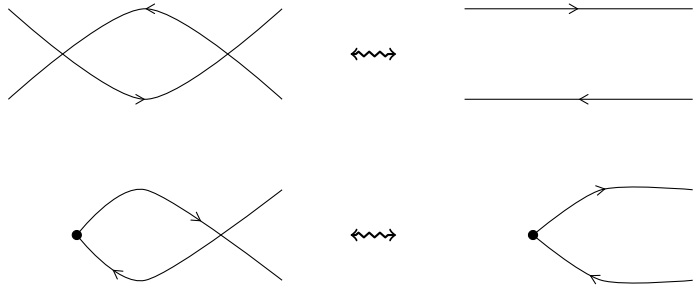
A Postnikov diagram divides the disk into regions, whose boundaries consist of strand segments and pieces of the boundary circle. There are three kinds of such regions,
according to whether their boundary is oriented clockwise, counterclockwise, or alternating in orientation (ignoring the boundary of the disk).
Each alternating region can be assigned a label $I\in \binom{[n]}{k}$ consisting of the names of the strands that have this region to their left side. These labels are all distinct.
Figure \ref{fig:post39} shows a reduced Postnikov diagram with labelled alternating regions. Not all Postnikov diagrams have rotational symmetry, but we are particularly interested in symmetric ones.
\begin{figure}[h]
\[
\begin{tikzpicture}[baseline=(bb.base)]

\newcommand{\goodarrow}{\arrow{angle 60}}
\newcommand{\bstart}{130} 
\newcommand{\ninth}{40} 
\newcommand{\qstart}{150} 
\newcommand{\radius}{4.8cm} 
\newcommand{\eps}{11pt} 
\newcommand{\dotrad}{0.07cm} 

\path (0,0) node (bb) {};


\draw (0,0) circle(\radius) [thick,dashed];

\foreach \n in {1,...,9}
{ \coordinate (b\n) at (\bstart-\ninth*\n:\radius);
  \draw (\bstart-\ninth*\n:\radius+\eps) node {$\n$}; }
  

\foreach \n in {1,2,3,4,5,6,7,8,9} {\draw (b\n) circle(\dotrad) [fill=black];}


\draw  plot[smooth]
coordinates {(b9) (b3)}
[ postaction=decorate, decoration={markings,
  mark= at position 0.2 with \goodarrow, mark= at position 0.3 with \goodarrow,
mark= at position 0.4 with \goodarrow,mark= at position 0.5 with \goodarrow, mark= at position 0.8 with \goodarrow}];
\draw  plot[smooth]
coordinates {(b3) (b6)}
[ postaction=decorate, decoration={markings,
  mark= at position 0.2 with \goodarrow, mark= at position 0.3 with \goodarrow,
mark= at position 0.4 with \goodarrow,mark= at position 0.5 with \goodarrow, mark= at position 0.8 with \goodarrow}];
\draw  plot[smooth]
coordinates {(b6) (b9)}
[ postaction=decorate, decoration={markings,
  mark= at position 0.2 with \goodarrow, mark= at position 0.3 with \goodarrow,
mark= at position 0.4 with \goodarrow,mark= at position 0.5 with \goodarrow, mark= at position 0.8 with \goodarrow}];
\draw  plot[smooth] coordinates {(b7)(230:\radius*25/40) (-30:.25*\radius) (70:\radius*25/40) (b1)}
[ postaction=decorate, decoration={markings,
  mark= at position 0.1 with \goodarrow, mark= at position 0.29 with \goodarrow,
  mark= at position 0.42 with \goodarrow, mark= at position 0.5 with \goodarrow,
  mark= at position 0.6 with \goodarrow, mark= at position 0.69 with \goodarrow,
  mark= at position 0.79 with \goodarrow, mark= at position 0.85 with \goodarrow,
  mark= at position 0.92 with \goodarrow }];
\draw  plot[smooth] coordinates {(b1) (110:\radius*25/40)(-150:\radius*.25) (-50:\radius*25/40)(b4)}
[ postaction=decorate, decoration={markings,
  mark= at position 0.1 with \goodarrow, mark= at position 0.29 with \goodarrow,
  mark= at position 0.42 with \goodarrow, mark= at position 0.5 with \goodarrow,
  mark= at position 0.6 with \goodarrow, mark= at position 0.69 with \goodarrow,
  mark= at position 0.79 with \goodarrow, mark= at position 0.85 with \goodarrow,
  mark= at position 0.92 with \goodarrow }];
\draw  plot[smooth] coordinates {(b4)(-10:\radius*25/40) (90:\radius*.25) (190:\radius*25/40)(b7)}
[ postaction=decorate, decoration={markings,
  mark= at position 0.1 with \goodarrow, mark= at position 0.29 with \goodarrow,
  mark= at position 0.42 with \goodarrow, mark= at position 0.5 with \goodarrow,
  mark= at position 0.6 with \goodarrow, mark= at position 0.69 with \goodarrow,
  mark= at position 0.79 with \goodarrow, mark= at position 0.85 with \goodarrow,
  mark= at position 0.92 with \goodarrow }];
\draw  plot[smooth] coordinates {(b2) (80:\radius*23/40)(30:\radius*7/40) (-25:\radius*26/40)(b5)}
[ postaction=decorate, decoration={markings,
  mark= at position 0.14 with \goodarrow, mark= at position 0.25 with \goodarrow,
  mark= at position 0.35 with \goodarrow, mark= at position 0.47 with \goodarrow,
  mark= at position 0.6 with \goodarrow, mark= at position 0.72 with \goodarrow,
 mark= at position 0.86 with \goodarrow}];
 \draw  plot[smooth] coordinates {(b5) (-40:\radius*23/40)(-90:\radius*7/40) (-145:\radius*26/40)(b8)}
 [ postaction=decorate, decoration={markings,
  mark= at position 0.14 with \goodarrow, mark= at position 0.25 with \goodarrow,
  mark= at position 0.35 with \goodarrow, mark= at position 0.47 with \goodarrow,
  mark= at position 0.6 with \goodarrow, mark= at position 0.72 with \goodarrow,
 mark= at position 0.86 with \goodarrow}];
 \draw  plot[smooth] coordinates {(b8) (200:\radius*23/40)(150:\radius*7/40) (95:\radius*26/40)(b2)}
 [ postaction=decorate, decoration={markings,
  mark= at position 0.14 with \goodarrow, mark= at position 0.25 with \goodarrow,
  mark= at position 0.35 with \goodarrow, mark= at position 0.47 with \goodarrow,
  mark= at position 0.6 with \goodarrow, mark= at position 0.72 with \goodarrow,
 mark= at position 0.86 with \goodarrow}];


\foreach \n/\m/\r in {1/789/0.88, 2/891/0.88, 3/912/0.88, 4/123/0.85, 5/234/0.89, 6/345/0.85, 7/456/0.85, 8/567/.89, 9/678/.87}
{ \draw (\qstart-\ninth*\n:\r*\radius) node (q\m) {$\m$}; }

\foreach \m/\a/\r in {179/88/0.6 , 134/328/0.6, 467/208/0.6, 178/117/0.4, 124/-3/.4, 457/237/.4 , 147/0/0, 127/65/.38, 145/-55/.38, 478/185/.38}
{ \draw (\a:\r*\radius) node (q\m) {$\m$}; }

 \end{tikzpicture}
\]
\caption{A  symmetric $(3,9)$-Postnikov diagram.}
\label{fig:post39}
\end{figure}
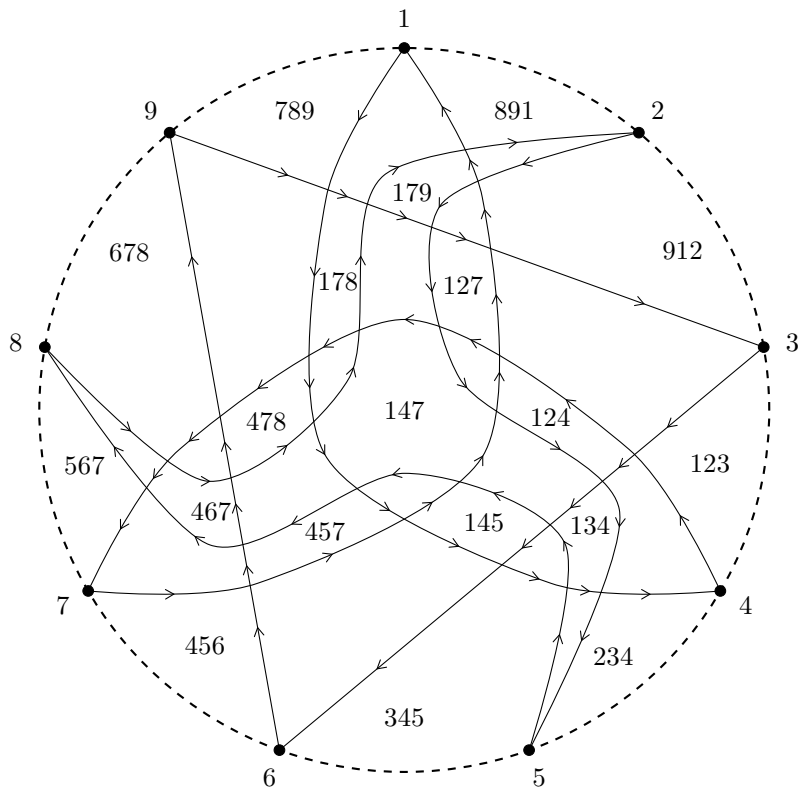
We call $\mathbb I = \mathbb I(D)$ the set of labels corresponding to $D$. 
\begin{defin}
  Two sets $I, J\in\binom{[n]}{k}$ are said to be \emph{noncrossing} or \emph{weakly separated} (see \cite[Definition 3]{Pos06}) if there exist no cyclically ordered $a, b, c, d\in [n]$
  with $a, c\in I\setminus J$ and $b, d\in J\setminus I$.
  We call a collection of $k$-element subsets of $[n]$ a \emph{noncrossing collection} if its elements are pairwise noncrossing. We call it a \emph{maximal noncrossing collection} if it is maximal with respect to inclusion.
\end{defin}

\begin{thm}\label{thm:classnoncross}
  \cite[Theorem 11.1]{OPS15}. Maximal noncrossing collections of elements of $\binom{[n]}{k}$ are precisely sets of labels of alternating regions in reduced $(k,n)$-Postnikov diagrams.
\end{thm}
Such collections are known to have $k(n-k)+1$ elements (this was conjectured in \cite{Sco06} and proved in \cite[Theorem 4.7]{OPS15}). 
There is an explicit construction of a Postnikov diagram having a prescribed maximal noncrossing collection as labels (\cite[\S 9]{OPS15}),
and this turns out to be unique (up to equivalence).
So the datum of a Postnikov diagram $D$ is equivalent to the datum of a maximal noncrossing collection $\mathbb I$. 
\begin{rmk}\label{rmk:labels}
  The label of the alternating region adjacent to the boundary segment of the disk from $i$ to $i+1$ is the set $\left\{ i-k+1, \dots, i \right\}$ for $i\in [n]$.
\end{rmk}

For simplicity, we assume from now on that Postnikov diagrams are reduced.
To a Postnikov diagram $D$ we can associate (see \cite[\S 3]{BKM16}) an ice quiver with potential $(Q, W, F)= (Q, W, F)(D)$ such that:
\begin{enumerate}
  \item Vertices of $Q$ are elements of $\mathbb I(D)$.
  \item Arrows of $Q$ correspond to intersection points of alternating regions, with orientation so that the arrows ``point in the same direction as the strands'', as illustrated in Figure \ref{fig:quiv39}.
  \item The potential $W$ is given by the sum of cycles corresponding to the clockwise regions minus the sum of the cycles corresponding to the counterclockwise regions.
  \item The frozen vertices are the boundary vertices, i.e.~the vertices corresponding to the boundary segments of the disk.
\end{enumerate}
Notice that there is a natural embedding of $Q$ in the disk. The assumption that $D$ is reduced implies that there are no 2-cycles in $Q$ (which we require in our definition of ice quivers with potential).
\begin{figure}[h]
\[
\begin{tikzpicture}[baseline=(bb.base),
quivarrow/.style={black, -latex, very thick}]

\newcommand{\goodarrow}{\arrow{angle 60}}
\newcommand{\bstart}{130} 
\newcommand{\ninth}{40} 
\newcommand{\qstart}{150} 
\newcommand{\radius}{4.8cm} 
\newcommand{\eps}{11pt} 
\newcommand{\dotrad}{0.07cm} 

\path (0,0) node (bb) {};


\draw (0,0) circle(\radius) [thick,dashed];

\foreach \n in {1,...,9}
{ \coordinate (b\n) at (\bstart-\ninth*\n:\radius);
  \draw (\bstart-\ninth*\n:\radius+\eps) node {$\n$}; }
  

\foreach \n in {1,2,3,4,5,6,7,8,9} {\draw (b\n) circle(\dotrad) [fill=black];}

\draw  plot[smooth]
coordinates {(b9) (b3)}
[ postaction=decorate, decoration={markings,
  mark= at position 0.2 with \goodarrow, mark= at position 0.3 with \goodarrow,
mark= at position 0.4 with \goodarrow,mark= at position 0.5 with \goodarrow, mark= at position 0.8 with \goodarrow}];
\draw  plot[smooth]
coordinates {(b3) (b6)}
[ postaction=decorate, decoration={markings,
  mark= at position 0.2 with \goodarrow, mark= at position 0.3 with \goodarrow,
mark= at position 0.4 with \goodarrow,mark= at position 0.5 with \goodarrow, mark= at position 0.8 with \goodarrow}];
\draw  plot[smooth]
coordinates {(b6) (b9)}
[ postaction=decorate, decoration={markings,
  mark= at position 0.2 with \goodarrow, mark= at position 0.3 with \goodarrow,
mark= at position 0.4 with \goodarrow,mark= at position 0.5 with \goodarrow, mark= at position 0.8 with \goodarrow}];
\draw  plot[smooth] coordinates {(b7)(230:\radius*25/40) (-30:.25*\radius) (70:\radius*25/40) (b1)}
[ postaction=decorate, decoration={markings,
  mark= at position 0.1 with \goodarrow, mark= at position 0.29 with \goodarrow,
  mark= at position 0.42 with \goodarrow, mark= at position 0.5 with \goodarrow,
  mark= at position 0.6 with \goodarrow, mark= at position 0.69 with \goodarrow,
  mark= at position 0.79 with \goodarrow, mark= at position 0.85 with \goodarrow,
  mark= at position 0.92 with \goodarrow }];
\draw  plot[smooth] coordinates {(b1) (110:\radius*25/40)(-150:\radius*.25) (-50:\radius*25/40)(b4)}
[ postaction=decorate, decoration={markings,
  mark= at position 0.1 with \goodarrow, mark= at position 0.29 with \goodarrow,
  mark= at position 0.42 with \goodarrow, mark= at position 0.5 with \goodarrow,
  mark= at position 0.6 with \goodarrow, mark= at position 0.69 with \goodarrow,
  mark= at position 0.79 with \goodarrow, mark= at position 0.85 with \goodarrow,
  mark= at position 0.92 with \goodarrow }];
\draw  plot[smooth] coordinates {(b4)(-10:\radius*25/40) (90:\radius*.25) (190:\radius*25/40)(b7)}
[ postaction=decorate, decoration={markings,
  mark= at position 0.1 with \goodarrow, mark= at position 0.29 with \goodarrow,
  mark= at position 0.42 with \goodarrow, mark= at position 0.5 with \goodarrow,
  mark= at position 0.6 with \goodarrow, mark= at position 0.69 with \goodarrow,
  mark= at position 0.79 with \goodarrow, mark= at position 0.85 with \goodarrow,
  mark= at position 0.92 with \goodarrow }];
\draw  plot[smooth] coordinates {(b2) (80:\radius*23/40)(30:\radius*7/40) (-25:\radius*26/40)(b5)}
[ postaction=decorate, decoration={markings,
  mark= at position 0.14 with \goodarrow, mark= at position 0.25 with \goodarrow,
  mark= at position 0.35 with \goodarrow, mark= at position 0.47 with \goodarrow,
  mark= at position 0.6 with \goodarrow, mark= at position 0.72 with \goodarrow,
 mark= at position 0.86 with \goodarrow}];
 \draw  plot[smooth] coordinates {(b5) (-40:\radius*23/40)(-90:\radius*7/40) (-145:\radius*26/40)(b8)}
 [ postaction=decorate, decoration={markings,
  mark= at position 0.14 with \goodarrow, mark= at position 0.25 with \goodarrow,
  mark= at position 0.35 with \goodarrow, mark= at position 0.47 with \goodarrow,
  mark= at position 0.6 with \goodarrow, mark= at position 0.72 with \goodarrow,
 mark= at position 0.86 with \goodarrow}];
 \draw  plot[smooth] coordinates {(b8) (200:\radius*23/40)(150:\radius*7/40) (95:\radius*26/40)(b2)}
 [ postaction=decorate, decoration={markings,
  mark= at position 0.14 with \goodarrow, mark= at position 0.25 with \goodarrow,
  mark= at position 0.35 with \goodarrow, mark= at position 0.47 with \goodarrow,
  mark= at position 0.6 with \goodarrow, mark= at position 0.72 with \goodarrow,
 mark= at position 0.86 with \goodarrow}];


\foreach \n/\m/\r in {1/789/0.88, 2/891/0.88, 3/912/0.88, 4/123/0.85, 5/234/0.89, 6/345/0.85, 7/456/0.85, 8/567/.89, 9/678/.87}
{ \draw (\qstart-\ninth*\n:\r*\radius) node (q\m) {$\m$}; }

\foreach \m/\a/\r in {179/88/0.6 , 134/328/0.6, 467/208/0.6, 178/117/0.4, 124/-3/.4, 457/237/.4 , 147/0/0, 127/65/.38, 145/-55/.38, 478/185/.38}
{ \draw (\a:\r*\radius) node (q\m) {$\m$}; }


\foreach \t/\h in {456/457, 457/145, 145/147, 134/145, 124/134,  134/234, 457/467, 467/567, 145/345,345/134, 123/124, 147/124, 127/147, 147/178, 
478/147, 147/457, 124/127, 179/127, 178/179, 178/478, 467/478, 127/912, 912/179, 179/891, 789/178, 478/678, 678/467, 234/123, 234/345, 567/456, 567/678, 891/912, 891/789, 345/456, 678/789, 912/123}
{ \draw [quivarrow] (q\t) edge (q\h);}

 \end{tikzpicture}
\]
\caption{The quiver associated to the Postnikov diagram in Figure \ref{fig:post39}.}
\label{fig:quiv39}
\end{figure}
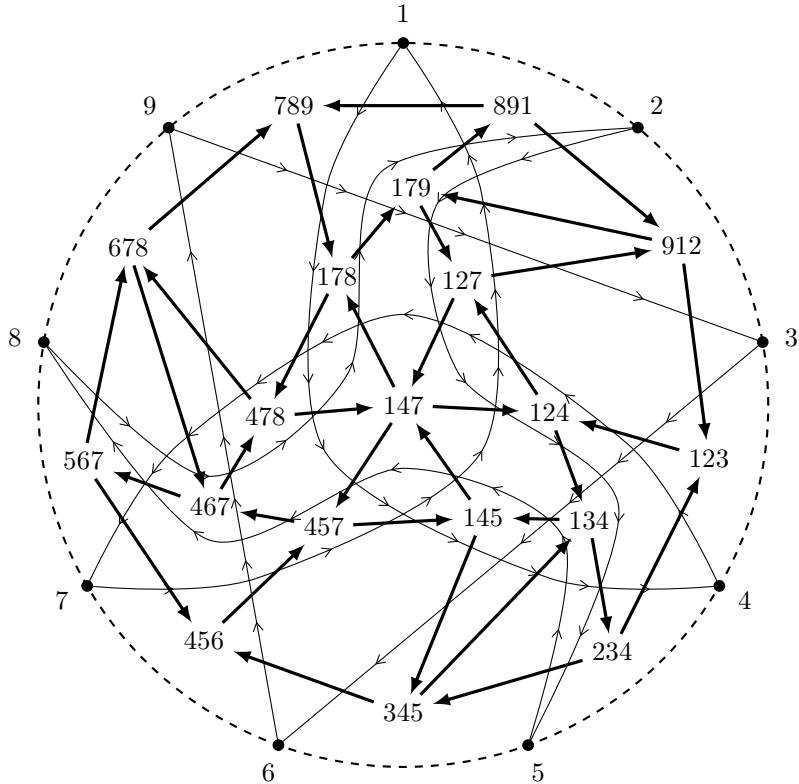

Thus we define the frozen Jacobian algebra $A = A(D) = \wp(Q, W, F)$ (this is the dimer algebra $A$ defined in \cite{BKM16}).
It is proved in \cite[Lemma 12.1]{BKM16} that the algebra $A$ is invariant up to isomorphism under equivalence of Postnikov diagrams.
Call $e$ the idempotent of $A$ given by the sum of the idempotents corresponding to the frozen vertices of $Q$. 
Then $eAe\subseteq A$ is an idempotent subalgebra isomorphic (see Section \ref{sec:tilting}) to the opposite of the algebra $B$ we discuss in Section \ref{sec:B}.
The algebra $eAe$ is the boundary algebra studied in \cite{BKM16}, and the algebra $B$ was introduced in \cite{JKS16}. We are especially interested in studying the algebra $\Lambda = A/AeA$. 
The latter is the Jacobian algebra $\wp(\underline Q, \underline W)$, where $\underline Q$ is the quiver obtained from $Q$ by removing the frozen vertices and the adjacent arrows, and $\underline W$ is the 
image of $W$ under the corresponding quotient map $\C Q\to \C\underline Q$ (see Section \ref{sec:jacobian}). 

We are interested in the case where the Postnikov diagram $D$ is symmetric under a rotation in the plane around the center of the disk. In particular, we consider invariance under $\rho$, the clockwise rotation by $2\pi k/n$. Since this notion is not invariant under isotopy, we call a Postnikov diagram \emph{symmetric} if it is equivalent to one which is invariant under $\rho$.
Another way of thinking about a symmetric Postnikov diagram is saying that it is equal (or isotopic) to the Postnikov
diagram obtained by changing the labels of the points on the disk, replacing every $i$ with $i+k$.
In this case we have
\begin{lemma}
  \label{lem:I}
   Let $\mathbb I$ be a maximal noncrossing collection in $\binom{[n]}{k}$. Then $\mathbb I = \mathbb I+k$ if and only if there exists a symmetric Postnikov diagram $D$ with $\mathbb I = \mathbb I(D)$.
\end{lemma}

\begin{proof}
   If $D$ is symmetric, it follows that if strand $i$ crosses in order the strands $i_1, i_2, \dots, i_l$, then strand $i+k$ crosses in order the strands $i_1+k, i_2+k, \dots i_l+k$.  
  Thus $I$ is the label of a region of $D$ if and only if $I+k$ is.

  Conversely, assume that $\mathbb I = \mathbb I+k$.
  We refer to \cite[\S 9]{OPS15} for the construction of a Postnikov diagram $D$ with $\mathbb I(D) = \mathbb I$. The construction proceeds by defining a 2-dimensional CW-complex $\Sigma(\mathbb I)$ whose vertex set is $\mathbb I$,
  embedding it in the plane,
and constructing strands as zig-zag paths. It is enough to observe that the images of the vertex sets of $\Sigma(\mathbb I)$ and of $\Sigma(\mathbb I+k)$ are related by rotation in the plane. This is true since the map is as follows.
One takes $v_1, \dots, v_n$ to be the vertices of a convex $n$-gon in $\mathbb R^2$, and one maps $I$ to $\sum_{a\in I} v_a$. We can in particular choose the $n$-gon to be regular and centred at the origin, and then the claim follows.
\end{proof}

Notice that $D$ is symmetric if and only if $Q$ is invariant under $\rho$. Moreover, $\rho$ must in this case map (counter-)clockwise cycles in $Q$ to (counter-)clockwise cycles, so 
it maps $W$ to itself, and so induces an automorphism $\Psi$ of $A$. Since $\rho$ maps $F$ to $F$, this induces an automorphism of $\Lambda$ which we still denote by $\Psi$.

We will need the following definition in Section \ref{sec:tilting}.

\begin{defin}\label{defin:u}\cite{BKM16}.
 From the relations in the definition of $A$ it follows that for any vertex $I\in Q_0$, the cyclic paths appearing in the potential and starting at $I$ are equal to the same element in $A$. We denote 
this element by $u_I\in A$, and define
\begin{align*}
  u = \sum_{I\in Q_0} u_I\in A.
\end{align*}
\end{defin}
\begin{rmk}
  It is easy to see that $u\in Z(A)$.
\end{rmk}

\section{Cluster tilting in 2-Calabi-Yau categories}\label{sec:selfinj}
Let $\mathcal C$ be a $\mathbb C$-linear, $\on{Hom}$-finite triangulated category.
\begin{defin}
  The category $\mathcal C$ is \emph{2-Calabi-Yau} if there is a functorial isomorphism
  \begin{align*}
    \on{Hom}_{\mathcal C}(X, Y)\cong D\on{Hom}_{\mathcal C}(Y, X[2]).
  \end{align*}
  An object $T\in \mathcal C$ is \emph{cluster tilting} if
  \begin{align*}
    \on{add}T = \left\{X\in \mathcal C\ |\ \on{Hom}_{\mathcal C}(T, X[1]) = 0\right\}.
  \end{align*}
  We call a cluster tilting object $T$ \emph{self-injective} if $\on{End}_\mathcal C(T)$ is a finite-dimensional self-injective $\mathbb C$-algebra.
  For convenience, we assume cluster tilting objects to be basic.
\end{defin}

Let us recall some facts and fix some notation about self-injective algebras.
Let $\Lambda$ be a finite-dimensional algebra, and let us fix a maximal set $\{e_1, \ldots, e_l\}$ of orthogonal idempotents. Then $\mathcal P_i = \Lambda e_i$ is a projective indecomposable $\Lambda$-module, and $\mathcal I_i = D(e_i\Lambda)$ is an injective indecomposable $\Lambda$-module.
If $T = \bigoplus T_i$ is a basic $B$-module for some algebra $B$, with indecomposable summands $T_i$, and $\Lambda = \on{End}_B(T)$,
then we choose $\mathcal P_i  = \on{Hom}_B(T_i,T)$ and $\mathcal I_i = D\on{Hom}_B(T,T_i)$.

An algebra $\Lambda$ is self-injective if and only if there exists an automorphism $\psi:\Lambda\to \Lambda$ such that $\Lambda_\psi\cong D\Lambda$ as $\Lambda-\Lambda$-bimodules.
This is called a Nakayama automorphism of $\Lambda$, and is unique up to inner automorphisms. In this case, we have that 
\begin{align*}
  \mathcal P_i \cong \mathcal I_{\sigma(i)}
\end{align*}
as left $\Lambda$-modules for some permutation $\sigma$ (i.e.~$\Lambda e_i \cong D(e_{\sigma(i)}\Lambda)$), and 
\begin{align*}
  \Lambda_\psi\otimes_{\Lambda} \mathcal P_{\sigma(i)} \cong \mathcal P_i
\end{align*}
for the same $\sigma$. This permutation does not depend on the choice of $\psi$, and is called the Nakayama permutation.

Let us now fix a 2-Calabi-Yau category $\mathcal C$. The following characterisation of self-injective cluster tilting objects will be useful.

\begin{prop}\label{prop:selfinj}\cite[Proposition 3.6]{IO13},  \cite[Proposition 4.4]{HI11b}.
  Let $T = \bigoplus_{i = 1}^l T_i\in \mathcal C$ be a cluster tilting object, with indecomposable summands $T_i$. Then 
  \begin{enumerate}
    \item $T$ is self-injective if and only if $T\cong T[2]$.
    \item In this case, the permutation $\sigma$ defined by $T_{\sigma(i)}\cong T_{i}[2]$ is the Nakayama permutation of $\on{End}_{\mathcal C}(T)$.
  \end{enumerate}
\end{prop}

\begin{proof}
  Part (1) is proved in \cite{IO13}. For part (2), observe
  \begin{align*}
    \mathcal P_i = \on{Hom}(T_{i}, T)&\cong D\on{Hom}(T, T_{i}[2])\cong D\on{Hom}(T, T_{\sigma(i)}) = \mathcal I_{\sigma(i)}.
  \end{align*}
\end{proof}

Now consider $\Lambda = \on{End}_\mathcal C(T)$ as in Proposition \ref{prop:selfinj}, and fix an isomorphism $\varphi: T\to T[2]$.
Then define an automorphism $\psi:\Lambda\to \Lambda$ by
\begin{align*}
  \psi(\lambda) = \varphi[-2]\circ\lambda[-2]\circ\left(\varphi[-2]\right)^{-1}.
\end{align*}

Then we have 

\begin{prop}
  \label{prop:naka}
  The map $\psi$ is a Nakayama automorphism of $\Lambda$.
\end{prop}

\begin{proof}
  First we define a left module morphism $\Lambda\to D\Lambda$. The Serre functor $[2]$ gives an isomorphism of bifunctors
  $$
  \on{Hom}_{\mathcal C}(-, ?)\cong D\on{Hom}_{\mathcal C}(?, -[2]).
  $$
  In our case this induces an isomorphism
  of vector spaces
  \begin{align*}
    \Lambda \to D\on{Hom}_\mathcal C(T, T[2])
  \end{align*}
  which we will denote by $a\mapsto a^*$ for $a\in \Lambda$.
Moreover, there is an isomorphism of vector spaces 
\begin{align*}
  D\on{Hom}_\mathcal C(T, T[2]) \to D\Lambda
\end{align*}
given by $F\mapsto F(\varphi\circ-)$.
Call $m:\Lambda\to D\Lambda $ the composition of these, i.e. $$m: a\mapsto a^*(\varphi\circ-)$$ for all $a\in \Lambda$.
Let us check that $m$ is a left module morphism. We have $(\lambda a)^* = a^*(-\circ \lambda)$ for $\lambda, a\in \Lambda$. So
\begin{align*}
  \lambda m(a) = a^*(\varphi\circ -\circ \lambda) = m(\lambda a).
\end{align*}
Let us prove that $m$ is a right module morphism $\Lambda_{\psi}\to D\Lambda$.
For $a, b\in \Lambda$, we have that $$m(ab) = (ab)^*(\varphi\circ -) =  a^*(b[2]\circ \varphi \circ -).$$
The right action of $\Lambda$ on $D\Lambda$ is given by $F\lambda = F(\lambda\circ-)$, so 
\begin{align*}
  (m(a))\lambda = a^*(\varphi\circ \lambda\circ-).
\end{align*}
On the other hand, 
\begin{align*}
  m(a*_\psi \lambda) &= m(a\circ\varphi[-2]\circ\lambda[-2]\circ\varphi^{-1}[-2]) = \\
  &= a^*( (\varphi[-2]\circ\lambda[-2]\circ\varphi^{-1}[-2])[2] \circ \varphi \circ-) = \\
  &= a^*(\varphi\circ \lambda \circ -) = (m(a))\lambda,
\end{align*}
which proves the claim.
\end{proof}

\section{The boundary algebra $B$}\label{sec:B}
In this section we discuss an algebra $B = B(k,n)$ which was introduced in \cite{JKS16} in order to categorify the cluster algebra structure of the coordinate ring of the Grassmannian $\on{Gr}_k(\C^n)$.
This algebra also plays a prominent role in \cite{BKM16}.

Let us consider a $\mathbb Z/n\mathbb Z$-grading on $\mathbb C[x,y]$ by $\on{deg}x = 1$ and $\on{deg}y = -1$.
Thus the element $x^k-y^{n-k}$ is homogeneous of degree $k$, and we can consider graded modules over $R= \mathbb C[x,y] /(x^k-y^{n-k})$. Denote degree shift on $
\on{mod}^{\Z/n\Z}R$ by $(1)$, and define $B$ to be the algebra
 $$B = \on{End}_R^{\mathbb Z/n\mathbb Z}\left(\bigoplus_{ i\in \mathbb Z/n\mathbb Z}R( i)\right).$$
 We can realise $B$ as a quiver algebra as follows. Consider the quiver with vertex set $[n]$, and arrows $x_i:i-1\to i$ and $y_i:i\to i-1$ for each $i\in [n]$. Call $x = \sum_i x_i$ and $y= \sum_i y_i$. Then $B$ is
 isomorphic to the quotient of the path algebra over $\mathbb C$ 
of this quiver by the ideal generated by the relations $xy=yx$ and $x^k = y^{n-k}$. Thus $B$ is a quotient of the preprojective algebra of type $\tilde A_{n-1}$ by the relation 
$x^k = y^{n-k}$.

We also need to introduce the completed algebra $\hat B$. This is the completion of $B$ with respect to the ideal $(x,y)$ \cite[Remarks 3.1, 3.2 and 3.4]{JKS16}.
Similarly, we write $\hat R = \C[ [x,y]]/\overline{(x^k- y^{n-k})}$.
The completion will turn out not to play an important role for us, due to Proposition \ref{prop:admissible}. 

The categories $\on{mod}B$ and $\on{mod}^{\mathbb Z/n\mathbb Z}R$ are equivalent, and similarly for $\hat B$ and $\hat R$.
We can consider the category $\on{CM}(B)$ of Cohen-Macaulay modules over $B$ and the category $\on{CM}^{\mathbb Z/n\mathbb Z}(R)$ of graded Cohen-Macaulay modules over $R$.
These also turn out to be equivalent (cf.~\cite[Corollary 3.7]{JKS16}), and again the same holds for the completed algebras.
The category $\on{CM}(\hat B)$ was studied in \cite{JKS16}, where the authors show that the Frobenius category $\on{Sub}Q_k$ used in \cite{GLS08} is a quotient of $\on{CM}(\hat B)$ by one indecomposable projective object.
In this way, many facts about $\on{CM}(\hat B)$ and its stable category $\underline{\on{CM}}(\hat B)$ can be deduced from what is known about $\on{Sub}Q_k$. In particular, we have
\begin{prop}
  The category $\underline{\on{CM}}(\hat B)$ is 2-Calabi-Yau.
\end{prop}

\begin{proof}
  This follows from \cite[Corollary 4.6]{JKS16} and \cite[Proposition 3.4]{GLS08}.
\end{proof}

Now we describe some additional properties of $B$, which are insensitive to the completion. 
We refer to \cite{JKS16} for a detailed discussion of the relationship between $B$ and $\hat B$. 

There is an automorphism $\Phi:B\to B$ given by mapping $e_i\mapsto e_{i+k}$, $x_i\mapsto x_{i+k}$ and $y_{i}\mapsto y_{i+k}$. The same function is also an automorphism $\Phi$ of $B^{opp}$. 

The center of $B$ is $Z = \mathbb C[t]\subseteq B$, where $t =xy$.
The algebra $B$ is finitely generated over $Z$, and the category $\on{CM}(B) $ consists exactly of the finitely generated $B$-modules that are free over $Z$. 
Such a module corresponds to a representation of the quiver of $B$ with at every vertex a free $Z$-module of the same rank (\cite[\S 3]{JKS16}).
Following \cite[Definition 3.5]{JKS16}, we say that a $B$-module has rank $d$ if it has rank $nd$ as a $Z$-module. Rank is additive over short exact sequences (cf.~\cite[\S3]{JKS16}), so in particular rank 1 modules are indecomposable.
\begin{defin}\cite[Definition 5.1]{JKS16}.
  For each $I\in \binom{[n]}{k}$, define the $B$-module $L_I$ of rank 1 by the following representation of the quiver:
  at every vertex $i$ we have a copy $U_i$ of $Z$, and 
  \begin{align*}
    x_i: U_{i-1}\to U_i \text{ acts as multiplication by }1 \text{ if } i\in I, \text{ and by }t \text{ else,}\\
    y_i: U_{i}\to U_{i-1}\text{ acts as multiplication by }t \text{ if } i\in I, \text{ and by }1 \text{ else.}
  \end{align*}
\end{defin}

Similarly, the center of $\hat B$ is $Z = \C[[t]]$, and all the above holds for $\hat B$. In particular, we will use the notation $Z$ for the center and $L_I$ for 
the modules defined above, both for $B$ and $\hat B$.

Such modules can be represented by lattice diagrams as in Figure \ref{fig:modex}, where the black dots on column $i$ represent the monomials $1, t, t^2, \dots$ in the corresponding $U_i$, the action of $x_i$ ($y_i$) 
is denoted by a rightward (leftward) arrow labelled $i$, and the edges of the figure are identified along the dotted lines.
The label $I$ can be read off from the arrows pointing to the right on the top profile of the diagram.

\begin{figure}[h]
\[
\begin{tikzpicture}[baseline=(bb.base),scale = 0.7,
quivarrow/.style={black, -latex,  very thick}]
\newcommand{\dotrad}{0.07cm} 
\path (0,0) node (bb) {}; 


\draw (0,0) circle(\dotrad) [fill=black];
\draw (1,1) circle(\dotrad) [fill=black];
\draw (2,0) circle(\dotrad) [fill=black];
\draw (2,2) circle(\dotrad) [fill=black];
\draw (3,1) circle(\dotrad) [fill=black];
\draw (3,3) circle(\dotrad) [fill=black];
\draw (4,0) circle(\dotrad) [fill=black];
\draw (4,2) circle(\dotrad) [fill=black];
\draw (5,1) circle(\dotrad) [fill=black];
\draw (6,0) circle(\dotrad) [fill=black];
\draw (6,2) circle(\dotrad) [fill=black];
\draw (7,1) circle(\dotrad) [fill=black];
\draw (8,2) circle(\dotrad) [fill=black];
\draw (8,0) circle(\dotrad) [fill=black];
\draw (9,3) circle(\dotrad) [fill=black];
\draw (9,1) circle(\dotrad) [fill=black];

\draw (0,4) node {$0$};
\draw (1,4) node {$1$};
\draw (2,4) node {$2$};
\draw (3,4) node {$3$};
\draw (4,4) node {$4$};
\draw (5,4) node {$5$};
\draw (6,4) node {$6$};
\draw (7,4) node {$7$};
\draw (8,4) node {$8$};
\draw (9,4) node {$9$};


\draw [quivarrow,shorten <=5pt, shorten >=5pt] (1,1) -- node[above]{$1$} (0,0);
\draw [quivarrow,shorten <=5pt, shorten >=5pt] (2,2) -- node[above]{$2$} (1,1);
\draw [quivarrow,shorten <=5pt, shorten >=5pt] (1,1) -- node[above]{$2$} (2,0);
\draw [quivarrow,shorten <=5pt, shorten >=5pt] (2,2) -- node[above]{$3$} (3,1);
\draw [quivarrow,shorten <=5pt, shorten >=5pt] (3,1) -- node[above]{$3$} (2,0);
\draw [quivarrow,shorten <=5pt, shorten >=5pt] (4,2) -- node[above]{$4$} (3,1);
\draw [quivarrow,shorten <=5pt, shorten >=5pt] (3,1) -- node[above]{$4$} (4,0);
\draw [quivarrow,shorten <=5pt, shorten >=5pt] (4,2) -- node[above]{\textcircled{\raisebox{-0.9pt}{5}}} (5,1);
\draw [quivarrow,shorten <=5pt, shorten >=5pt] (5,1) -- node[above]{$5$} (4,0);
\draw [quivarrow,shorten <=5pt, shorten >=5pt] (5,1) -- node[above]{$6$} (6,0);
\draw [quivarrow,shorten <=5pt, shorten >=5pt] (7,1) -- node[above]{$7$} (6,0);
\draw [quivarrow,shorten <=5pt, shorten >=5pt] (8,2) -- node[above]{$8$} (7,1);
\draw [quivarrow,shorten <=5pt, shorten >=5pt] (9,3) -- node[above]{$9$} (8,2);
\draw [quivarrow,shorten <=5pt, shorten >=5pt] (8,2) -- node[above]{$9$} (9,1);
\draw [quivarrow,shorten <=5pt, shorten >=5pt] (7,1) -- node[above]{$8$} (8,0);
\draw [quivarrow,shorten <=5pt, shorten >=5pt] (9,1) -- node[above]{$9$} (8,0);
\draw [quivarrow,shorten <=5pt, shorten >=5pt] (3,3) -- node[above]{$3$} (2,2);
\draw [quivarrow,shorten <=5pt, shorten >=5pt] (3,3) -- node[above]{\textcircled{\raisebox{-0.9pt}{4}}} (4,2);
\draw [quivarrow,shorten <=5pt, shorten >=5pt] (6,2) -- node[above]{$6$} (5,1);
\draw [quivarrow,shorten <=5pt, shorten >=5pt] (6,2) -- node[above]{\textcircled{\raisebox{-0.9pt}{7}}} (7,1);

\draw [densely dotted] (0,-2) -- (0,0);
\draw [densely dotted] (9,-2) -- (9,3);

\draw [dotted, very thick] (3,-2) -- (3,-1);
\draw [dotted, very thick] (6,-2) -- (6,-1);

\end{tikzpicture}
\]
\caption{The module $L_{457}$ in the case $k=3$, $n=9$.}
\label{fig:modex}
\end{figure}
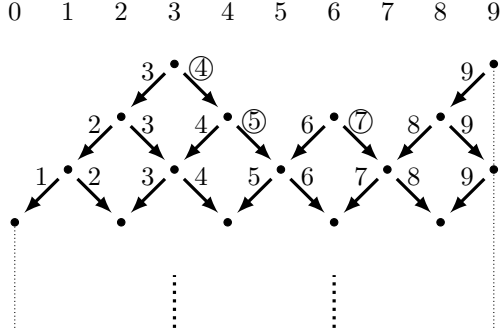
Every rank 1 module in $\on{CM}(B)$ (or $\on{CM}(\hat B)$) is of this form for some (unique) set $I\in \binom{[n]}{k}$ (\cite[Proposition 5.2]{JKS16}).

\begin{defin}
  \label{def:twist}
  From the definition of the modules $L_I$, it is clear that the effect of twisting by $\Phi$ is the same as relabelling the columns of the lattice diagram. 
  In other words, we have that $L_I \cong \,_\Phi L_{I+k}$ in a canonical way.
  We will denote by $\varphi_I:L_I\to \,_{\Phi}L_{I+k}$ this canonical isomorphism (given by identifying lattice diagrams).
\end{defin}

\begin{rmk}\label{rmk:projinj}
  There are some important differences between $\on{CM}(B)$ and $\on{CM}(\hat B)$.
  In both categories, we have for every $i\in [n]$ that $\mathcal P_i \cong \mathcal I_{i+k}\cong L_{i+1, \dots, i+k}$, cf.~\cite[Remark 7.2]{BKM16}.
  These are the only indecomposable projective objects in $\on{CM}(\hat B)$. Observe that we do not know whether this holds in $\on{CM}(B)$, cf.~\cite[paragraph after Remark 3.4]{JKS16}.
\end{rmk}

We need some notation about morphisms between the modules $L_I$ (cf.~\cite[Lemma 7.4]{BKM16}). The spaces $\on{Hom}_B(L_I, L_J)$ 
(respectively $\on{Hom}_{\hat B}(L_I, L_J)$) are free modules of rank 1 over $\C[t]$ (respectively
$\C[[t]]$), generated by 
the morphism $g_{JI}:L_I\to L_J$ corresponding to an embedding of lattice diagrams such that $\on{dim}_\C\on{coker}(g_{JI})$ is minimal. 
The map $g_{JI} t^N$ corresponds to the embedding where the diagram of $L_I$ is shifted downwards $N$ steps.

There is a functor $\mathcal F$ on $\on{CM}(B)$ given by $M\mapsto \,_\Phi M$ on objects and $f\mapsto f$ on morphisms.
We have that $\mathcal F(L_I) \cong L_{I-k}$ via the canonical isomorphism $\varphi_{I-k}$. It is clear from the definition of the morphisms $g_{JI}$ 
that $\mathcal F(g_{JI}) = g_{JI }= \varphi_J^{-1}\circ g_{J-k,I-k}\circ \varphi_I$. Notice that $\mathcal F$ is the identity on morphisms, but the map $g_{JI}$ changes name since we have relabelled 
the basis elements of both domain and codomain. We will sometimes treat the canonical isomorphisms as identifications and write $\mathcal F (g_{JI}) = g_{J-k,I-k}$.
Observe that via the equivalence $\on{CM}(B)\to \on{CM}^{\Z/n\Z}(R)$, the functor $\mathcal F$ is mapped to degree shift by $-k$ (cf.~\cite[Proposition 3.15]{BBE18}).
Again, we define a functor on $\on{CM}(\hat B)$ in the same way, and call it $\mathcal F$ as well.

\begin{defin}
  We denote by $\mathcal B$ the full additive subcategory of $\underline{\on{CM}}(\hat B)$ generated by $\left\{ L_I\ |\ I\in \binom{[n]}{k} \right\}$.
\end{defin}

Even though $\underline{\on{CM}}(\hat B)$ is triangulated, we remark that 
$\mathcal B$ is not a triangulated subcategory of $\underline{\on{CM}}(\hat B)$. However, we can explicitly describe the Serre functor $[2]$ of $\underline{\on{CM}}(\hat B)$ (see also \cite[Proposition 3.15]{BBE18}).
It turns out that inside $\underline{\on{CM}}^{\Z/n\Z}(R)$ there is an isomorphism of functors
$[2]\cong (-k)$, which in our setting translates into the following:

\begin{thm}[{\cite[Theorem 3.22]{DL16}}]\label{thm:shift}
  There is an isomorphism of functors $\mathcal F\cong [2]$ on $\underline{\on{CM}}(\hat B)$.
\end{thm}

\begin{rmk}
	It follows in particular that $\mathcal B$ is closed under the Serre functor $[2]$.
\end{rmk}

\begin{rmk}
	The statement that $L_I\cong L_{I+k}[2]$ appears also in \cite[Proposition 2.7]{BB17}, and is implicitly used in \cite[\S7]{JKS16} in some specific cases.
	For our purposes, we will only need that $[2]\cong \mathcal F$ as functors on $\mathcal B$. We provide a direct proof of the latter fact.
\end{rmk}

\begin{proof}[Proof that {$\mathcal F\cong [2]$} on $\mathcal B$]
	Let us denote by $[k]$ the interval $\left\{ 1, 2, \dots, k \right\}\subseteq [n]$, and for $x\in [n]$ let us denote by $x+[k]$ the cyclic interval $\left\{ x+1, \dots, x+k \right\}\subseteq [n]$.
	
	Let us first prove that for every $I\in \binom {[n]}{k}$, there is an exact sequence in $\on{CM}(\hat B)$
	$$
	\xymatrix{
		0 \ar[r] & L_{I+k} \ar^-f[r] & \displaystyle{\bigoplus_{v\in V}} L_{v+[k]} \ar^\partial[r] & \displaystyle{\bigoplus_{u\in U}}L_{u+[k]} \ar[r]^-h & L_I\ar[r] & 0
	}
	$$
	where $U = \left\{ u\not \in I\ |\ u+1 \in I \right\}$ and $V= \left\{ v\in I\ |\ v+1 \not \in I \right\}$.
	The map $f$ is given by 
	\begin{align*}
	f = \begin{pmatrix}
	g_{v+[k], I+k}
	\end{pmatrix}_{v\in V},
	\end{align*}
	the map $h$ is given by
	\begin{align*}
	h = \begin{pmatrix}
	g_{I, u+[k]}
	\end{pmatrix}_{u\in U},
	\end{align*}
	and the map $\partial$ is given by
	\begin{align*}
	\partial = \begin{pmatrix}
	\partial_{uv}
	\end{pmatrix}_{u\in U, v\in V}
	\end{align*}
	with 
	\begin{align*}
	\partial_{uv} = 
	\begin{cases}
	g_{u+[k], v+[k]} &\text{ if $u$ is the predecessor of $v$ in the cyclic order on $U\cup V$};\\
	-g_{u+[k], v+[k]} &\text{ if $u$ is the successor of $v$ in the cyclic order on $U\cup V$};\\
	0 &\text{ otherwise}.
	\end{cases}
	\end{align*}
	In particular we mean that $\partial = 0$ if $L_I$ is projective.
	To prove the assertion that the above sequence is exact, observe that 
	$$
	\xymatrix{
		\displaystyle{\bigoplus_{v\in V}} L_{v+[k]} \ar^\partial[r] & \displaystyle{\bigoplus_{u\in U}}L_{u+[k]} \ar[r]^-h & L_I\ar[r] & 0
	}
	$$
	is a projective presentation of $L_I$ (cf.~\cite[Proposition 5.6]{JKS16}), and similarly
	$$
	\xymatrix{
		0 \ar[r] & L_{I+k} \ar^-f[r] & \displaystyle{\bigoplus_{v\in V}} L_{v+[k]} \ar^\partial[r] & \displaystyle{\bigoplus_{u\in U}}L_{u+[k]}
	}
	$$
	is an injective presentation of $L_{I+k}$.
	
	Now let us consider two $k$-element subsets $I, I'$, and the corresponding exact sequences. 
	We will construct a commutative diagram
	\[
	\xymatrix{
		0 \ar[r] & L_{I+k} \ar^-f[r] \ar[d]^{g_{I'+k,I'+k}}& \displaystyle{\bigoplus_{v\in V}} L_{v+[k]} \ar^\partial[r]\ar[d]^{N} & 
		\displaystyle{\bigoplus_{u\in U}}L_{u+[k]} \ar[r]^-h \ar[d]^{M}& L_I\ar[r]\ar[d]^{g_{I'I}} & 0\\
		0 \ar[r] & L_{I'+k} \ar^-{f'}[r] & \displaystyle{\bigoplus_{v'\in V'}} L_{v'+[k]} \ar^{\partial'}[r] & \displaystyle{\bigoplus_{u'\in U'}}L_{u'+[k]} \ar[r]^-{h'} & L_{I'}\ar[r] & 0.
	}
	\]
	
	By the definition of the triangulated structure on $\underline{\on{CM}}(\hat B)$, this means that  
	$g_{I'+k, I+k}[2] = g_{I'I} \cong \mathcal F (g_{I'+k, I+k})$. Since all morphism spaces in $\mathcal B$ are generated by maps of this form (in particular, recall that $g_{II} = \on{id}_{L_I}$), this will be enough to prove the assertion that $[2]\cong \mathcal F$ on $\mathcal B$.
	
	Let us fix some more notation to simplify the construction. We will drop the $+[k]$ in indices to avoid clogging the formulas. For instance, we will write $L_u$ for $L_{u+[k]}$ and similarly $g_{uI}$ for $g_{u+[k], I}$.
	For the cyclic orders on $U\cup V$ and on $U'\cup V'$, we write $\s$ and $\p$ for the successor and predecessor functions.
	
	Let us fix $u\in U$. We write $(L_{I'})_u$ for the $Z$-module of rank 1 corresponding to vertex $u$ inside $L_{I'}$. The generator of $(L_{I'})_u$ as a $Z$-module is in the image via $h'$ of either one or two of the $L_{u'}$. If there is only one such $L_{u'}$, then define $\uu(u)=u'$. If there are two such $L_{u'_1}$ and $L_{u'_2}$ (this happens if and only if $|U'|\geq 2$ and $u\in V'$), with $u'_1< u < u'_2$, then define $\uu(u)= u'_1$. 
	In other words, $\uu(u)$ is the unique element of $U'$ such that 
	$\p(\uu(u))< u \leq \s(\uu(u))$.
	By construction we have $g_{I'\uu(u)}\circ g_{\uu(u)u} = g_{I'u}$.
	We define $d(u)$ by the equation $g_{I'u}t^{d(u)} = g_{I'I}\circ g_{Iu}$.
	
	Dually, let us fix $v'\in V'$. The generator of $(L_{I+k})_{v'+k}$ as a $Z$-module is mapped via $f$ to a $Z$-module generator in either one or two of the $L_v$. If there is only one such $L_v$, then define $\vv(v') = v$. If there are two such $L_{v_1}$ and $L_{v_2}$ (this happens if and only if $|V|\geq 2$ and $v'\in U$), with $v_1<v'<v_2$, then define $\vv(v') = v_2$.
	In other words, $\vv(v')$ is the unique element of $V$ such that 
	$\p(\vv(v'))\leq v'< \s(\vv(v'))$. 
	By construction we have $g_{v'\vv(v')}\circ g_{\vv(v'), I+k} = g_{v', I+k}$.
	We define $d(v')$ by the equation $g_{v', I+k}t^{d(v')} = g_{v', I'+k}\circ g_{I'+k, I+k}$.
	
	Now we can define maps $M: \bigoplus_{u\in U}L_u \to \bigoplus_{u'\in U'} L_{u'}$ and $N:\bigoplus_{v\in V}L_v \to \bigoplus_{v'\in V'} L_{v'}$ by setting
	$$
	M_{u'u} = 
	\begin{cases}
	g_{u'u}t^{d(u)}, & \text{ if } u' = \uu(u);\\
	0, & \text{ otherwise}
	\end{cases}
	$$ 
	and 
	$$
	N_{v'v} = 
	\begin{cases}
	g_{v'v}t^{d(v')}, & \text{ if } v = \vv(v');\\
	0, & \text{ otherwise.}
	\end{cases}
	$$ 
	
	Let us check that the right square commutes, the left square being similar. 
	We have
	\begin{align*}
	(h'\circ M)_u &= 
	\sum_{u'\in U'} g_{I'u'}\circ M_{u'u} = g_{I'\uu(u)}\circ g_{\uu(u)u}t^{d(u)} = \\
	&= g_{I'u}t^{d(u)} = g_{I'I}\circ g_{Iu} = \\
	&= (g_{I'I}\circ h)_u. 
	\end{align*}
	Let us now consider the middle square. We have
	$$
	(M\circ \partial)_{u'v} = M_{u' \p(v)}\circ g_{\p(v)v} - 
	M_{u' \s(v)}\circ g_{\s(v)v}
	$$
	and 
	$$
	(\partial' \circ N)_{u'v} = g_{u'\s(u')}\circ N_{\s(u')v}- g_{u'\p(u')}\circ N_{\p(u')v}.$$
	
	There are four cases to consider. 
	
	\textbf{Case 1.} Let us first assume that $\uu(\p(v))= \uu(\s(v)) = u'$. In this case we have 
	\begin{align*}
	(M\circ \partial)_{u'v} &= 
	g_{u'\p(v)}\circ g_{\p(v)v}t^{d(\p(v))} - g_{u'\s(v)}\circ g_{\s(v)v}t^{d(\s(v))}.
	\end{align*}
	In particular, $(M\circ \partial)_{u'v} = (t^a-t^b)g_{u'v}$ for some $a, b\geq 0$. 
	From $h'\circ M\circ \partial = 0 $ we get then that $a = b$ and thus that $(M\circ \partial)_{u'v} = 0$. 
	Now since $\p(u')< \p(v)\leq \s(u')$ and $\p(u')< \s(v)\leq \s(u')$, we must have that either $\vv(\p(u')) = \vv(\s(u')) =v$ or $\vv(\p(u')) \neq v \neq \vv(\s(u'))$.
	In the first case, 
	$$(\partial' \circ N)_{u'v} = g_{u'\s(u')}\circ g_{\s(u')v}t^{d(\s(u'))}- g_{u'\p(u')}\circ g_{\p(u')v}t^{d(\p(u'))}$$
	and as above we can argue that this is 0. 
	In the second case, $(\partial' \circ N)_{u'v} = 0-0 $ directly.
	
	\textbf{Case 2.} In a similar way we can argue that $(M\circ \partial)_{u'v} =(\partial' \circ N)_{u'v} = 0$ whenever $\uu(\p(v))\neq u'\neq \uu(\s(v))$.
	
	\textbf{Case 3.} Let us now assume that $\uu(\p(v)) = u' \neq \uu(\s(v))$.
	This means that $\p(u')< \p(v)\leq \s(u')<\s(v)$, so we obtain that $\vv(\s(u')) = v\neq \vv(\p(u'))$. Thus 
	$$(M\circ \partial)_{u'v} = g_{u'\p(v)}\circ g_{\p(v)v}t^{d(\p(v))} $$
	and 
	$$(\partial' \circ N)_{u'v} = g_{u'\s(u')}\circ g_{\s(u')v}t^{d(\s(u'))}.$$
	
	We need to make some observations (cf.~\cite[Proposition~2.7]{BB17} for a pictorial interpretation). First, the maps $ g_{Iv}\circ g_{v,I+k}$ are all equal, and we can call them $\iota_I$. The maps $g_{Iu}\circ g_{u, I+k}$ are also all equal to $\iota_I$. Defining $\iota_{I'}$ in the same way, we remark that $g_{I'I}\circ \iota_I = \iota_{I'}\circ g_{I'+k, I+k}$. 
	
	With these observations, we can write
	\begin{align*}
	g_{I'u'}\circ (M\circ \partial)_{u'v} \circ	g_{v, I+k} &=
	g_{I'u'}\circ g_{u'\p(v)}\circ g_{\p(v)v}t^{d(\p(v))}\circ g_{v, I+k}=\\
	&= g_{I'I}\circ g_{Iv}\circ g_{v, I+k}= \\
	&= g_{I'I}\circ \iota_I=\\
	&= \iota_{I'}\circ g_{I'+k, I+k}=\\
	&= g_{I'u'}\circ g_{u', I'+k}\circ g_{I'+k, I+k}=\\
	&= g_{I'u'}\circ g_{u'\s(u')}\circ g_{\s(u')v}t^{d(\s(u'))}\circ g_{v, I+k}=\\
	&= g_{I'u'}\circ (\partial'\circ N)_{u'v}\circ 	g_{v, I+k}.
	\end{align*}
	Since both $(M\circ \partial)_{u'v}$ and $(\partial'\circ N)_{u'v}$ are equal to a power of $t$ times $g_{u'v}$, we conclude that they must be equal.
	
	\textbf{Case 4.} The case $\uu(\p(v)) \neq u' = \uu(\s(v))$ is similar to Case 3.
	We conclude that the middle square and thus the whole diagram commutes, and so we are done.
\end{proof}

\section{Cluster tilting in $\underline{\on{CM}}(\hat B)$}\label{sec:tilting}

There is a strong relationship between combinatorics of Postnikov diagrams and homological algebra in $\underline{\on{CM}}(\hat B)$.
We are interested in cluster tilting objects in the Frobenius category $\on{CM}(\hat B)$ and in the 2-Calabi-Yau category $\underline{\on{CM}}(\hat B)$, and these are the same objects.
To be more precise, there is a bijection between isomorphism classes of basic cluster tilting objects in $\on{CM}(\hat B)$ and in $\underline{\on{CM}}(\hat B)$, given by adding or removing the projective indecomposables.

The noncrossing property introduced in Section \ref{sec:post} corresponds to $\on{Ext}$-vanishing in $\on{CM}(\hat B)$.
\begin{prop}{\cite[Proposition 5.6]{JKS16}}.\label{prop:ext}
  Let $I, J\in\binom{[n]}{k}$. Then $\on{Ext}_B^1(L_I,L_J) = 0$ if and only if $I$ and $J$ are noncrossing.
\end{prop}

Let $D$ be a reduced $(k,n)$-Postnikov diagram, and let $\mathbb I= \mathbb I(D)$. 
Define the $\hat B$-module $T$ by $$T= \bigoplus_{I\in \mathbb I} L_I.$$
We also denote by $T$ the $B$-module defined in the same way. This abuse of notation is justified by the fact that the stable endomorphism algebra of $T$ 
is the same regardless of the completion, as we will see.

The following theorem was proved in \cite{JKS16}.
\begin{thm}
  For any reduced $(k,n)$-Postnikov diagram, the module $T$ defined above is a cluster tilting object in $\on{CM}(\hat B)$.
\end{thm}

\begin{proof}
  Since $\mathbb I$ is a maximal noncrossing collection, it follows by Proposition $\ref{prop:ext}$ that $T$ is a maximal rigid object in $\on{CM}(\hat B)$.
  Since $\on{CM}(\hat B) $ is known to have at least one cluster tilting object, this is equivalent to $T$ being cluster tilting (cf.~\cite[Remark 4.8]{JKS16}).
\end{proof}

We remark that $T$ is also a maximal rigid object in $\on{CM}(B)$, but we do not know whether it is actually cluster tilting.

\begin{rmk}
  Any maximal noncrossing collection contains the $n$ cyclic intervals of length $k$. 
  Remark \ref{rmk:labels} says that the labels of the $n$ boundary regions of a Postnikov diagram are precisely these $n$ cyclic intervals.
  Indeed, any cluster tilting object in $\on{CM}(\hat B)$ has as summands the $n$ indecomposable projective-injective objects, which are labelled 
  by such intervals. 
\end{rmk}

\begin{thm}\cite[Theorem 10.3 and Theorem 11.2]{BKM16}\label{thm:bkm}.
  Let $D$ be a reduced $(k,n)$-Postnikov diagram, let $T$ be as above and let $A(D) = \wp(Q, W, F)$ be as in Section \ref{sec:post}.
  Then there exists a unique isomorphism $A(D)\to \on{End}_B(T)$ such that the vertex $I$ of $Q$ is mapped to $\on{id}_{L_I}$ and any arrow $I\to J$ in $Q$ is mapped to 
  the morphism $g_{JI}:L_I\to L_J$.
  Moreover, this induces an isomorphism $\hat A(D) = \hat \wp(Q, W, F)\to \on{End}_{\hat B}(T)$.
\end{thm}

We call $g:A\to \on{End}_B(T)$ this isomorphism.
In particular, the frozen vertices of $Q$ correspond to the indecomposable projective $B$-modules, and $B^{opp} $ is identified with $eAe\subseteq A$ (\cite[Corollary 10.4]{BKM16}).
In this article, we focus on the study of the algebra $\Lambda = A/AeA$. This corresponds to quotienting out endomorphisms of $T$ factoring through its projective summands, thus moving to the stable category.

\begin{lemma}\label{lem:descent}
  The isomorphism $g: A\to \on{End}_B(T)$ induces an isomorphism $\underline g: A/AeA \to \underline{\on{End}}_B(T)$.
  In the same way, the isomorphism $\hat A\cong \on{End}_{\hat B}(T)$ induces an isomorphism $\hat A/\hat Ae\hat A \cong \underline{\on{End}}_{\hat B}(T)$.
\end{lemma}

\begin{proof}
  We give the proof for the non-complete case; the other case is similar.
  We have $T\cong T'\oplus P$, where $P$ is the sum of the indecomposable projective $B$-modules $\mathcal P_i$. Call 
  $E$ the subset of $\on{End}_B(T)$ consisting of maps that factor through $P$.
  There is a commutative diagram
  $$
  \xymatrix{
    0\ar[r] &AeA \ar[r]\ar[d]&A\ar[r]\ar[d]^g& A/AeA\ar[d]\ar[r] & 0\\
    0\ar[r] & E\ar[r] & \on{End}_B(T)\ar[r] & \underline{\on{End}}_B(T)\ar[r] &0
  }
  $$
  where the two leftmost vertical maps are isomorphisms, thus the claim is proved.
\end{proof}

The following results will justify our claims that the completion
does not play a big role in our setting.

\begin{prop}\label{prop:complete}
  We have $A/ Ae A \cong \hat\wp(\underline Q, \underline W)\cong \hat A/\hat A e\hat A $.
\end{prop}

\begin{proof}
  Let us first prove the first isomorphism.
  By Corollary \ref{cor:admissible}, it is enough to prove that every sufficiently long path in $Q$ is equivalent in $A$ to a path through a frozen vertex.
  By \cite[Corollary 9.4]{BKM16}, we have that a basis of $e_JAe_I$ is given by the set
  \begin{align*}
    \left\{ u^N p_{IJ}\ |\ N\in \N \right\}
  \end{align*}
  where $u$ is as in Definition \ref{defin:u} and $p_{IJ}$ is a chosen path from $I$ to $J$ that is not equivalent to a path containing a cycle. There is a unique equivalence class of paths from $I$ to $J$ containing such an element. 
  This basis is mapped via $g$ to the basis 
  \begin{align*}
    \left\{ t^N g_{JI}\ | \ N\in \N \right\} 
  \end{align*}
  of $\on{Hom}_B(L_I, L_J)$.

  Now observe that paths of a fixed degree in $u$ have bounded length, so for any $d$ we can find a path with degree larger than $d$.
  Translated into maps from $L_I$ to $L_J$, it is then enough to prove that every map of the form $g_{JI}t^N$ for $N\gg 0$ factors through a map $g_{PI}$ for a projective $L_P$.
  Given the description of maps of the form $g_{JI}t^N$ as embeddings of lattice diagrams, this is clear.

  To conclude, observe that the above argument implies that the two-sided ideal $\hat A e\hat A$ is closed in $\hat A$, so the second isomorphism follows.
\end{proof}

From Theorem~\ref{thm:bkm}, Lemma~\ref{lem:descent} and Proposition~\ref{prop:complete}, we get:

\begin{cor}\label{cor:endiso}
  There is an algebra isomorphism 
  \begin{align*}
    \underline{\on{End}}_B(T)\cong\underline{\on{End}}_{\hat B}(T).
  \end{align*}
\end{cor}

In view of Corollary~\ref{cor:endiso}, we can essentially ignore the completions. In particular, all statements about $T$ that depend on the triangulated or Calabi-Yau structure
of $\underline{\on{CM}}(B)$ (such as mutation and suspensions) can be carried out in $\underline{\on{CM}}(\hat B)$ instead, without affecting $\underline{\on{End}}_B(T)$ and 
therefore $\Lambda$.

The following results about $\on{CM}(\hat B)$ coming from the combinatorics of Postnikov diagrams still hold true for $\on{CM}(B)$, if we replace ``cluster tilting'' by ``maximal rigid''.

\begin{lemma}
  For any fixed $(k,n)$ there is a bijection
  $$
  \xymatrix{
	*{\left\{\begin{matrix}
	  \text{Basic cluster tilting objects in } \on{CM}(\hat B)\\ \text{that lie in } \mathcal B  \end{matrix}\right\}}
      \ar@{<->}[rr] &&
	*{\left\{\begin{matrix}
	  \text{Maximal noncrossing}\\ \text{collections of}\\ \text{elements of } \binom{[n]}{k} 
      \end{matrix}\right\}.}
  }
  $$
\end{lemma}

\begin{proof}
  Since the modules $L_I$ are indecomposable, they are precisely the indecomposable objects in $\mathcal B$.
  It follows that maximal rigid objects in $\on{CM}(\hat B)$ that lie in $\mathcal B$ correspond precisely to maximal noncrossing collections.
\end{proof}

Combining this with Theorem \ref{thm:classnoncross} we get

\begin{prop}\label{prop:correspond}
  Basic cluster tilting objects in $\on{CM}(\hat B)$ (respectively in $\underline{\on{CM}}(\hat B)$) 
  that lie in $\mathcal B$ are precisely those contructed as above from reduced $(k,n)$-Postnikov diagrams.
\end{prop}

There are various notions of mutation for the various objects we are considering, and in a sense they all correspond to each other. The rest of this section is devoted to 
making this statement a bit more precise.

There is a well-defined mutation of cluster tilting objects in $\on{CM}(\hat B)$ \cite[Remark~4.8]{JKS16}. Namely, if $X\oplus T$ is a cluster tilting object
and $X$ is indecomposable nonprojective then there is a unique indecomposable nonprojective $Y\not\cong X$ such that $Y\oplus T$ is cluster tilting. 
Now if $T$ is cluster tilting in $\on{CM}(\hat B)$ and moreover $T\in \mathcal B$, then $T= \bigoplus_{I\in \mathbb I}L_I$ for some 
maximal noncrossing collection $\mathbb I$. Suppose that $I\in \mathbb I$ is not a cyclic interval of $[n]$ (i.e.~not the label of a projective $\hat B$-module). Then, under 
some condition,
there is a unique $I'$ by which we can replace $I$ so that $\mathbb I\setminus \left\{ I \right\}\cup \left\{ I' \right\}$ is a maximal noncrossing collection.
The precise description of $I'$ is rather cumbersome, and can be found for instance in \cite[Theorem~1.4]{OPS15}. 
If we start from the cluster tilting object $T = \bigoplus_{J\in \mathbb I} L_J\in \mathcal B$, and we mutate it at $L_I$, the new cluster tilting object will be $\bigoplus_{J\in \mathbb I'}L_J$ by Proposition~\ref{prop:correspond}.

There is a combinatorial interpretation of mutation of cluster tilting objects (or more directly of maximal noncrossing collections) in terms of Postnikov diagrams.
This is given by the notion of \emph{geometric exchange} on a Postnikov diagram, i.e.~applying the local operation depicted in Figure \ref{fig:geomex}, followed by untwisting and boundary untwisting moves 
to make the new Postnikov diagram reduced.
\begin{figure}[h]
\[
\begin{tikzpicture}[baseline=(bb.base),
   quivarrow/.style={black, -latex, very thick}, doublearrow/.style={black, to-to,thick, line join=round,
decorate, decoration={
    zigzag,
    segment length=4,
    amplitude=.9,post=lineto, pre= lineto, pre length = 2pt,
  post length=2pt}}, scale = 0.5]
\newcommand{\goodarrow}{\arrow{angle 60}}
\newcommand{\dotrad}{0.07cm} 
\path (0,0) node (bb) {}; 

\draw (0,0) node {$I$};


\draw [quivarrow,shorten <=8pt,shorten >=5pt]
(0,0)--(0,4);

\draw [quivarrow,shorten <=8pt,shorten >=5pt]
(0,0)--(0,-4);

\draw [quivarrow,shorten <=5pt,shorten >=8pt]
(4,0)--(0,0);

\draw [quivarrow,shorten <=5pt,shorten >=8pt]
(-4,0)--(0,0);


\draw  plot coordinates {(4,-2) (-2,4)}
[postaction=decorate,decoration={markings,
mark= at position 0.2 with \goodarrow,
mark= at position 0.5 with \goodarrow,
mark= at position 0.8 with \goodarrow}];

\draw  plot coordinates {(-4,2) (2,-4)}
[postaction=decorate,decoration={markings,
mark= at position 0.2 with \goodarrow,
mark= at position 0.5 with \goodarrow,
mark= at position 0.8 with \goodarrow}];

\draw  plot coordinates {(-4,-2) (2,4)}
[postaction=decorate,decoration={markings,
mark= at position 0.2 with \goodarrow,
mark= at position 0.5 with \goodarrow,
mark= at position 0.8 with \goodarrow}];

\draw  plot coordinates {(4,2) (-2,-4)}
[postaction=decorate,decoration={markings,
mark= at position 0.2 with \goodarrow,
mark= at position 0.5 with \goodarrow,
mark= at position 0.82 with \goodarrow}];

\begin{scope}[shift={(12,0)}]

\draw (0,0) node {$I'$};


\draw [quivarrow,shorten <=5pt,shorten >=8pt]
(0,4)--(0,0);

\draw [quivarrow,shorten <=5pt,shorten >=8pt]
(0,-4)--(0,0);

\draw [quivarrow,shorten <=8pt,shorten >=5pt]
(0,0)--(4,0);

\draw [quivarrow,shorten <=8pt,shorten >=5pt]
(0,0)--(-4,0);

\draw [quivarrow,shorten <=5pt,shorten >=5pt]
(4,0)--(0,-4);

\draw [quivarrow,shorten <=5pt,shorten >=5pt]
(-4,0)--(0,-4);

\draw [quivarrow,shorten <=5pt,shorten >=5pt]
(4,0)--(0,4);

\draw [quivarrow,shorten <=5pt,shorten >=5pt]
(-4,0)--(0,4);


\draw  plot[smooth] 
coordinates {(4,-2) (2,-2) (0,-2) (-2,0) (-2,2) (-2,4)}
[postaction=decorate,decoration={markings,
mark= at position 0.1 with \goodarrow,
mark= at position 0.3 with \goodarrow,
mark= at position 0.5 with \goodarrow,
mark= at position 0.7 with \goodarrow,
mark= at position 0.9 with \goodarrow
}];

\draw  plot[smooth] 
coordinates {(-4,2) (-2,2) (0,2) (2,0) (2,-2) (2,-4)}
[postaction=decorate,decoration={markings,
mark= at position 0.1 with \goodarrow,
mark= at position 0.3 with \goodarrow,
mark= at position 0.5 with \goodarrow,
mark= at position 0.7 with \goodarrow,
mark= at position 0.9 with \goodarrow}];

\draw  plot[smooth] 
coordinates {(-4,-2) (-2,-2) (0,-2) (2,0) (2,2) (2,4)}
[postaction=decorate,decoration={markings,
mark= at position 0.1 with \goodarrow,
mark= at position 0.3 with \goodarrow,
mark= at position 0.5 with \goodarrow,
mark= at position 0.7 with \goodarrow,
mark= at position 0.9 with \goodarrow}];

\draw  plot[smooth] 
coordinates {(4,2) (2,2) (0,2) (-2,0) (-2,-2) (-2,-4)}
[postaction=decorate,decoration={markings,
mark= at position 0.1 with \goodarrow,
mark= at position 0.3 with \goodarrow,
mark= at position 0.5 with \goodarrow,
mark= at position 0.7 with \goodarrow,
mark= at position 0.9 with \goodarrow}];

\end{scope}


\draw [doublearrow] (5.5,0) -- (6.5,0);

\end{tikzpicture}
\]
\caption{Geometric exchange and the corresponding operation on the quiver.}
\label{fig:geomex}
\end{figure}
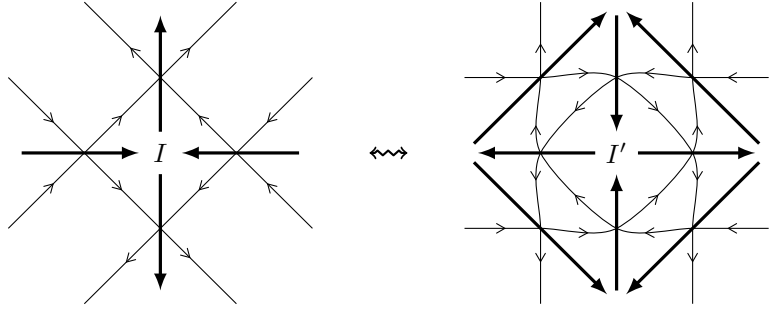

Notice that the labels of vertices do not change except at the chosen vertex. The label $I'$ is precisely the only $k$-element set which is not $I$
which makes the collection of labels noncrossing.
The effect on the corresponding quiver is almost Fomin-Zelevinsky mutation. The step of removing any new 2-cycles must be replaced as follows: remove any new 2-cycle consisting of non-boundary arrows, then 
for every 2-cycle consisting of a boundary arrow and a non-boundary arrow, remove the boundary arrow and treat the non-boundary arrow as a new boundary arrow. This corresponds to the effect of applying a 
boundary untwisting move, as opposed to a ``normal'' untwisting move (cf.~\cite[Lemma 12.1]{BKM16}).
This mutation rule also coincides with mutation of ice quivers with potential presented in \cite{Pre18}. 
If we restrict our attention to the quiver $\underline Q$, this difference disappears (since arrows between frozen vertices are not arrows in $\underline Q$).

By the above discussion, the notions of mutation of Postnikov diagrams (i.e.~geometric exchange), of cluster tilting objects in $\underline{\on{CM}}(\hat B)$, 
and of quivers with potential all correspond to each other when they make sense. We remark that sometimes mutation of a cluter tilting object 
in $\mathcal B$ will produce a cluster tilting object which does not lie in $\mathcal B$, and that will happen precisely when geometric exchange is not possible (because 
the chosen vertex does not have valency 4).
The correspondence between mutation of cluster tilting objects and quivers with potential is a widespread phenomenon, see for instance \cite{BIRS11}.

In particular, we can read off mutation of cluster tilting objects in $\on{CM}(\hat B)$ (respectively in $\underline{\on{CM}}(\hat B)$) from the Postnikov diagram $D$, from the quiver $Q$, or from the collection $\mathbb I$.
In Figures \ref{fig:mutpost39} and \ref{fig:mutpostquiv39}, we illustrate the geometric exchange at $134$ of the Postnikov diagram of Figure \ref{fig:post39} and the corresponding mutation of the quiver with potential.
Vertex $134$ is mutated to $245$. 
We can deduce that mutation is transitive on cluster tilting objects that lie in $\mathcal B$:
\begin{figure}
\[
\begin{tikzpicture}[baseline=(bb.base), quivarrow/.style={black, -latex, very thick}
]

\newcommand{\goodarrow}{\arrow{angle 60}}
\newcommand{\bstart}{130} 
\newcommand{\ninth}{40} 
\newcommand{\qstart}{150} 
\newcommand{\radius}{4.8cm} 
\newcommand{\eps}{11pt} 
\newcommand{\dotrad}{0.07cm} 

\path (0,0) node (bb) {};


\draw (0,0) circle(\radius) [thick,dashed];

\foreach \n in {1,...,9}
{ \coordinate (b\n) at (\bstart-\ninth*\n:\radius);
  \draw (\bstart-\ninth*\n:\radius+\eps) node {$\n$}; }
  

\foreach \n in {1,2,3,4,5,6,7,8,9} {\draw (b\n) circle(\dotrad) [fill=black];}


\draw  plot[smooth]
coordinates {(b9) (b3)}
[ postaction=decorate, decoration={markings,
  mark= at position 0.2 with \goodarrow, mark= at position 0.3 with \goodarrow,
mark= at position 0.4 with \goodarrow,mark= at position 0.5 with \goodarrow, mark= at position 0.8 with \goodarrow}];
\draw  plot[smooth]
coordinates {(b3) (b6)}
[ postaction=decorate, decoration={markings,
  mark= at position 0.2 with \goodarrow, mark= at position 0.4 with \goodarrow,
mark= at position 0.6 with \goodarrow,mark= at position 0.73 with \goodarrow, mark= at position 0.9 with \goodarrow}];
\draw  plot[smooth]
coordinates {(b6) (b9)}
[ postaction=decorate, decoration={markings,
  mark= at position 0.2 with \goodarrow, mark= at position 0.3 with \goodarrow,
mark= at position 0.4 with \goodarrow,mark= at position 0.5 with \goodarrow, mark= at position 0.8 with \goodarrow}];
\draw  plot[smooth] coordinates {(b7)(230:\radius*25/40) (-30:.25*\radius) (70:\radius*25/40) (b1)}
[ postaction=decorate, decoration={markings,
  mark= at position 0.1 with \goodarrow, mark= at position 0.29 with \goodarrow,
  mark= at position 0.42 with \goodarrow, mark= at position 0.5 with \goodarrow,
  mark= at position 0.6 with \goodarrow, mark= at position 0.69 with \goodarrow,
  mark= at position 0.79 with \goodarrow, mark= at position 0.85 with \goodarrow,
  mark= at position 0.92 with \goodarrow }];
\draw  plot[smooth] coordinates {(b1) (110:\radius*25/40)(-150:\radius*.25) (-50:\radius*25/40)(b4)}
[ postaction=decorate, decoration={markings,
  mark= at position 0.1 with \goodarrow, mark= at position 0.29 with \goodarrow,
  mark= at position 0.42 with \goodarrow, mark= at position 0.5 with \goodarrow,
  mark= at position 0.6 with \goodarrow, mark= at position 0.65 with \goodarrow,
  mark= at position 0.69 with \goodarrow, mark= at position 0.73 with \goodarrow,
  mark= at position 0.85 with \goodarrow }];
\draw  plot[smooth] coordinates {(b4)(-10:\radius*25/40) (90:\radius*.25) (190:\radius*25/40)(b7)}
[ postaction=decorate, decoration={markings,
  mark= at position 0.1 with \goodarrow, mark= at position 0.29 with \goodarrow,
  mark= at position 0.42 with \goodarrow, mark= at position 0.5 with \goodarrow,
  mark= at position 0.6 with \goodarrow, mark= at position 0.69 with \goodarrow,
  mark= at position 0.79 with \goodarrow, mark= at position 0.85 with \goodarrow,
  mark= at position 0.92 with \goodarrow }];
\draw  plot[smooth] coordinates {(b2) (80:\radius*23/40)(30:\radius*7/40) (-25:\radius*.38) (-100:\radius*.6)(b5)}
[ postaction=decorate, decoration={markings,
  mark= at position 0.13 with \goodarrow, mark= at position 0.22 with \goodarrow,
  mark= at position 0.3 with \goodarrow, mark= at position 0.44 with \goodarrow,
  mark= at position 0.53 with \goodarrow, mark= at position 0.63 with \goodarrow,
 mark= at position 0.75 with \goodarrow, mark = at position 0.9 with \goodarrow}];
 \draw  plot[smooth] coordinates {(b5) (-80:\radius*.5)(-40:\radius*.4) (-90:\radius*7/40) (-145:\radius*26/40)(b8)}
 [ postaction=decorate, decoration={markings,
  mark= at position 0.13 with \goodarrow, mark= at position 0.23 with \goodarrow,
  mark= at position 0.32 with \goodarrow, mark= at position 0.4 with \goodarrow,
  mark= at position 0.48 with \goodarrow,mark = at position 0.6 with \goodarrow, mark= at position 0.72 with \goodarrow,
 mark= at position 0.86 with \goodarrow}];
 \draw  plot[smooth] coordinates {(b8) (200:\radius*23/40)(150:\radius*7/40) (95:\radius*26/40)(b2)}
 [ postaction=decorate, decoration={markings,
  mark= at position 0.14 with \goodarrow, mark= at position 0.25 with \goodarrow,
  mark= at position 0.35 with \goodarrow, mark= at position 0.47 with \goodarrow,
  mark= at position 0.6 with \goodarrow, mark= at position 0.72 with \goodarrow,
 mark= at position 0.86 with \goodarrow}];


\foreach \n/\m/\r in {1/789/0.88, 2/891/0.88, 3/912/0.88, 4/123/0.85, 5/234/0.89, 6/345/0.85, 7/456/0.85, 8/567/.89, 9/678/.87}
{ \draw (\qstart-\ninth*\n:\r*\radius) node (q\m) {$\m$}; }

\foreach \m/\a/\r in {179/88/0.6 , 467/208/0.6, 178/117/0.4, 124/-20/.45, 457/237/.42 , 147/0/0, 127/65/.38, 145/-60/.31, 478/185/.38, 245/-89/.55}
{ \draw (\a:\r*\radius) node (q\m) {$\m$}; }


\end{tikzpicture}
\]
\caption{The geometric exchange at $134$ of the Postnikov diagram of Figure \ref{fig:post39}.}
\label{fig:mutpost39}
\end{figure}
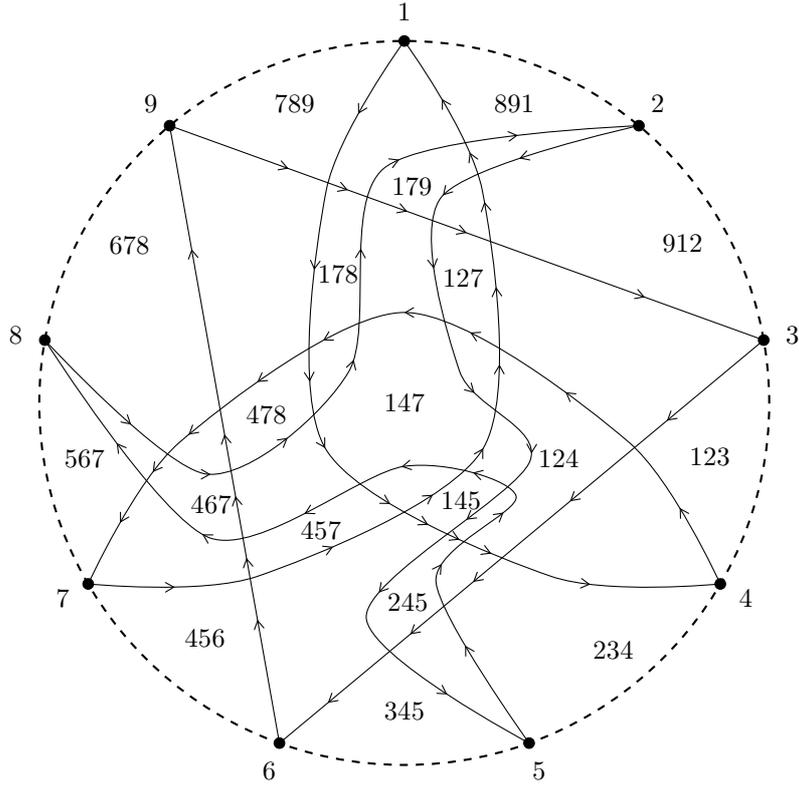

\begin{figure}
\[
\begin{tikzpicture}[baseline=(bb.base), quivarrow/.style={black, -latex, very thick}
]

\newcommand{\goodarrow}{\arrow{angle 60}}
\newcommand{\bstart}{130} 
\newcommand{\ninth}{40} 
\newcommand{\qstart}{150} 
\newcommand{\radius}{4.8cm} 
\newcommand{\eps}{11pt} 
\newcommand{\dotrad}{0.07cm} 

\path (0,0) node (bb) {};


\draw (0,0) circle(\radius) [thick,dashed];

\foreach \n in {1,...,9}
{ \coordinate (b\n) at (\bstart-\ninth*\n:\radius);
  \draw (\bstart-\ninth*\n:\radius+\eps) node {$\n$}; }
  

\foreach \n in {1,2,3,4,5,6,7,8,9} {\draw (b\n) circle(\dotrad) [fill=black];}


\draw  plot[smooth]
coordinates {(b9) (b3)}
[ postaction=decorate, decoration={markings,
  mark= at position 0.2 with \goodarrow, mark= at position 0.3 with \goodarrow,
mark= at position 0.4 with \goodarrow,mark= at position 0.5 with \goodarrow, mark= at position 0.8 with \goodarrow}];
\draw  plot[smooth]
coordinates {(b3) (b6)}
[ postaction=decorate, decoration={markings,
  mark= at position 0.2 with \goodarrow, mark= at position 0.4 with \goodarrow,
mark= at position 0.6 with \goodarrow,mark= at position 0.73 with \goodarrow, mark= at position 0.9 with \goodarrow}];
\draw  plot[smooth]
coordinates {(b6) (b9)}
[ postaction=decorate, decoration={markings,
  mark= at position 0.2 with \goodarrow, mark= at position 0.3 with \goodarrow,
mark= at position 0.4 with \goodarrow,mark= at position 0.5 with \goodarrow, mark= at position 0.8 with \goodarrow}];
\draw  plot[smooth] coordinates {(b7)(230:\radius*25/40) (-30:.25*\radius) (70:\radius*25/40) (b1)}
[ postaction=decorate, decoration={markings,
  mark= at position 0.1 with \goodarrow, mark= at position 0.29 with \goodarrow,
  mark= at position 0.42 with \goodarrow, mark= at position 0.5 with \goodarrow,
  mark= at position 0.6 with \goodarrow, mark= at position 0.69 with \goodarrow,
  mark= at position 0.79 with \goodarrow, mark= at position 0.85 with \goodarrow,
  mark= at position 0.92 with \goodarrow }];
\draw  plot[smooth] coordinates {(b1) (110:\radius*25/40)(-150:\radius*.25) (-50:\radius*25/40)(b4)}
[ postaction=decorate, decoration={markings,
  mark= at position 0.1 with \goodarrow, mark= at position 0.29 with \goodarrow,
  mark= at position 0.42 with \goodarrow, mark= at position 0.5 with \goodarrow,
  mark= at position 0.6 with \goodarrow, mark= at position 0.65 with \goodarrow,
  mark= at position 0.69 with \goodarrow, mark= at position 0.73 with \goodarrow,
  mark= at position 0.85 with \goodarrow }];
\draw  plot[smooth] coordinates {(b4)(-10:\radius*25/40) (90:\radius*.25) (190:\radius*25/40)(b7)}
[ postaction=decorate, decoration={markings,
  mark= at position 0.1 with \goodarrow, mark= at position 0.29 with \goodarrow,
  mark= at position 0.42 with \goodarrow, mark= at position 0.5 with \goodarrow,
  mark= at position 0.6 with \goodarrow, mark= at position 0.69 with \goodarrow,
  mark= at position 0.79 with \goodarrow, mark= at position 0.85 with \goodarrow,
  mark= at position 0.92 with \goodarrow }];
\draw  plot[smooth] coordinates {(b2) (80:\radius*23/40)(30:\radius*7/40) (-25:\radius*.38) (-100:\radius*.6)(b5)}
[ postaction=decorate, decoration={markings,
  mark= at position 0.13 with \goodarrow, mark= at position 0.22 with \goodarrow,
  mark= at position 0.3 with \goodarrow, mark= at position 0.44 with \goodarrow,
  mark= at position 0.53 with \goodarrow, mark= at position 0.63 with \goodarrow,
 mark= at position 0.75 with \goodarrow, mark = at position 0.9 with \goodarrow}];
 \draw  plot[smooth] coordinates {(b5) (-80:\radius*.5)(-40:\radius*.4) (-90:\radius*7/40) (-145:\radius*26/40)(b8)}
 [ postaction=decorate, decoration={markings,
  mark= at position 0.13 with \goodarrow, mark= at position 0.23 with \goodarrow,
  mark= at position 0.32 with \goodarrow, mark= at position 0.4 with \goodarrow,
  mark= at position 0.48 with \goodarrow,mark = at position 0.6 with \goodarrow, mark= at position 0.72 with \goodarrow,
 mark= at position 0.86 with \goodarrow}];
 \draw  plot[smooth] coordinates {(b8) (200:\radius*23/40)(150:\radius*7/40) (95:\radius*26/40)(b2)}
 [ postaction=decorate, decoration={markings,
  mark= at position 0.14 with \goodarrow, mark= at position 0.25 with \goodarrow,
  mark= at position 0.35 with \goodarrow, mark= at position 0.47 with \goodarrow,
  mark= at position 0.6 with \goodarrow, mark= at position 0.72 with \goodarrow,
 mark= at position 0.86 with \goodarrow}];


\foreach \n/\m/\r in {1/789/0.88, 2/891/0.88, 3/912/0.88, 4/123/0.85, 5/234/0.89, 6/345/0.85, 7/456/0.85, 8/567/.89, 9/678/.87}
{ \draw (\qstart-\ninth*\n:\r*\radius) node (q\m) {$\m$}; }

\foreach \m/\a/\r in {179/88/0.6 , 467/208/0.6, 178/117/0.4, 124/-20/.45, 457/237/.42 , 147/0/0, 127/65/.38, 145/-60/.31, 478/185/.38, 245/-89/.55}
{ \draw (\a:\r*\radius) node (q\m) {$\m$}; }


\foreach \t/\h in { 145/147,127/147, 147/178, 457/467, 457/145, 145/245, 124/145, 147/124, 245/124,
478/147, 147/457, 124/127, 179/127, 178/179, 178/478, 467/478}
{ \draw [quivarrow] (q\t) edge (q\h);}

 \end{tikzpicture}
\]
\caption{The quiver $\mu_{134}(\underline Q)$, where $Q$ is the quiver of Figure \ref{fig:quiv39}.}
\label{fig:mutpostquiv39}
\end{figure}
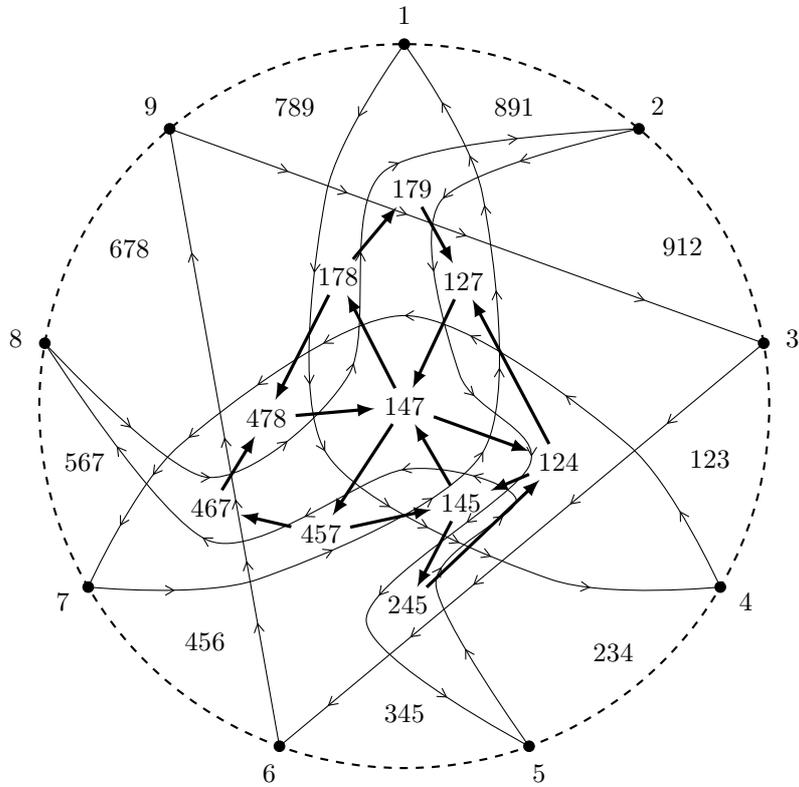

\begin{thm}\label{thm:transitive}
  \cite[Theorem 13.4]{Pos06} Any two reduced $(k,n)$-Postnikov diagrams are related by a sequence of geometric exchange, twisting and untwisting moves.
\end{thm}

\begin{cor}
  Any two basic cluster tilting objects in $\on{CM}(\hat B)$ (or $\underline{\on{CM}}(\hat B)$) that lie in $\mathcal B$ are related by a sequence of mutations.
\end{cor}
\begin{rmk}
  Given two cluster tilting objects as above $T, T'$, one can go from $T$ to $T'$ via a sequence of quiver mutations \emph{at vertices of valency 4}. Applying arbitrary mutations to the quiver can cause indecomposable summands of rank $\geq 2$ to appear in the cluster tilting object.
\end{rmk}

In \cite{HI11b}, the authors discuss the concept of \emph{planar mutation}, which is a more restrictive notion than that of quiver mutation. It has the property of preserving planarity. Their 
definition allows mutation at internal vertices of valency 4, or at boundary vertices of valency at most 4.
Mutation at an internal vertex of valency 4 of $\underline Q$ is precisely what is allowed by geometric exchange, but for boundary vertices the 
situation is different. Namely, we can mutate at the boundary vertices of $\underline Q$ if and only if they have valency 4 as vertices of $Q$, which is a stronger condition than having valency at most 4 in $\underline Q$.

\section{Self-injective cluster tilting objects in $\underline{\on{CM}}(B)$}\label{sec:main}

We are now ready to state our main result. In this section, let $D$ be a reduced $(k,n)$-Postnikov diagram.

\begin{lemma}
  \label{lem:T}
  The Postnikov diagram $D$ is symmetric if and only if $\,_\Phi T\cong T$ as left $B$-modules.
\end{lemma}

\begin{proof}
  Assume that $D$ is symmetric.
  As left $B$-modules, we have 
  \begin{align*}
    T = \bigoplus_{I\in \mathbb I} L_I  \cong\bigoplus_{I\in \mathbb I}\,_\Phi L_{I+k} = \bigoplus_{I\in \mathbb I+k}\,_\Phi L_{I} = \bigoplus_{I\in \mathbb I}\,_\Phi L_{I} = \,_\Phi T
  \end{align*}
  where we have used the isomorphism of Definition \ref{def:twist} and Lemma \ref{lem:I}.
  On the other hand, there can be an isomorphism $\bigoplus_{I\in \mathbb I}\,_\Phi L_{I+k} \cong \bigoplus_{I\in \mathbb I+k}\,_\Phi L_{I+k} $ only if $\mathbb I = \mathbb I+k$, which 
  by Lemma \ref{lem:I} implies that $D$ is symmetric.
\end{proof}

In other words, $\mathcal F T \cong T$, and recall that $\mathcal F = [2]$ on $\mathcal B\subseteq \underline{\on{CM}}(\hat B)$. 
If we call $\varphi:T \to \mathcal F T$ the canonical isomorphism with components $\varphi_I: L_I\to \,_{\Phi}(L_{I+k})$ as in Definition \ref{def:twist}, 
then there is an automorphism $\psi$ of $\on{End}_B(T)$ given by $$\psi: a\mapsto \varphi\circ a\circ \varphi^{-1}.$$
By Remark \ref{rmk:projinj}, $\mathcal F$ sends projectives to projectives, so the automorphism $\psi$ induces an automorphism of $\underline{\on{End}}_B(T)$, which we still denote by $\psi$.

\begin{thm}\label{thm:main}
  Let $D$ be a reduced $(k,n)$-Postnikov diagram. Then $D$ is symmetric if and only if the $B$-module $T\in \underline{\on{CM}}(B)$ is self-injective. In this case the Nakayama permutation
  is given by $\sigma(I) = I-k$, and a
  Nakayama automorphism given by $\psi$.\end{thm}

\begin{proof}
  By Corollary~\ref{cor:endiso}, $T$ is self-injective as a $B$-module if and only if it is self-injective as a $\hat B$-module. Thus we can work in $\underline{\on{CM}}(\hat B)$, 
  where we have a 2-Calabi-Yau structure.
  By Lemma \ref{lem:T} and Theorem \ref{thm:shift}, we have that $D$ is symmetric if and only if
    $T\cong \,_{\Phi}T\cong T[2]$.
    Moreover, $T\cong T[2]$ if and only if $T$ is self-injective, by Proposition \ref{prop:selfinj}.
    In this case, we have $L_I[2]\cong L_{I-k}$, which gives $\sigma:I\mapsto I-k$.
    Since $[2] = \mathcal F$ on the modules $L_I$, the map $\psi$ we have defined is exactly the map in the statement of Proposition \ref{prop:naka}.
    Thus we conclude that $\psi$ is a Nakayama automorphism of $\underline{\on{End}}_{\hat B}(T)\cong\underline{\on{End}}_{ B}(T) $.
\end{proof}

\begin{cor}\label{cor:main2}
  Let $D$ be a reduced $(k,n)$-Postnikov diagram. Then $D$ is symmetric if and only if
   $(\underline Q, \underline W)$ is a self-injective quiver with potential.
  In this case, the Nakayama permutation is $\sigma(I) = I-k$, induced by rotation by $2\pi k/n$, and a Nakayama automorphism of $\wp(\underline Q, \underline W)$ is given by $\Psi$ (see Section \ref{sec:post}).
\end{cor}

\begin{proof}
  By Theorem \ref{thm:bkm} and Lemma \ref{lem:descent} we have that $\underline{\on{End}}_B(T)\cong \wp(\underline Q, \underline W)$. 
  The functor $\mathcal F$ on $\underline {\on{CM}}(B)$ sends $L_I$ to $L_{I-k}$ and $g_{JI}$ to $g_{J-k, I-k}$, so the automorphism 
  $\psi$ of $\Lambda$ defined by twisting with the canonical isomorphism $\varphi:T\to \mathcal F T$ corresponds to the quiver automorphism sending vertex $I$ to $I-k$ and an arrow $I\to J$ to an 
  arrow $I-k\to J-k$. Thus the action of $\psi$ on the quiver coincides with that of $\rho$, which in turn is the action on $\wp(\underline Q, \underline W)$ of the automorphism $\Psi$.
\end{proof}

\begin{rmk}
	Strictly speaking, the rotation $\rho$ acts on $D$ only if $D$ is chosen appropriately in the equivalence class modulo isotopy. In other words, the Nakayama automorphism acts by $\rho$ on $Q$ provided that $Q$ is embedded in the plane with the embedding of Lemma~\ref{lem:I}.
\end{rmk}

\begin{rmk}
  The automorphism $\psi$ of $\on{End}_B(T)$ induces the automorphism $\Phi^{-1}$ on $B^{opp}\subseteq \on{End}_B(T)$.
\end{rmk}

\begin{defin}\cite[Definition 4.1]{HI11b}.
  Let $(\underline Q, \underline W)$ be a self-injective quiver with potential constructed from a reduced $(k,n)$-Postnikov diagram. In this case, the Nakayama permutation acts on vertices by $\sigma: I\mapsto I-k$.
  Call $(I) = \left\{ \sigma ^j(I)\ |\ j\in \Z \right\}$ the orbit of $I$.
  Suppose that there are no arrows between any two vertices in $(I)$. Then we define the \emph{mutation at $(I)$} $\mu_{(I)}(\underline Q, \underline W)$ to be the composition of mutations at the vertices in $(I)$, applied to $\underline Q$.
  It is well defined since, by the assumption, it does not depend on the order of composition.
\end{defin}

The following theorem is stated in greater generality in \cite{HI11b}.
\begin{thm}
  \cite[Theorem 4.2]{HI11b}.
  If $(\underline Q, \underline W)$ is self-injective and $I$ satisfies the above condition (allowing $\mu_{(I)}$ to be defined), then $\mu_{(I)}(\underline Q, \underline W)$ is a self-injective quiver with potential with the same Nakayama permutation.
\end{thm}
In our setting, this result can be deduced immediately from Corollary \ref{cor:main2} if $I$ is mutable.
Indeed, applying geometric exchange along a mutable orbit of $\rho$ produces another symmetric $(k,n)$-Postnikov diagram, so the corresponding quiver is again self-injective with the same permutation.

\begin{rmk}
  By Theorem \ref{thm:transitive}, any two symmetric reduced $(k,n)$-Postnikov diagrams are related by a sequence of geometric exchanges. However, we do not know whether they are related by a sequence
  of geometric exchanges along Nakayama orbits.
\end{rmk}

\section{Cuts of self-injective quivers with potential}\label{sec:cuts}

In this section we study the 2-representation finite algebras one can get from a self-injective quiver with potential. We want to use the results of \cite{HI11b}, so again we need our Jacobian algebras to be completed.

\begin{defin}
  For a quiver with potential $(R, P)$, a \emph{cut} is a set of arrows which contains exactly one arrow from every cycle in $P$.
  The quiver $(R, P)$ \emph{has enough cuts} if every arrow is contained in a cut.

We can define a grading on $\hat\wp(R, P)$ by giving degree 1 to the arrows in a cut $C$ of $R$, since by definition the potential is then homogeneous of degree 1. The degree 0 
part of $\hat\wp(R, P)$ is denoted $\hat\wp(R, P)_C$ and called the \emph{truncated Jacobian algebra} of $\hat \wp(R, P)$ associated to $C$.
\end{defin}

Recall that an algebra is called $2$-representation finite if it has global dimension at most 2 and admits a cluster tilting module (cf.~\cite{Iya08}).
One reason to look at truncated Jacobian algebras is the following result (see for instance \cite{HI11b} for the definition of $3$-preprojective algebras).
\begin{thm}\cite[Theorem 3.11]{HI11b}.
  For any self-injective quiver with potential $(R, P)$ and cut $C$, the truncated Jacobian algebra $\hat\wp(R, P)_C$ is 2-representation finite. All basic 
  2-representation finite algebras arise in this way. Moreover, the 3-preprojective algebra of $\hat\wp(R, P)_C$ is isomorphic to $\hat\wp(R, P)$.
\end{thm}

Now if $D$ is a symmetric Postnikov diagram, by Theorem \ref{thm:main} the associated Jacobian algebra 
$ \Lambda = \wp(\underline Q, \underline W)$  is self-injective, and by Proposition \ref{prop:complete} it is isomorphic to $\hat\wp(\underline Q, \underline W)$.
So for any cut $C$ of $(\underline Q, \underline W)$ the truncated Jacobian algebra $\Lambda_C$ is 2-representation finite with 3-preprojective algebra isomorphic to $\Lambda$.

We need some notation for regions determined by Postnikov diagrams. A \emph{boundary region} is a region whose boundary is alternating (ignoring the boundary of the disk) 
and has a piece of the boundary circle as part of its boundary. These are precisely the regions labeled by cyclic intervals. A \emph{cyclic boundary region} is a cyclic region which shares an edge with a boundary region.
On the level of Postnikov diagrams, a cut of $(\underline Q, \underline W)$ is a set $C$ of (non-boundary) crossings of strands such that
\begin{enumerate}
  \item for every crossing $c\in C$, the two cyclic regions adjacent to $c$ are not both cyclic boundary regions, and
  \item every cyclic region which is not a cyclic boundary region is adjacent to exactly one crossing in $C$.
\end{enumerate}
In Figure \ref{fig:cut39} we illustrate such a cut, and in Figure \ref{fig:quivcut39} we show the corresponding cut on the quiver $\underline Q$ (dotted arrows are arrows in the cut).
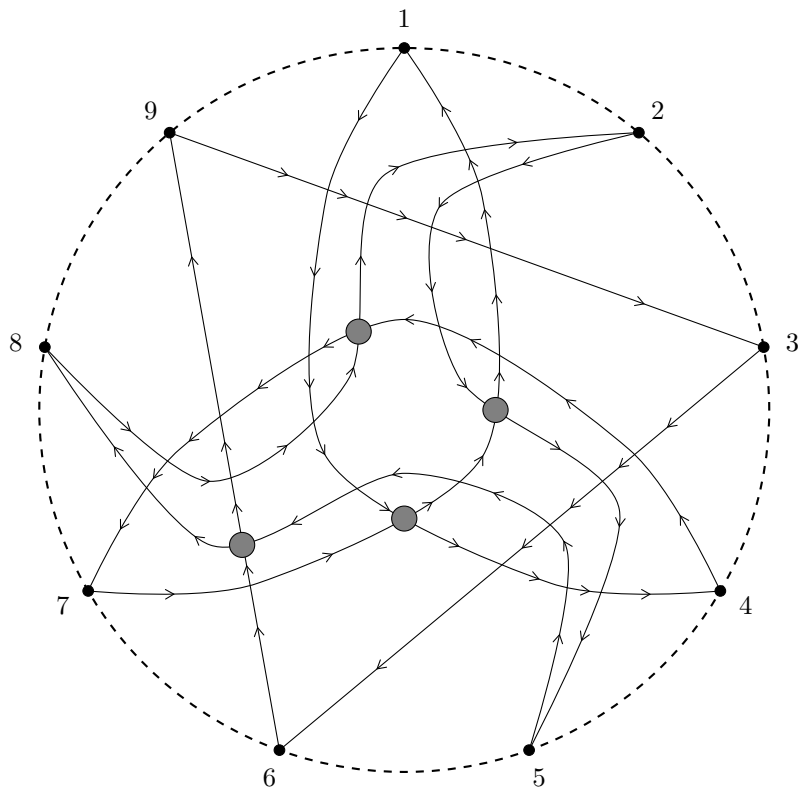
\begin{figure}
\[
\begin{tikzpicture}[baseline=(bb.base)]

\newcommand{\goodarrow}{\arrow{angle 60}}
\newcommand{\bstart}{130} 
\newcommand{\ninth}{40} 
\newcommand{\qstart}{150} 
\newcommand{\radius}{4.8cm} 
\newcommand{\eps}{11pt} 
\newcommand{\dotrad}{0.07cm} 

\path (0,0) node (bb) {};


\draw (0,0) circle(\radius) [thick,dashed];

\foreach \n in {1,...,9}
{ \coordinate (b\n) at (\bstart-\ninth*\n:\radius);
  \draw (\bstart-\ninth*\n:\radius+\eps) node {$\n$}; }
  

\foreach \n in {1,2,3,4,5,6,7,8,9} {\draw (b\n) circle(\dotrad) [fill=black];}


\draw  plot[smooth]
coordinates {(b9) (b3)}
[ postaction=decorate, decoration={markings,
  mark= at position 0.2 with \goodarrow, mark= at position 0.3 with \goodarrow,
mark= at position 0.4 with \goodarrow,mark= at position 0.5 with \goodarrow, mark= at position 0.8 with \goodarrow}];
\draw  plot[smooth]
coordinates {(b3) (b6)}
[ postaction=decorate, decoration={markings,
  mark= at position 0.2 with \goodarrow, mark= at position 0.3 with \goodarrow,
mark= at position 0.4 with \goodarrow,mark= at position 0.5 with \goodarrow, mark= at position 0.8 with \goodarrow}];
\draw  plot[smooth]
coordinates {(b6) (b9)}
[ postaction=decorate, decoration={markings,
  mark= at position 0.2 with \goodarrow, mark= at position 0.3 with \goodarrow,
mark= at position 0.4 with \goodarrow,mark= at position 0.5 with \goodarrow, mark= at position 0.8 with \goodarrow}];
\draw  plot[smooth] coordinates {(b7)(230:\radius*25/40) (-30:.25*\radius) (70:\radius*25/40) (b1)}
[ postaction=decorate, decoration={markings,
  mark= at position 0.1 with \goodarrow, mark= at position 0.29 with \goodarrow,
  mark= at position 0.42 with \goodarrow, mark= at position 0.5 with \goodarrow,
  mark= at position 0.6 with \goodarrow, mark= at position 0.69 with \goodarrow,
  mark= at position 0.79 with \goodarrow, mark= at position 0.85 with \goodarrow,
  mark= at position 0.92 with \goodarrow }];
\draw  plot[smooth] coordinates {(b1) (110:\radius*25/40)(-150:\radius*.25) (-50:\radius*25/40)(b4)}
[ postaction=decorate, decoration={markings,
  mark= at position 0.1 with \goodarrow, mark= at position 0.29 with \goodarrow,
  mark= at position 0.42 with \goodarrow, mark= at position 0.5 with \goodarrow,
  mark= at position 0.6 with \goodarrow, mark= at position 0.69 with \goodarrow,
  mark= at position 0.79 with \goodarrow, mark= at position 0.85 with \goodarrow,
  mark= at position 0.92 with \goodarrow }];
\draw  plot[smooth] coordinates {(b4)(-10:\radius*25/40) (90:\radius*.25) (190:\radius*25/40)(b7)}
[ postaction=decorate, decoration={markings,
  mark= at position 0.1 with \goodarrow, mark= at position 0.29 with \goodarrow,
  mark= at position 0.42 with \goodarrow, mark= at position 0.5 with \goodarrow,
  mark= at position 0.6 with \goodarrow, mark= at position 0.69 with \goodarrow,
  mark= at position 0.79 with \goodarrow, mark= at position 0.85 with \goodarrow,
  mark= at position 0.92 with \goodarrow }];
\draw  plot[smooth] coordinates {(b2) (80:\radius*23/40)(30:\radius*7/40) (-25:\radius*26/40)(b5)}
[ postaction=decorate, decoration={markings,
  mark= at position 0.14 with \goodarrow, mark= at position 0.25 with \goodarrow,
  mark= at position 0.35 with \goodarrow, mark= at position 0.47 with \goodarrow,
  mark= at position 0.6 with \goodarrow, mark= at position 0.72 with \goodarrow,
 mark= at position 0.86 with \goodarrow}];
 \draw  plot[smooth] coordinates {(b5) (-40:\radius*23/40)(-90:\radius*7/40) (-145:\radius*26/40)(b8)}
 [ postaction=decorate, decoration={markings,
  mark= at position 0.14 with \goodarrow, mark= at position 0.25 with \goodarrow,
  mark= at position 0.35 with \goodarrow, mark= at position 0.47 with \goodarrow,
  mark= at position 0.6 with \goodarrow, mark= at position 0.72 with \goodarrow,
 mark= at position 0.86 with \goodarrow}];
 \draw  plot[smooth] coordinates {(b8) (200:\radius*23/40)(150:\radius*7/40) (95:\radius*26/40)(b2)}
 [ postaction=decorate, decoration={markings,
  mark= at position 0.14 with \goodarrow, mark= at position 0.25 with \goodarrow,
  mark= at position 0.35 with \goodarrow, mark= at position 0.47 with \goodarrow,
  mark= at position 0.6 with \goodarrow, mark= at position 0.72 with \goodarrow,
 mark= at position 0.86 with \goodarrow}];


 \node[shape = circle, fill = gray, draw] at (120:.25*\radius) {};
 \node[shape = circle, fill = gray, draw] at (0:.25*\radius) {};
 \node[shape = circle, fill = gray, draw] at (-90:.3*\radius) {};
 \node[shape = circle, fill = gray, draw] at (220:.58*\radius) {};

 \end{tikzpicture}
\]
\caption{A cut on a symmetric $(3,9)$-Postnikov diagram.}
\label{fig:cut39}
\end{figure}
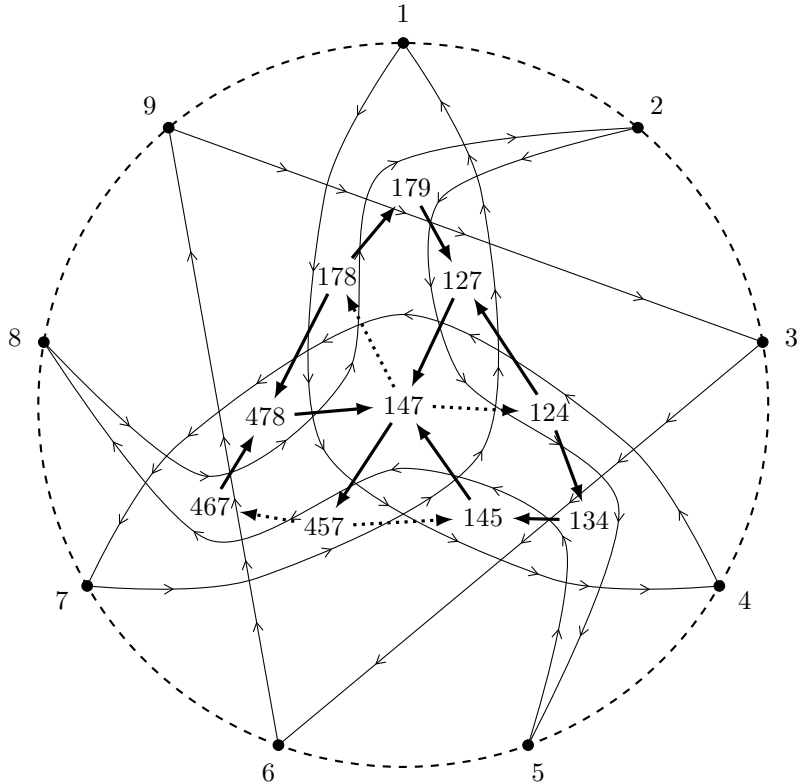
\begin{figure}
\[
\begin{tikzpicture}[baseline=(bb.base),
    quivarrow/.style={black, -latex, very thick},
    cutarrow/.style = {black, -latex,very thick, dotted}
  ]

\newcommand{\goodarrow}{\arrow{angle 60}}
\newcommand{\bstart}{130} 
\newcommand{\ninth}{40} 
\newcommand{\qstart}{150} 
\newcommand{\radius}{4.8cm} 
\newcommand{\eps}{11pt} 
\newcommand{\dotrad}{0.07cm} 

\path (0,0) node (bb) {};


\draw (0,0) circle(\radius) [thick,dashed];

\foreach \n in {1,...,9}
{ \coordinate (b\n) at (\bstart-\ninth*\n:\radius);
  \draw (\bstart-\ninth*\n:\radius+\eps) node {$\n$}; }
  

\foreach \n in {1,2,3,4,5,6,7,8,9} {\draw (b\n) circle(\dotrad) [fill=black];}


\draw  plot[smooth]
coordinates {(b9) (b3)}
[ postaction=decorate, decoration={markings,
  mark= at position 0.2 with \goodarrow, mark= at position 0.3 with \goodarrow,
mark= at position 0.4 with \goodarrow,mark= at position 0.5 with \goodarrow, mark= at position 0.8 with \goodarrow}];
\draw  plot[smooth]
coordinates {(b3) (b6)}
[ postaction=decorate, decoration={markings,
  mark= at position 0.2 with \goodarrow, mark= at position 0.3 with \goodarrow,
mark= at position 0.4 with \goodarrow,mark= at position 0.5 with \goodarrow, mark= at position 0.8 with \goodarrow}];
\draw  plot[smooth]
coordinates {(b6) (b9)}
[ postaction=decorate, decoration={markings,
  mark= at position 0.2 with \goodarrow, mark= at position 0.3 with \goodarrow,
mark= at position 0.4 with \goodarrow,mark= at position 0.5 with \goodarrow, mark= at position 0.8 with \goodarrow}];
\draw  plot[smooth] coordinates {(b7)(230:\radius*25/40) (-30:.25*\radius) (70:\radius*25/40) (b1)}
[ postaction=decorate, decoration={markings,
  mark= at position 0.1 with \goodarrow, mark= at position 0.29 with \goodarrow,
  mark= at position 0.42 with \goodarrow, mark= at position 0.5 with \goodarrow,
  mark= at position 0.6 with \goodarrow, mark= at position 0.69 with \goodarrow,
  mark= at position 0.79 with \goodarrow, mark= at position 0.85 with \goodarrow,
  mark= at position 0.92 with \goodarrow }];
\draw  plot[smooth] coordinates {(b1) (110:\radius*25/40)(-150:\radius*.25) (-50:\radius*25/40)(b4)}
[ postaction=decorate, decoration={markings,
  mark= at position 0.1 with \goodarrow, mark= at position 0.29 with \goodarrow,
  mark= at position 0.42 with \goodarrow, mark= at position 0.5 with \goodarrow,
  mark= at position 0.6 with \goodarrow, mark= at position 0.69 with \goodarrow,
  mark= at position 0.79 with \goodarrow, mark= at position 0.85 with \goodarrow,
  mark= at position 0.92 with \goodarrow }];
\draw  plot[smooth] coordinates {(b4)(-10:\radius*25/40) (90:\radius*.25) (190:\radius*25/40)(b7)}
[ postaction=decorate, decoration={markings,
  mark= at position 0.1 with \goodarrow, mark= at position 0.29 with \goodarrow,
  mark= at position 0.42 with \goodarrow, mark= at position 0.5 with \goodarrow,
  mark= at position 0.6 with \goodarrow, mark= at position 0.69 with \goodarrow,
  mark= at position 0.79 with \goodarrow, mark= at position 0.85 with \goodarrow,
  mark= at position 0.92 with \goodarrow }];
\draw  plot[smooth] coordinates {(b2) (80:\radius*23/40)(30:\radius*7/40) (-25:\radius*26/40)(b5)}
[ postaction=decorate, decoration={markings,
  mark= at position 0.14 with \goodarrow, mark= at position 0.25 with \goodarrow,
  mark= at position 0.35 with \goodarrow, mark= at position 0.47 with \goodarrow,
  mark= at position 0.6 with \goodarrow, mark= at position 0.72 with \goodarrow,
 mark= at position 0.86 with \goodarrow}];
 \draw  plot[smooth] coordinates {(b5) (-40:\radius*23/40)(-90:\radius*7/40) (-145:\radius*26/40)(b8)}
 [ postaction=decorate, decoration={markings,
  mark= at position 0.14 with \goodarrow, mark= at position 0.25 with \goodarrow,
  mark= at position 0.35 with \goodarrow, mark= at position 0.47 with \goodarrow,
  mark= at position 0.6 with \goodarrow, mark= at position 0.72 with \goodarrow,
 mark= at position 0.86 with \goodarrow}];
 \draw  plot[smooth] coordinates {(b8) (200:\radius*23/40)(150:\radius*7/40) (95:\radius*26/40)(b2)}
 [ postaction=decorate, decoration={markings,
  mark= at position 0.14 with \goodarrow, mark= at position 0.25 with \goodarrow,
  mark= at position 0.35 with \goodarrow, mark= at position 0.47 with \goodarrow,
  mark= at position 0.6 with \goodarrow, mark= at position 0.72 with \goodarrow,
 mark= at position 0.86 with \goodarrow}];


\foreach \m/\a/\r in {179/88/0.6 , 134/328/0.6, 467/208/0.6, 178/117/0.4, 124/-3/.4, 457/237/.4 , 147/0/0, 127/65/.38, 145/-55/.38, 478/185/.38}
{ \draw (\a:\r*\radius) node (q\m) {$\m$}; }

\foreach \t/\h in { 145/147, 134/145, 124/134, 127/147,  
478/147, 147/457, 124/127, 179/127, 178/179, 178/478, 467/478}
{ \draw [quivarrow] (q\t) edge (q\h);}

\draw [cutarrow] (q457) edge (q145);
\draw [cutarrow] (q457) edge (q467);
\draw [cutarrow] (q147) edge (q178);
\draw [cutarrow] (q147) edge (q124);

 \end{tikzpicture}
\]
\caption{A cut on the self-injective quiver with potential corresponding to the Postnikov diagram of Figure \ref{fig:cut39}.}
\label{fig:quivcut39}
\end{figure}

There is a notion of mutation of cuts that corresponds exactly to taking the quiver of the corresponding 2-APR tilt of $\Lambda$ (see \cite{HI11b} for details).
\begin{defin}\cite[Definition 6.10]{HI11b}.
  Let $(R, P)$ be a quiver with potential with a cut $C$. A vertex $x$ of $R$ is a \emph{strict source} if all arrows ending at $x$ belong to $C$ and all arrows 
  starting at $x$ do not belong to $C$. For a strict source $x$, call the \emph{cut-mutation} $\mu_x^+(C)$ of $R_1$ the set of arrows we get by removing all arrows ending at $x$ from $C$, and 
  adding all arrows starting at $x$ to $C$. Dually, we define a \emph{strict sink} and the cut-mutation $\mu_x^-(C)$.
\end{defin}
It is clear that cut-mutation transforms strict sources into strict sinks and vice-versa.

For a quiver with potential $(\underline Q, \underline W)$ constructed from a Postnikov diagram, strict sources and strict sinks are precisely 
alternating regions such that every second crossing (except those with a boundary region) on their boundary is contained in the cut. Cut-mutation consists of replacing the
crossings on the boundary of such alternating regions with their complement (again, ignoring the crossings with a boundary region).
In Figures \ref{fig:cutmut39} and \ref{fig:quivcutmut39} we illustrate $\mu_{457}^+(C)$ for the cut $C$ of Figures \ref{fig:cut39} and \ref{fig:quivcut39}.
\begin{figure}
\[
\begin{tikzpicture}[baseline=(bb.base)]

\newcommand{\goodarrow}{\arrow{angle 60}}
\newcommand{\bstart}{130} 
\newcommand{\ninth}{40} 
\newcommand{\qstart}{150} 
\newcommand{\radius}{4.8cm} 
\newcommand{\eps}{11pt} 
\newcommand{\dotrad}{0.07cm} 

\path (0,0) node (bb) {};


\draw (0,0) circle(\radius) [thick,dashed];

\foreach \n in {1,...,9}
{ \coordinate (b\n) at (\bstart-\ninth*\n:\radius);
  \draw (\bstart-\ninth*\n:\radius+\eps) node {$\n$}; }
  

\foreach \n in {1,2,3,4,5,6,7,8,9} {\draw (b\n) circle(\dotrad) [fill=black];}


\draw  plot[smooth]
coordinates {(b9) (b3)}
[ postaction=decorate, decoration={markings,
  mark= at position 0.2 with \goodarrow, mark= at position 0.3 with \goodarrow,
mark= at position 0.4 with \goodarrow,mark= at position 0.5 with \goodarrow, mark= at position 0.8 with \goodarrow}];
\draw  plot[smooth]
coordinates {(b3) (b6)}
[ postaction=decorate, decoration={markings,
  mark= at position 0.2 with \goodarrow, mark= at position 0.3 with \goodarrow,
mark= at position 0.4 with \goodarrow,mark= at position 0.5 with \goodarrow, mark= at position 0.8 with \goodarrow}];
\draw  plot[smooth]
coordinates {(b6) (b9)}
[ postaction=decorate, decoration={markings,
  mark= at position 0.2 with \goodarrow, mark= at position 0.3 with \goodarrow,
mark= at position 0.4 with \goodarrow,mark= at position 0.5 with \goodarrow, mark= at position 0.8 with \goodarrow}];
\draw  plot[smooth] coordinates {(b7)(230:\radius*25/40) (-30:.25*\radius) (70:\radius*25/40) (b1)}
[ postaction=decorate, decoration={markings,
  mark= at position 0.1 with \goodarrow, mark= at position 0.29 with \goodarrow,
  mark= at position 0.42 with \goodarrow, mark= at position 0.5 with \goodarrow,
  mark= at position 0.6 with \goodarrow, mark= at position 0.69 with \goodarrow,
  mark= at position 0.79 with \goodarrow, mark= at position 0.85 with \goodarrow,
  mark= at position 0.92 with \goodarrow }];
\draw  plot[smooth] coordinates {(b1) (110:\radius*25/40)(-150:\radius*.25) (-50:\radius*25/40)(b4)}
[ postaction=decorate, decoration={markings,
  mark= at position 0.1 with \goodarrow, mark= at position 0.29 with \goodarrow,
  mark= at position 0.42 with \goodarrow, mark= at position 0.5 with \goodarrow,
  mark= at position 0.6 with \goodarrow, mark= at position 0.69 with \goodarrow,
  mark= at position 0.79 with \goodarrow, mark= at position 0.85 with \goodarrow,
  mark= at position 0.92 with \goodarrow }];
\draw  plot[smooth] coordinates {(b4)(-10:\radius*25/40) (90:\radius*.25) (190:\radius*25/40)(b7)}
[ postaction=decorate, decoration={markings,
  mark= at position 0.1 with \goodarrow, mark= at position 0.29 with \goodarrow,
  mark= at position 0.42 with \goodarrow, mark= at position 0.5 with \goodarrow,
  mark= at position 0.6 with \goodarrow, mark= at position 0.69 with \goodarrow,
  mark= at position 0.79 with \goodarrow, mark= at position 0.85 with \goodarrow,
  mark= at position 0.92 with \goodarrow }];
\draw  plot[smooth] coordinates {(b2) (80:\radius*23/40)(30:\radius*7/40) (-25:\radius*26/40)(b5)}
[ postaction=decorate, decoration={markings,
  mark= at position 0.14 with \goodarrow, mark= at position 0.25 with \goodarrow,
  mark= at position 0.35 with \goodarrow, mark= at position 0.47 with \goodarrow,
  mark= at position 0.6 with \goodarrow, mark= at position 0.72 with \goodarrow,
 mark= at position 0.86 with \goodarrow}];
 \draw  plot[smooth] coordinates {(b5) (-40:\radius*23/40)(-90:\radius*7/40) (-145:\radius*26/40)(b8)}
 [ postaction=decorate, decoration={markings,
  mark= at position 0.14 with \goodarrow, mark= at position 0.25 with \goodarrow,
  mark= at position 0.35 with \goodarrow, mark= at position 0.47 with \goodarrow,
  mark= at position 0.6 with \goodarrow, mark= at position 0.72 with \goodarrow,
 mark= at position 0.86 with \goodarrow}];
 \draw  plot[smooth] coordinates {(b8) (200:\radius*23/40)(150:\radius*7/40) (95:\radius*26/40)(b2)}
 [ postaction=decorate, decoration={markings,
  mark= at position 0.14 with \goodarrow, mark= at position 0.25 with \goodarrow,
  mark= at position 0.35 with \goodarrow, mark= at position 0.47 with \goodarrow,
  mark= at position 0.6 with \goodarrow, mark= at position 0.72 with \goodarrow,
 mark= at position 0.86 with \goodarrow}];

 \draw (237:.4*\radius) node (q457) {$457$};


 \node[shape = circle, fill = gray, draw] at (120:.25*\radius) {};
 \node[shape = circle, fill = gray, draw] at (0:.25*\radius) {};
 \node[shape = circle, fill = gray, draw] at (240:.25*\radius) {};
 \node[shape = circle, fill = white ,draw] at (-90:.3*\radius) {};
 \node[shape = circle, fill = white, draw] at (220:.58*\radius) {};

 \end{tikzpicture}
\]
\caption{The cut-mutation $\mu_{457}^+(C)$ of the cut in Figure \ref{fig:cut39}.}
\label{fig:cutmut39}
\end{figure}
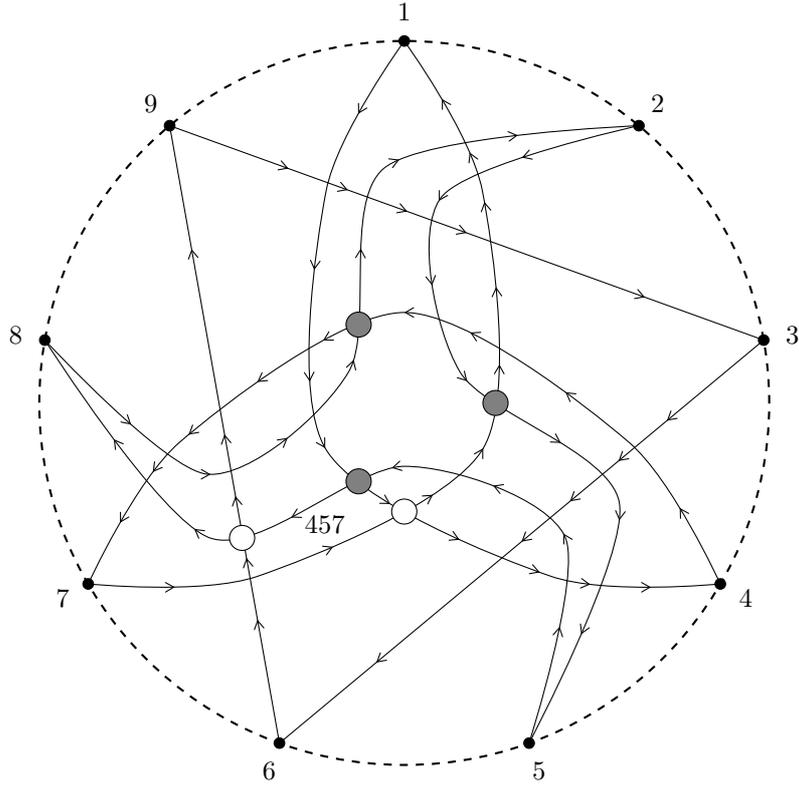
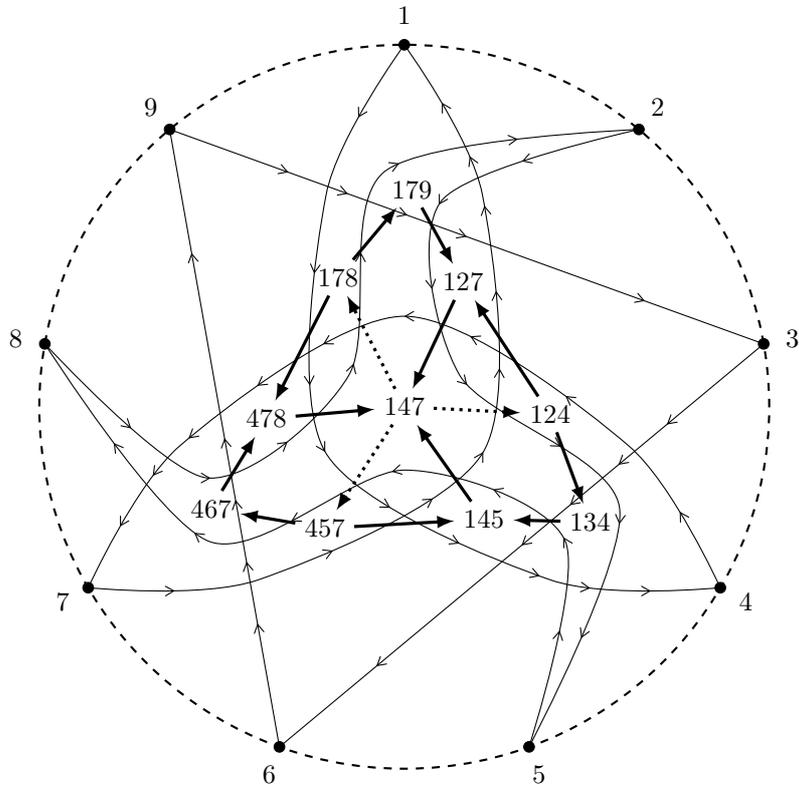
\begin{figure}
\[
\begin{tikzpicture}[baseline=(bb.base),
    quivarrow/.style={black, -latex, very thick},
    cutarrow/.style = {black, -latex,very thick, dotted}
  ]

\newcommand{\goodarrow}{\arrow{angle 60}}
\newcommand{\bstart}{130} 
\newcommand{\ninth}{40} 
\newcommand{\qstart}{150} 
\newcommand{\radius}{4.8cm} 
\newcommand{\eps}{11pt} 
\newcommand{\dotrad}{0.07cm} 

\path (0,0) node (bb) {};


\draw (0,0) circle(\radius) [thick,dashed];

\foreach \n in {1,...,9}
{ \coordinate (b\n) at (\bstart-\ninth*\n:\radius);
  \draw (\bstart-\ninth*\n:\radius+\eps) node {$\n$}; }
  

\foreach \n in {1,2,3,4,5,6,7,8,9} {\draw (b\n) circle(\dotrad) [fill=black];}


\draw  plot[smooth]
coordinates {(b9) (b3)}
[ postaction=decorate, decoration={markings,
  mark= at position 0.2 with \goodarrow, mark= at position 0.3 with \goodarrow,
mark= at position 0.4 with \goodarrow,mark= at position 0.5 with \goodarrow, mark= at position 0.8 with \goodarrow}];
\draw  plot[smooth]
coordinates {(b3) (b6)}
[ postaction=decorate, decoration={markings,
  mark= at position 0.2 with \goodarrow, mark= at position 0.3 with \goodarrow,
mark= at position 0.4 with \goodarrow,mark= at position 0.5 with \goodarrow, mark= at position 0.8 with \goodarrow}];
\draw  plot[smooth]
coordinates {(b6) (b9)}
[ postaction=decorate, decoration={markings,
  mark= at position 0.2 with \goodarrow, mark= at position 0.3 with \goodarrow,
mark= at position 0.4 with \goodarrow,mark= at position 0.5 with \goodarrow, mark= at position 0.8 with \goodarrow}];
\draw  plot[smooth] coordinates {(b7)(230:\radius*25/40) (-30:.25*\radius) (70:\radius*25/40) (b1)}
[ postaction=decorate, decoration={markings,
  mark= at position 0.1 with \goodarrow, mark= at position 0.29 with \goodarrow,
  mark= at position 0.42 with \goodarrow, mark= at position 0.5 with \goodarrow,
  mark= at position 0.6 with \goodarrow, mark= at position 0.69 with \goodarrow,
  mark= at position 0.79 with \goodarrow, mark= at position 0.85 with \goodarrow,
  mark= at position 0.92 with \goodarrow }];
\draw  plot[smooth] coordinates {(b1) (110:\radius*25/40)(-150:\radius*.25) (-50:\radius*25/40)(b4)}
[ postaction=decorate, decoration={markings,
  mark= at position 0.1 with \goodarrow, mark= at position 0.29 with \goodarrow,
  mark= at position 0.42 with \goodarrow, mark= at position 0.5 with \goodarrow,
  mark= at position 0.6 with \goodarrow, mark= at position 0.69 with \goodarrow,
  mark= at position 0.79 with \goodarrow, mark= at position 0.85 with \goodarrow,
  mark= at position 0.92 with \goodarrow }];
\draw  plot[smooth] coordinates {(b4)(-10:\radius*25/40) (90:\radius*.25) (190:\radius*25/40)(b7)}
[ postaction=decorate, decoration={markings,
  mark= at position 0.1 with \goodarrow, mark= at position 0.29 with \goodarrow,
  mark= at position 0.42 with \goodarrow, mark= at position 0.5 with \goodarrow,
  mark= at position 0.6 with \goodarrow, mark= at position 0.69 with \goodarrow,
  mark= at position 0.79 with \goodarrow, mark= at position 0.85 with \goodarrow,
  mark= at position 0.92 with \goodarrow }];
\draw  plot[smooth] coordinates {(b2) (80:\radius*23/40)(30:\radius*7/40) (-25:\radius*26/40)(b5)}
[ postaction=decorate, decoration={markings,
  mark= at position 0.14 with \goodarrow, mark= at position 0.25 with \goodarrow,
  mark= at position 0.35 with \goodarrow, mark= at position 0.47 with \goodarrow,
  mark= at position 0.6 with \goodarrow, mark= at position 0.72 with \goodarrow,
 mark= at position 0.86 with \goodarrow}];
 \draw  plot[smooth] coordinates {(b5) (-40:\radius*23/40)(-90:\radius*7/40) (-145:\radius*26/40)(b8)}
 [ postaction=decorate, decoration={markings,
  mark= at position 0.14 with \goodarrow, mark= at position 0.25 with \goodarrow,
  mark= at position 0.35 with \goodarrow, mark= at position 0.47 with \goodarrow,
  mark= at position 0.6 with \goodarrow, mark= at position 0.72 with \goodarrow,
 mark= at position 0.86 with \goodarrow}];
 \draw  plot[smooth] coordinates {(b8) (200:\radius*23/40)(150:\radius*7/40) (95:\radius*26/40)(b2)}
 [ postaction=decorate, decoration={markings,
  mark= at position 0.14 with \goodarrow, mark= at position 0.25 with \goodarrow,
  mark= at position 0.35 with \goodarrow, mark= at position 0.47 with \goodarrow,
  mark= at position 0.6 with \goodarrow, mark= at position 0.72 with \goodarrow,
 mark= at position 0.86 with \goodarrow}];


\foreach \m/\a/\r in {179/88/0.6 , 134/328/0.6, 467/208/0.6, 178/117/0.4, 124/-3/.4, 457/237/.4 , 147/0/0, 127/65/.38, 145/-55/.38, 478/185/.38}
{ \draw (\a:\r*\radius) node (q\m) {$\m$}; }

\foreach \t/\h in { 145/147, 134/145, 124/134, 127/147, 457/145, 457/467,
478/147,  124/127, 179/127, 178/179, 178/478, 467/478}
{ \draw [quivarrow] (q\t) edge (q\h);}

\draw [cutarrow] (q147) edge (q457);
\draw [cutarrow] (q147) edge (q178);
\draw [cutarrow] (q147) edge (q124);

 \end{tikzpicture}
\]
\caption{The cut-mutation $\mu_{457}^+(C)$ of the cut $C$ of Figure \ref{fig:quivcut39}.}
\label{fig:quivcutmut39}
\end{figure}

Quivers with potential obtained from Postnikov diagrams are by definition planar in the sense of \cite[Definition 9.1]{HI11b}. We illustrate one application.

\begin{thm}\cite[Theorem 9.2]{HI11b}.
  Let $(R, P)$ be a self-injective planar quiver with potential that has enough cuts. Then all truncated Jacobian algebras $\hat\wp(R, P)_C$
  are iterated 2-APR tilts of each other. In particular they are derived equivalent.
\end{thm}

The assumption is satisfied in our setting:
\begin{prop}
  If $(\underline Q, \underline W)$ is a self-injective quiver with potential constructed from a symmetric Postnikov diagram, then $(\underline Q, \underline W)$ has enough cuts.
\end{prop}

\begin{proof}
  The planar embedding of $\underline Q$ can be taken to be a so-called \emph{isoradial embedding} (\cite[Theorem 5.7]{BKM16}). This means that 
  all the faces (i.e.~cycles in $\underline W$) of $\underline Q$ are polygons inscribed in a unit circle. Then proceed as follows. Pick an arrow $a$, and a face $F$ adjacent to $a$. Without loss of generality, 
  assume that $F$ is oriented clockwise.
  Now choose a point on the unit circle lying on the arc determined by $a$ on the circle around $F$. Mark the same point on every copy of the unit circle around 
  all clockwise-oriented faces. One can make the initial choice of a point such that no vertices are marked this way. For every clockwise-oriented face $F'$, mark the arrow on its 
  boundary corresponding to the arc determined by the marked point on the circle around $F'$.
  The set of arrows marked this way has the following property: every face has exactly one boundary arrow in this set, except possibly some counterclockwise-oriented faces adjacent to 
  the boundary of the quiver (\cite[\S 0.9]{Boc16}). Thus if we choose one boundary arrow for each of these faces, we get a cut containing $a$ and we are done.
\end{proof}

\begin{cor}
   If $(\underline Q, \underline W)$ is a self-injective quiver with potential constructed from a symmetric Postnikov diagram, then all truncated Jacobian algebras 
   $\wp(\underline Q, \underline W)_C$ are iterated 2-APR tilts of each other. In particular they are derived equivalent.
\end{cor}

Thus we get not only a way of generating many new examples of self-injective quivers with potential, but also the corresponding new 2-representation finite algebras.

An interesting property that 2-representation finite algebras can have is that of being \emph{$l$-homogeneous} (see \cite[Definition 1.2]{HI11}). 
It follows from \cite[Theorem 2.3]{HI11} that a truncated Jacobian algebra $\wp(R, W)_C$ as above is $l$-homogenous for some $l$ if and only if $\psi(C) = C$ (which in our case 
just means that $C$ is invariant under rotation by $2\pi k/n$).
Thus, examples coming from Postnikov diagrams are a good source of $l$-homogeneous, 2-representation finite algebras. One property that these algebras have is the following:
\begin{thm}\cite[Theorem 1.3]{HI11}.
  A finite-dimensional algebra of global dimension at most 2 is $l$-homogeneous 2-representation finite if and only if it is twisted $2\frac{l-1}{l}$-Calabi-Yau.
\end{thm}
Twisted fractionally Calabi-Yau algebras can be tensored over $\C$, so we get for instance:
\begin{prop}\cite[Corollary 1.5]{HI11}
  If $\Lambda_1, \Lambda_2$ are
  $l$-homogeneous 2-representation finite algebras, then $\Lambda_1 \otimes_{\C} \Lambda_2$ is $l$-homogeneous 4-representation finite.
\end{prop}
As we mentioned, in the case of Postnikov diagrams it is easy to see whether a truncated Jacobian algebra is homogeneous: one needs to check whether the cut is invariant under $\rho$. 
For instance, the truncated Jacobian algebra of Figure \ref{fig:quivcut39} is not homogeneous, but the one of Figure \ref{fig:quivcutmut39} is.
In particular, the $N$-fold tensor product of the latter algebra with itself is $2N$-representation finite.

\section{Examples}\label{sec:ex}

We present some self-injective quivers with potential obtained from symmetric $(k,n)$-Postnikov diagrams. The Nakayama permutation 
acts by rotation by $2\pi k/n$, and has order $a = n\left/\on{GCD}(k,n).\right.$

The quivers with potential on the left hand side of Figure \ref{fig:quiv4a} and of Figure \ref{fig:cob1} (corresponding to $(k,n) \in\left\{ 
(3,12), (4,10)\right\}$) had already been found by Martin Herschend, and the latter had also been found 
independently by Sefi Ladkani. These results are not published.
A symmetric $(4,8)$-Postnikov diagram had appeared in \cite[Section 11]{MR13}.

\subsection{The case $a = 2$}

If $a = 2$ then we must have $n = 2k$. 
\begin{prop}
  For every $k>1$, there exists a symmetric $(k, 2k)$-Postnikov diagram whose associated self-injective quiver with potential is a square grid with $(k-1)$ vertices on each side.
\end{prop}

\begin{proof}
  The construction in Figure \ref{fig:postsq} yields such a symmetric Postnikov diagram, and it produces the correct quiver. To avoid clogging the picture, we have not 
  marked the direction of the strands. They should be understood as follows: strand $i$ crosses strand $i+k$ coming from the left at vertex $i$ if and only if $i$ is odd.
  The strands $k$ and $2k$ cross strands $k-1, 2k-2, k-3, 2k-4, \dots, k-2, 2k-1$ in this order for $k$ and the opposite for $2k$, or viceversa depending on the parity of $k$.
\end{proof}
Such quivers and their planar mutations were already studied in \cite[\S9.3]{HI11b}. In Figure \ref{fig:squares} we show the cases $k = 4,5$.

\subsection{The case $a = 3$}

If $a = 3$ then we may assume $n = 3k$. Notice that we are treating the cases of clockwise and counterclockwise rotation together, and this is justified by the fact that
now we are focusing on the self-injective algebra, which does not change if we reflect the quiver (even though the two quivers are not isomorphic as planar quivers with faces).
Here one could expect to get the family of self-injective quivers with potential given by 3-preprojective algebras of type $A_j$ (cf.\cite[\S9.2]{HI11b}). This is true for $k\in \left\{ 2,3,4 \right\}$, 
where we get the quivers in Figures \ref{fig:triangles} and \ref{fig:triangles2} (corresponding to type $A_{2}, A_4, A_6$). Notice that the quiver corresponding to the symmetric Postnikov diagram of Figure \ref{fig:post39}
is equivalent to the one of type $A_4$ by mutation at the orbit consisting of the vertices of the big triangle. 
Notice that type $A_j$ with $j$ odd cannot appear this way, since the number of alternating internal faces of a $(k, 3k)$-Postnikov diagram is $(k-1)(2k-1)$.
The Postnikov diagrams corresponding to these three quivers are shown in Figures \ref{fig:tri1}, \ref{fig:tri2} and \ref{fig:tri3}.

\subsection{The case $a= 4$}

For $a = 4$ (we may then assume that $n = 4k$, since the only elements of order 4 in $\Z/4\Z$ are $\pm 1$) we present three self-injective quivers with potential coming from symmetric Postnikov diagrams,
 for $(k, n)\in\left\{ (3, 12), (4, 16), (5,20) \right\}$. They are shown in Figure \ref{fig:quiv4a} and Figure \ref{fig:quiv4b}.
The corresponding Postnikov diagrams are shown in Figures \ref{fig:quad1}, \ref{fig:quad2} and \ref{fig:quad3}.

\subsection{Cobwebs}

We obtain a new infinite family of self-injective quivers with potential, with arbitrarily large order of $\sigma$.

For any odd integer $x\geq 3$, define a graph $\on{Cob}(x)$ as follows. Set 
\[
  \on{Cob}(x)_0 = \left\{ c_{1}, \dots, c_x \right\} \cup \left\{ d_{st} \right\}_{s = 1, \dots, (x-3)/2}^{ t = 1, \dots, 2x}.
\]
Set 
\begin{align*}
  \on{Cob}(x)_1 &= \left\{ (c_t, c_{t+1}) \right\}_{t= 1, \dots, x} \\
  &\cup \left\{ (d_{st}, d_{s, t+1}) \right\}_{s = 1, \dots, (x-3)/2}^{ t = 1, \dots, 2x}  \\
  &\cup \left\{ (d_{st}, d_{s+1, t}) \right\}_{s = 1, \dots, (x-5)/2}^{ t = 1, \dots, 2x} \\
  &\cup \left\{ (c_t, d_{1,2t-1}), (c_t, d_{1,2t}) \right\}_{t= 1, \dots, x}
\end{align*}
where indices are taken modulo $x$ in the first row and modulo $2x$ in the second row.
This graph has a natural embedding in the plane given by arranging the vertices $c_t$ clockwise in a regular $x$-gon of radius 1, and the vertices $d_{st}$ clockwise in a regular $2x$-gon of radius $s+1$ for every $s$.
This embedding equips $\on{Cob}(x)$ with faces bounded by cycles (one $x$-gon, $x$ triangles and $x^2-4x $ squares). Choosing an orientation of an edge, we can turn $\on{Cob}(x)$ into a quiver by requiring that
all these cycles be cyclically oriented. Call $\on{Cob}^+(x)$ and $\on{Cob}^-(x)$ the quivers one gets by orienting the $x$-gon counterclockwise and clockwise respectively (see Figures \ref{fig:cob1} and \ref{fig:cob2}).
As usual, one can define potentials on these quivers by taking the alternating sum of all cycles bounding faces. 

\begin{prop}
  For every odd $x\geq 3$, there exists a symmetric $(x-1, 2x)$-Postnikov diagram whose associated self-injective quiver with potential is $\on{Cob}^+(x)$. Similarly, there exists a 
  symmetric $(x+1, 2x)$-Postnikov diagram whose associated self-injective quiver with potential is $\on{Cob}^-(x)$.
\end{prop}

\begin{proof}
 We give the construction for the first case, the second case being similar. 
Start by connecting vertex $i$ to $i+x-1$ with a straight strand for every $i$ odd (creating a $x$-pointed star shape). Then for every $i$ even, draw a strand $i\to i+x-1$ as in Figure \ref{fig:cobconst}: cross strand $i-1$, then 
follow strand $i+x$ as close as possible until its start, and cross strand $i+1$ as last crossing.
This construction yields a symmetric Postnikov diagram, and it produces the correct quiver.
\end{proof}
In Figure \ref{fig:post614} we illustrate the case $x = 7$.
The quivers $\on{Cob}^-(x)$ are shown in Figure \ref{fig:cob1} and Figure \ref{fig:cob2} for $x = 5, 7, 9$. The Nakayama permutation acts by rotation by $\pi(x+1)/x$, which has order $x$.

\subsection{Miscellaneous}

We have two more examples of self-injective quivers with potential coming from symmetric Postnikov diagrams, for $(k, n) = (6,15)$ and $(6,21)$.
We show the first one in Figure \ref{fig:quiv615}.
For $(k, n) = (6, 21)$ we get Figure \ref{fig:quiv621}. In Figure \ref{fig:post615} and Figure \ref{fig:post621} we show the corresponding Postnikov diagrams.

\newpage
\section{Figures}\label{sec:figs}

\begin{figure}[H]
  \[
    \begin{tikzpicture}[baseline=(bb.base)]

\newcommand{\goodarrow}{\arrow{angle 60}}
\newcommand{\bstart}{90} %
\newcommand{\nth}{18} 
\newcommand{\radius}{5cm} 
\newcommand{\eps}{11pt} 
\newcommand{\dotrad}{0.07cm} 

\path (0,0) node (bb) {};


\draw (0,0) circle(\radius) [thick,dashed];


 \coordinate (b1) at (\bstart:\radius);
 \coordinate (b2) at (\bstart-\nth:\radius);
 \coordinate (b3) at (\bstart-2*\nth:\radius);
 \coordinate (b4) at (\bstart-3*\nth:\radius);
 \coordinate (b5) at (\bstart-4*\nth:\radius);
 \draw (\bstart:\radius+\eps) node {$1$}; 
 \draw (\bstart-\nth:\radius+\eps) node {$2$}; 
 \draw (\bstart-2*\nth:\radius+\eps) node {$3$}; 
 \draw (\bstart -4*\nth :\radius+\eps) node {$k$};
  
 \coordinate (c1) at (-\bstart:\radius);
 \coordinate (c2) at (-\bstart-\nth:\radius);
 \coordinate (c3) at (-\bstart-2*\nth:\radius);
 \coordinate (c4) at (-\bstart-3*\nth:\radius);
 \coordinate (c5) at (-\bstart-4*\nth:\radius);
 \draw (-\bstart:\radius+\eps) node {$1+k$}; 
 \draw (-\bstart-\nth:\radius+\eps) node {$2+k$}; 
 \draw (-\bstart-2*\nth:\radius+\eps) node {$3+k$}; 
 \draw (-\bstart -4*\nth :\radius+\eps) node {$2k$};
\foreach \n in {1,2,3,5} {\draw (b\n) circle(\dotrad) [fill=black];
\draw (c\n) circle(\dotrad) [fill=black];
}


\draw  plot [smooth]
coordinates {(b1) (0: .7*\radius)(c1)};
\draw  plot [smooth]
coordinates {(c1) (180: .7*\radius)(b1)};
\draw  plot [smooth]
coordinates {(b2) (90: .9*\radius) (180: .6*\radius) (250:.8*\radius)(c2)};
\draw  plot [smooth]
coordinates {(c2) (-90: .9*\radius) (0: .6*\radius) (70:.8*\radius)(b2)};
\draw  plot [smooth]
coordinates {(b3) (90-\nth: .96*\radius) (90: .8*\radius) (180:.5*\radius) (250:.75*\radius)(c3)};
\draw  plot [smooth]
coordinates {(c3) (270-\nth: .96*\radius) (270: .8*\radius) (0:.5*\radius) (70:.75*\radius)(b3)};


\draw  plot [smooth]
coordinates {(b5) (2*\nth: .93*\radius)};
\draw  plot [smooth]
coordinates {(b5) (2*\nth: .88*\radius)};

\draw  plot [smooth]
coordinates {(\bstart-2*\nth:\radius*.9)  (\bstart-\nth:\radius*.9) (90:\radius*.6) (180: \radius*.2) (270-\nth:.7*\radius) (270-\nth -3:.85*\radius) (270-2*\nth:.85*\radius)};

\draw  plot [smooth]
coordinates {(c5) (2*\nth+180: .93*\radius)};
\draw  plot [smooth]
coordinates {(c5) (2*\nth+180: .88*\radius)};

\draw  plot [smooth]
coordinates {(\bstart-2*\nth+180:\radius*.9)  (\bstart-\nth+180:\radius*.9) (90+180:\radius*.6) (0: \radius*.2) (270-\nth+180:.7*\radius) (270-\nth -3+180:.85*\radius) (270-2*\nth+180:.85*\radius)};


\draw [dotted, very thick] plot
coordinates {(180:.4*\radius) (180:.3*\radius)};
\draw [dotted, very thick] plot
coordinates {(0:.4*\radius) (0:.3*\radius)};

\draw[dotted, very thick] plot 
coordinates {(90-2*\nth -6:.88*\radius) (90-2*\nth -12:.88*\radius)};

\draw[dotted, very thick] plot 
coordinates {(270-2*\nth -6:.88*\radius) (-90-2*\nth -12:.88*\radius)};

\end{tikzpicture}
\]
\caption{The construction of a symmetric $(k,2k)$-Postnikov diagram.}
\label{fig:postsq}
\end{figure}
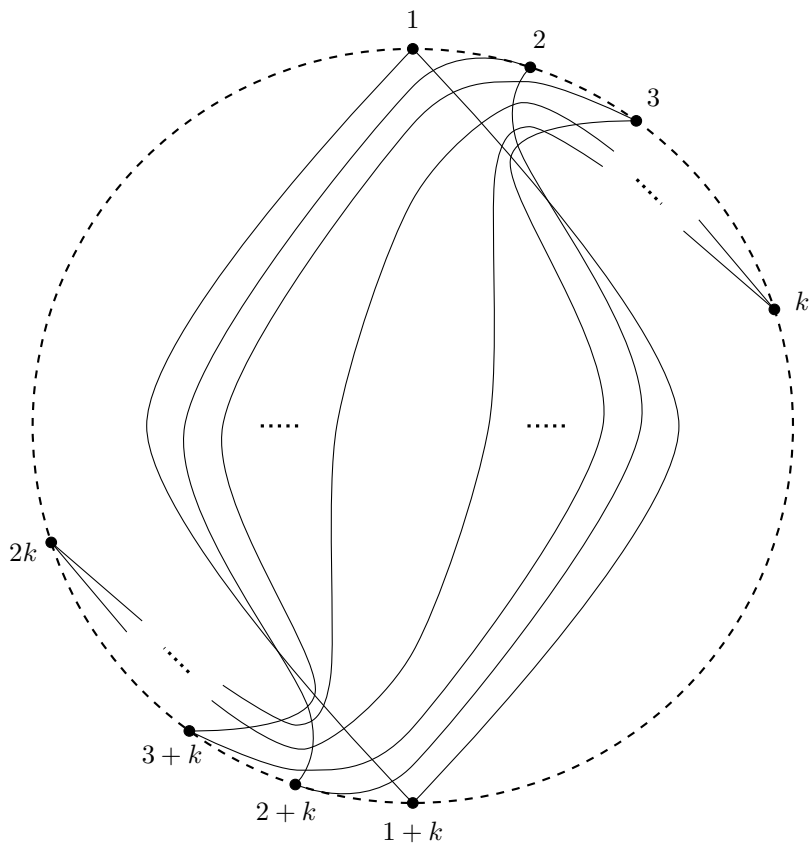

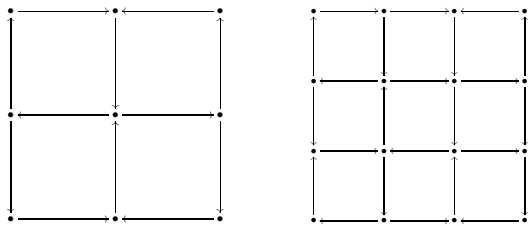
\begin{figure}[H]
\[
  \resizebox{3cm}{!}{
\begin{xy} 0;<4pt,0pt>:<0pt,-4pt>:: 
(0,0) *+{\bullet} ="0",
(20,0) *+{\bullet} ="1",
(40,0) *+{\bullet} ="2",
(0,20) *+{\bullet} ="3",
(20,20) *+{\bullet} ="4",
(40,20) *+{\bullet} ="5",
(0,40) *+{\bullet} ="6",
(20,40) *+{\bullet} ="7",
(40,40) *+{\bullet} ="8",
"0", {\ar"1"},
"3", {\ar"0"},
"2", {\ar"1"},
"1", {\ar"4"},
"5", {\ar"2"},
"4", {\ar"3"},
"3", {\ar"6"},
"4", {\ar"5"},
"7", {\ar"4"},
"5", {\ar"8"},
"6", {\ar"7"},
"8", {\ar"7"},
\end{xy}}
\hspace{1cm}
 \resizebox{3cm}{!}{
\begin{xy} 0;<.8pt,0pt>:<0pt,-.8pt>:: 
(0,0) *+{\bullet} ="0",
(75,0) *+{\bullet} ="1",
(150,0) *+{\bullet} ="2",
(225,0) *+{\bullet} ="3",
(0,75) *+{\bullet} ="4",
(75,75) *+{\bullet} ="5",
(150,75) *+{\bullet} ="6",
(225,75) *+{\bullet} ="7",
(0,150) *+{\bullet} ="8",
(75,150) *+{\bullet} ="9",
(150,150) *+{\bullet} ="10",
(225,150) *+{\bullet} ="11",
(0,225) *+{\bullet} ="12",
(75,225) *+{\bullet} ="13",
(150,225) *+{\bullet} ="14",
(225,225) *+{\bullet} ="15",
"0", {\ar"1"},
"4", {\ar"0"},
"1", {\ar"2"},
"1", {\ar"5"},
"3", {\ar"2"},
"2", {\ar"6"},
"7", {\ar"3"},
"5", {\ar"4"},
"4", {\ar"8"},
"5", {\ar"6"},
"9", {\ar"5"},
"6", {\ar"7"},
"6", {\ar"10"},
"7", {\ar"11"},
"8", {\ar"9"},
"12", {\ar"8"},
"10", {\ar"9"},
"9", {\ar"13"},
"10", {\ar"11"},
"14", {\ar"10"},
"11", {\ar"15"},
"13", {\ar"12"},
"13", {\ar"14"},
"15", {\ar"14"},
\end{xy}}
\]

\caption{Two square grid self-injective quivers with potential (here $n = 2k, k = 4,5$).}
\label{fig:squares}

\end{figure}
\begin{figure}[H]
  \[
  \resizebox{3cm}{!}{\begin{xy} 0;<1.2pt,0pt>:<0pt,-1.2pt>:: 
(55,0) *+{\bullet} ="0",
(0,95) *+{\bullet} ="1",
(110,95) *+{\bullet} ="2",
"1", {\ar"0"},
"0", {\ar"2"},
"2", {\ar"1"},
\end{xy}

  }
  \hspace{1cm}
 \resizebox{3.1cm}{!}{\begin{xy} 0;<.5pt,0pt>:<0pt,-.5pt>:: 
(175,0) *+{\bullet} ="0",
(116,96) *+{\bullet} ="1",
(234,96) *+{\bullet} ="2",
(59,193) *+{\bullet} ="3",
(175,193) *+{\bullet} ="4",
(291,193) *+{\bullet} ="5",
(0,289) *+{\bullet} ="6",
(116,289) *+{\bullet} ="7",
(234,289) *+{\bullet} ="8",
(350,289) *+{\bullet} ="9",
"1", {\ar"0"},
"0", {\ar"2"},
"2", {\ar"1"},
"3", {\ar"1"},
"1", {\ar"4"},
"4", {\ar"2"},
"2", {\ar"5"},
"4", {\ar"3"},
"6", {\ar"3"},
"3", {\ar"7"},
"5", {\ar"4"},
"7", {\ar"4"},
"4", {\ar"8"},
"8", {\ar"5"},
"5", {\ar"9"},
"7", {\ar"6"},
"8", {\ar"7"},
"9", {\ar"8"},
\end{xy}
}
\]
\caption{The quivers of the 3-preprojective algebras of type $A_2$ and $A_4$.}
\label{fig:triangles}
\end{figure}
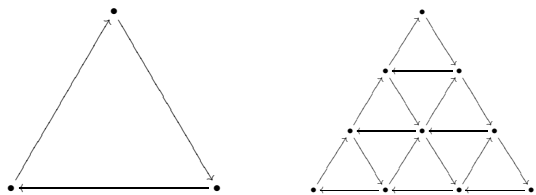

\newpage
\begin{figure}[H]
  \vspace{1cm}
\[
 \resizebox{8cm}{!}{\begin{xy} 0;<1.2pt,0pt>:<0pt,-1.2pt>:: 
(175,0) *+{\bullet} ="0",
(140,58) *+{\bullet} ="1",
(210,58) *+{\bullet} ="2",
(105,117) *+{\bullet} ="3",
(175,117) *+{\bullet} ="4",
(245,117) *+{\bullet} ="5",
(70,174) *+{\bullet} ="6",
(140,174) *+{\bullet} ="7",
(210,174) *+{\bullet} ="8",
(280,174) *+{\bullet} ="9",
(35,233) *+{\bullet} ="10",
(105,233) *+{\bullet} ="11",
(175,233) *+{\bullet} ="12",
(245,233) *+{\bullet} ="13",
(315,233) *+{\bullet} ="14",
(0,291) *+{\bullet} ="15",
(70,291) *+{\bullet} ="16",
(140,291) *+{\bullet} ="17",
(210,291) *+{\bullet} ="18",
(280,291) *+{\bullet} ="19",
(350,291) *+{\bullet} ="20",
"1", {\ar"0"},
"0", {\ar"2"},
"2", {\ar"1"},
"3", {\ar"1"},
"1", {\ar"4"},
"4", {\ar"2"},
"2", {\ar"5"},
"4", {\ar"3"},
"6", {\ar"3"},
"3", {\ar"7"},
"5", {\ar"4"},
"7", {\ar"4"},
"4", {\ar"8"},
"8", {\ar"5"},
"5", {\ar"9"},
"7", {\ar"6"},
"10", {\ar"6"},
"6", {\ar"11"},
"8", {\ar"7"},
"11", {\ar"7"},
"7", {\ar"12"},
"9", {\ar"8"},
"12", {\ar"8"},
"8", {\ar"13"},
"13", {\ar"9"},
"9", {\ar"14"},
"11", {\ar"10"},
"15", {\ar"10"},
"10", {\ar"16"},
"12", {\ar"11"},
"16", {\ar"11"},
"11", {\ar"17"},
"13", {\ar"12"},
"17", {\ar"12"},
"12", {\ar"18"},
"14", {\ar"13"},
"18", {\ar"13"},
"13", {\ar"19"},
"19", {\ar"14"},
"14", {\ar"20"},
"16", {\ar"15"},
"17", {\ar"16"},
"18", {\ar"17"},
"19", {\ar"18"},
"20", {\ar"19"},
\end{xy}
}
\]
\caption{The quiver of the 3-preprojective algebra of type $A_6$.}
\label{fig:triangles2}
\end{figure}
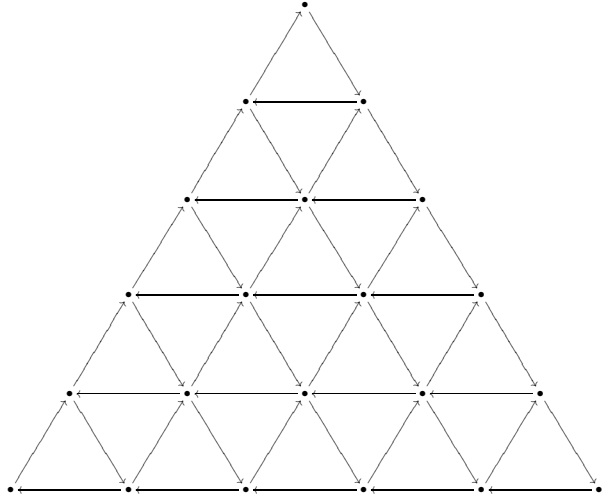

\begin{figure}[H]
  \[
    \begin{tikzpicture}[baseline=(bb.base)]

\newcommand{\strarrow}{\arrow{angle 60}}
\newcommand{\bstart}{150} %
\newcommand{\nth}{60} 
\newcommand{\radius}{4.5cm} 
\newcommand{\eps}{11pt} 
\newcommand{\dotrad}{0.07cm} 

\path (0,0) node (bb){};

\draw (0,0) circle(\radius) [thick, dashed];
\foreach \n in {1,...,6}
{ \coordinate (b\n) at (\bstart-\nth*\n:\radius);
  \draw (\bstart-\nth*\n:\radius+\eps) node {$\n$}; }
  

\foreach \n in {1,...,6} {\draw (b\n) circle(\dotrad) [fill=black];}

\foreach\n in {1, 3,5}
{ 
  \coordinate (c\n) at (\bstart-\nth*\n-2*\nth:\radius);
  \draw plot coordinates {(b\n) (c\n)};
}
\foreach\n in {2, 4,6}
{ 
  \coordinate (c\n) at (\bstart-\nth*\n-2*\nth:\radius);
  \draw plot [smooth] coordinates {(b\n) (\bstart-\nth*\n-4*\nth:\radius*.2)(c\n)};
}

    \end{tikzpicture}
  \]
\caption{The $(2,6)$-Postnikov diagram corresponding to the 3-preprojective algebra of type $A_2$.}
  \label{fig:tri1}
\end{figure}
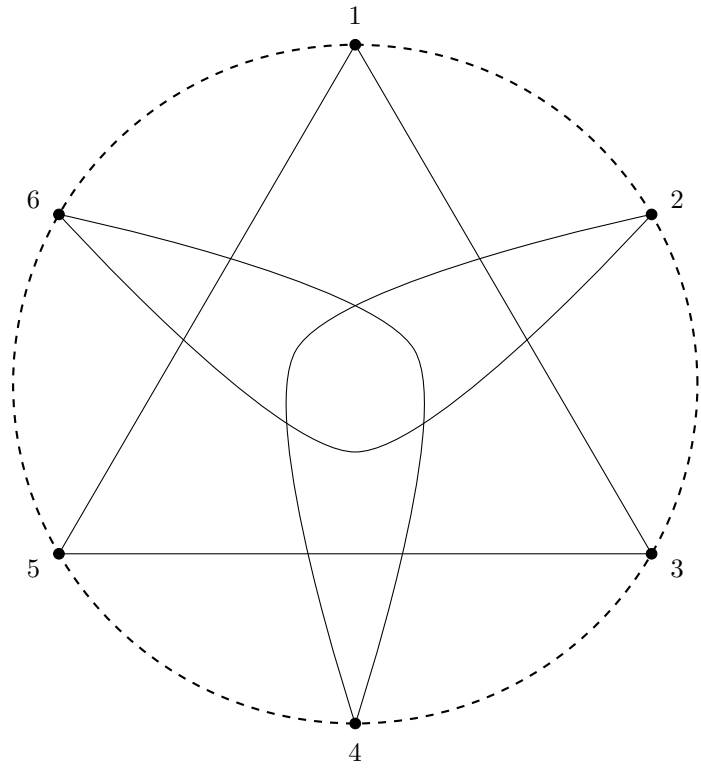

\newpage
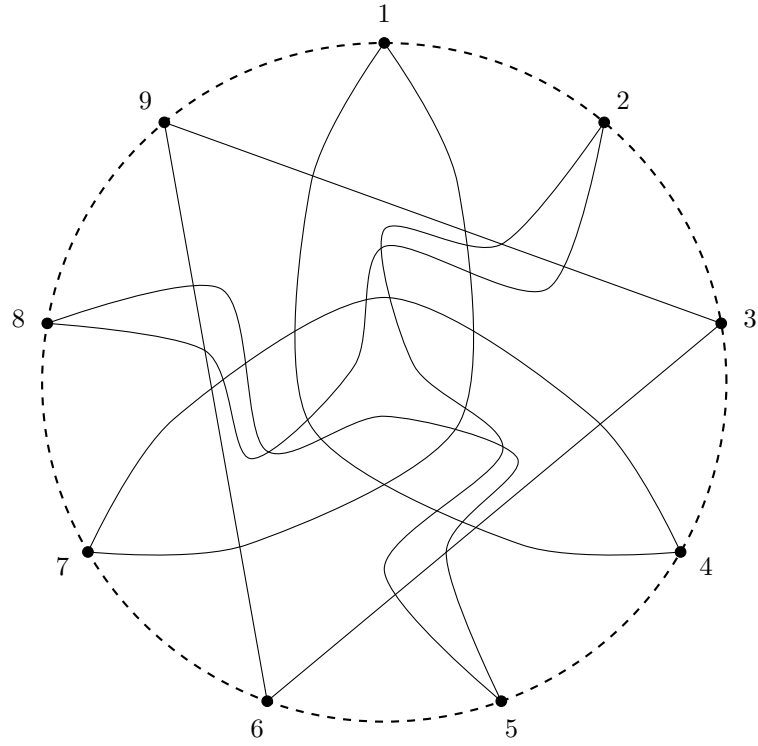
\begin{figure}[H]
\[
\begin{tikzpicture}[baseline=(bb.base), quivarrow/.style={black, -latex, very thick}
]

\newcommand{\goodarrow}{\arrow{angle 60}}
\newcommand{\bstart}{130} 
\newcommand{\nth}{40} 
\newcommand{\qstart}{150} 
\newcommand{\radius}{4.5cm} 
\newcommand{\eps}{11pt} 
\newcommand{\dotrad}{0.07cm} 

\path (0,0) node (bb) {};


\draw (0,0) circle(\radius) [thick,dashed];

\foreach \n in {1,...,9}
{ \coordinate (b\n) at (\bstart-\nth*\n:\radius);
  \draw (\bstart-\nth*\n:\radius+\eps) node {$\n$}; }
  

\foreach \n in {1,2,3,4,5,6,7,8,9} {\draw (b\n) circle(\dotrad) [fill=black];}


\draw  plot[smooth]
coordinates {(b9) (b3)}
;
\draw  plot[smooth]
coordinates {(b3) (b6)}
;
\draw  plot[smooth]
coordinates {(b6) (b9)}
;
\draw  plot[smooth] coordinates {(b7)(230:\radius*25/40) (-30:.25*\radius) (70:\radius*25/40) (b1)}
;
\draw  plot[smooth] coordinates {(b1) (110:\radius*25/40)(-150:\radius*.25) (-50:\radius*25/40)(b4)}
;
\draw  plot[smooth] coordinates {(b4)(-10:\radius*25/40) (90:\radius*.25) (190:\radius*25/40)(b7)}
;

\foreach \n in {2,5,8}
{
\coordinate (c\n) at (\bstart-\nth*\n-3*\nth:\radius);
\draw  plot[smooth] coordinates {(b\n) (\bstart-\nth*\n- 0*\nth:\radius*.53) ((\bstart-\nth*\n+1*\nth:\radius*.45)((\bstart-\nth*\n-.5*\nth:\radius*.1)
(\bstart-\nth*\n-2*\nth:\radius*.4)(\bstart-\nth*\n-3.5*\nth:\radius*.55)(c\n)
}
;
}

 \end{tikzpicture}
\]
\caption{The $(3,9)$-Postnikov diagram corresponding to the 3-preprojective algebra of type $A_4$.}
\label{fig:tri2}
\end{figure}

\begin{figure}[H]
  \[
    \begin{tikzpicture}[baseline=(bb.base)]

\newcommand{\strarrow}{\arrow{angle 60}}
\newcommand{\bstart}{120} %
\newcommand{\nth}{30} 
\newcommand{\radius}{4.5cm} 
\newcommand{\eps}{11pt} 
\newcommand{\dotrad}{0.07cm} 

\path (0,0) node (bb){};

\draw (0,0) circle(\radius) [thick, dashed];
\foreach \n in {1,...,12}
{ \coordinate (b\n) at (\bstart-\nth*\n:\radius);
  \draw (\bstart-\nth*\n:\radius+\eps) node {$\n$}; }
  

\foreach \n in {1,...,12} {\draw (b\n) circle(\dotrad) [fill=black];}

\foreach\n in {1,5,9}
{ 
  \coordinate (c\n) at (\bstart-\nth*\n-4*\nth:\radius);
  \draw plot coordinates {(b\n) (c\n)};
}

\foreach\n in {4, 8,12}
{
  \coordinate (c\n) at (\bstart-\nth*\n-4*\nth:\radius);
  \draw plot [smooth]
  coordinates {(b\n) (\bstart-\nth*\n+.2*\nth:\radius*.5)
(\bstart-\nth*\n-4.2*\nth:\radius*.5)(c\n)
};
}
\foreach\n in {2,6,10}
{
  \coordinate (c\n) at (\bstart-\nth*\n-4*\nth:\radius);
  \draw plot [smooth]
  coordinates {(b\n) (\bstart-\nth*\n+3.8*\nth:\radius*.33)
  (\bstart-\nth*\n-4*\nth:\radius*.2)
(c\n)
};
}
\foreach\n in {3,7,11}
{
  \coordinate (c\n) at (\bstart-\nth*\n-4*\nth:\radius);
  \draw plot [smooth]
  coordinates {(b\n) 
  (\bstart-\nth*\n-0*\nth:\radius*.4)
  (\bstart-\nth*\n-3*\nth:\radius*.03)
    (\bstart-\nth*\n+6.1*\nth:\radius*.5)
  (\bstart-\nth*\n-4.5*\nth:\radius*.5)
(c\n)
};
}

    \end{tikzpicture}
  \]
\caption{The $(4,12)$-Postnikov diagram corresponding to the 3-preprojective algebra of type $A_6$.}
\label{fig:tri3}
\end{figure}
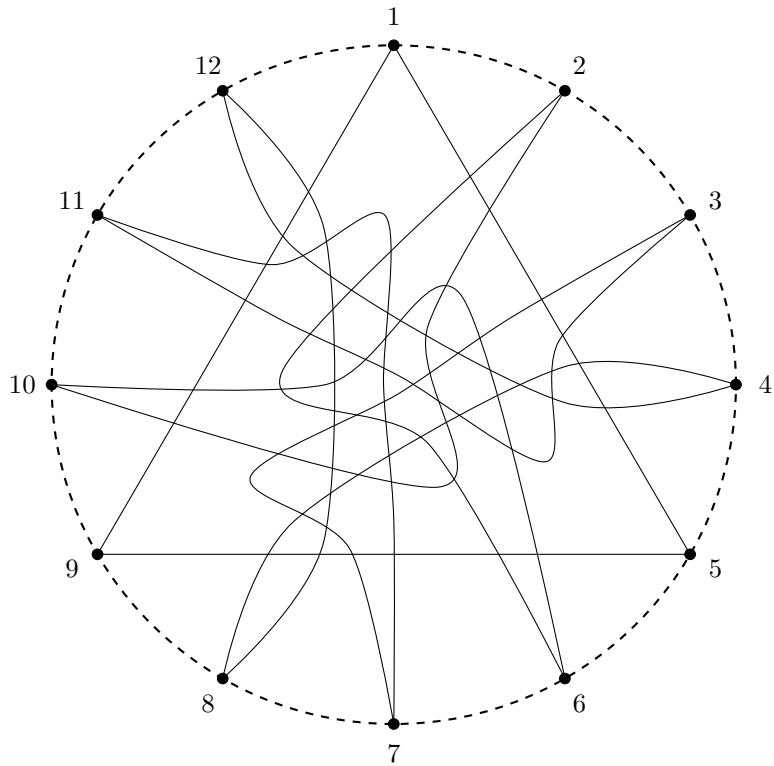

\newpage

\begin{figure}[H] 
  \vspace{1cm}
  \[
\resizebox{7cm}{!}{
  \begin{xy} 0;<1.3pt,0pt>:<0pt,-1.3pt>:: 
(150,0) *+{\bullet} ="0",
(100,50) *+{\bullet} ="1",
(50,100) *+{\bullet} ="2",
(0,150) *+{\bullet} ="3",
(50,200) *+{\bullet} ="4",
(100,250) *+{\bullet} ="5",
(150,300) *+{\bullet} ="6",
(200,250) *+{\bullet} ="7",
(250,200) *+{\bullet} ="8",
(300,150) *+{\bullet} ="9",
(250,100) *+{\bullet} ="10",
(200,50) *+{\bullet} ="11",
(100,100) *+{\bullet} ="12",
(100,200) *+{\bullet} ="13",
(200,200) *+{\bullet} ="14",
(200,100) *+{\bullet} ="15",
"1", {\ar"0"},
"0", {\ar"11"},
"2", {\ar"1"},
"1", {\ar"12"},
"3", {\ar"2"},
"12", {\ar"2"},
"4", {\ar"3"},
"5", {\ar"4"},
"4", {\ar"13"},
"6", {\ar"5"},
"13", {\ar"5"},
"7", {\ar"6"},
"8", {\ar"7"},
"7", {\ar"14"},
"9", {\ar"8"},
"14", {\ar"8"},
"10", {\ar"9"},
"11", {\ar"10"},
"10", {\ar"15"},
"15", {\ar"11"},
"13", {\ar"12"},
"12", {\ar"15"},
"14", {\ar"13"},
"15", {\ar"14"},
"2", {\ar"4"},
"11", {\ar"1"},
"5",{\ar"7"},
"8",{\ar"10"},

\end{xy} 
}
\hspace{1cm}
\resizebox{7cm}{!}{
 \begin{xy} 0;<.5pt,0pt>:<0pt,-.5pt>:: 
(360,360) *+{\bullet} ="0",
(360,180) *+{\bullet} ="1",
(540,360) *+{\bullet} ="2",
(360,540) *+{\bullet} ="3",
(180,360) *+{\bullet} ="4",
(270,270) *+{\bullet} ="5",
(450,270) *+{\bullet} ="6",
(450,450) *+{\bullet} ="7",
(270,450) *+{\bullet} ="8",
(450,90) *+{\bullet} ="9",
(540,180) *+{\bullet} ="10",
(630,450) *+{\bullet} ="11",
(540,540) *+{\bullet} ="12",
(270,630) *+{\bullet} ="13",
(180,540) *+{\bullet} ="14",
(90,270) *+{\bullet} ="15",
(180,180) *+{\bullet} ="16",
(360,90) *+{\bullet} ="17",
(180,90) *+{\bullet} ="18",
(90,180) *+{\bullet} ="19",
(270,0) *+{\bullet} ="20",
(630,180) *+{\bullet} ="21",
(630,360) *+{\bullet} ="22",
(540,90) *+{\bullet} ="23",
(720,270) *+{\bullet} ="24",
(360,630) *+{\bullet} ="25",
(540,630) *+{\bullet} ="26",
(630,540) *+{\bullet} ="27",
(450,720) *+{\bullet} ="28",
(90,360) *+{\bullet} ="29",
(90,540) *+{\bullet} ="30",
(180,630) *+{\bullet} ="31",
(0,450) *+{\bullet} ="32",
"0", {\ar"1"},
"0", {\ar"2"},
"0", {\ar"3"},
"0", {\ar"4"},
"5", {\ar"0"},
"6", {\ar"0"},
"7", {\ar"0"},
"8", {\ar"0"},
"1", {\ar"5"},
"1", {\ar"6"},
"9", {\ar"1"},
"16", {\ar"1"},
"1", {\ar"17"},
"2", {\ar"6"},
"2", {\ar"7"},
"10", {\ar"2"},
"11", {\ar"2"},
"2", {\ar"22"},
"3", {\ar"7"},
"3", {\ar"8"},
"12", {\ar"3"},
"13", {\ar"3"},
"3", {\ar"25"},
"4", {\ar"5"},
"4", {\ar"8"},
"14", {\ar"4"},
"15", {\ar"4"},
"4", {\ar"29"},
"5", {\ar"16"},
"6", {\ar"10"},
"7", {\ar"12"},
"8", {\ar"14"},
"10", {\ar"9"},
"17", {\ar"9"},
"9", {\ar"23"},
"21", {\ar"10"},
"12", {\ar"11"},
"22", {\ar"11"},
"11", {\ar"27"},
"26", {\ar"12"},
"14", {\ar"13"},
"25", {\ar"13"},
"13", {\ar"31"},
"30", {\ar"14"},
"16", {\ar"15"},
"15", {\ar"19"},
"29", {\ar"15"},
"18", {\ar"16"},
"17", {\ar"18"},
"20", {\ar"17"},
"19", {\ar"18"},
"18", {\ar"20"},
"22", {\ar"21"},
"23", {\ar"21"},
"21", {\ar"24"},
"24", {\ar"22"},
"25", {\ar"26"},
"28", {\ar"25"},
"27", {\ar"26"},
"26", {\ar"28"},
"29", {\ar"30"},
"32", {\ar"29"},
"31", {\ar"30"},
"30", {\ar"32"},
\end{xy}
}
\]
\caption{Self-injective quivers with potential for $(k,n) = (3,12)$ and $(k,n) = (4,16)$.}
\label{fig:quiv4a}
\end{figure}
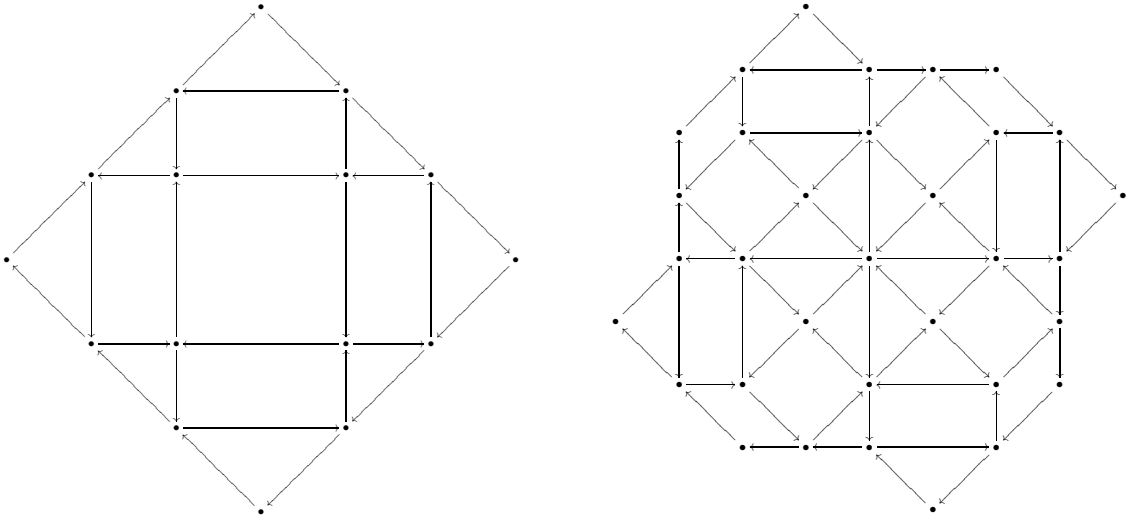

\begin{figure}[H]
  \vspace{1cm}
  \[
\resizebox{9cm}{!}{

\begin{xy} 0;<1.5pt,0pt>:<0pt,-1.5pt>::
(128,192) *+{\bullet} ="0",
(128,160) *+{\bullet} ="1",
(160,128) *+{\bullet} ="2",
(192,128) *+{\bullet} ="3",
(224,160) *+{\bullet} ="4",
(224,192) *+{\bullet} ="5",
(192,224) *+{\bullet} ="6",
(160,224) *+{\bullet} ="7",
(112,112) *+{\bullet} ="8",
(160,96) *+{\bullet} ="9",
(128,64) *+{\bullet} ="10",
(96,64) *+{\bullet} ="11",
(64,96) *+{\bullet} ="12",
(64,128) *+{\bullet} ="13",
(96,160) *+{\bullet} ="14",
(192,96) *+{\bullet} ="15",
(240,112) *+{\bullet} ="16",
(224,64) *+{\bullet} ="17",
(256,64) *+{\bullet} ="18",
(288,96) *+{\bullet} ="19",
(288,128) *+{\bullet} ="20",
(256,160) *+{\bullet} ="21",
(256,192) *+{\bullet} ="22",
(240,240) *+{\bullet} ="23",
(288,224) *+{\bullet} ="24",
(288,256) *+{\bullet} ="25",
(256,288) *+{\bullet} ="26",
(224,288) *+{\bullet} ="27",
(192,256) *+{\bullet} ="28",
(112,240) *+{\bullet} ="29",
(160,256) *+{\bullet} ="30",
(128,288) *+{\bullet} ="31",
(96,288) *+{\bullet} ="32",
(64,256) *+{\bullet} ="33",
(64,224) *+{\bullet} ="34",
(96,192) *+{\bullet} ="35",
(32,96) *+{\bullet} ="36",
(256,32) *+{\bullet} ="37",
(320,256) *+{\bullet} ="38",
(96,320) *+{\bullet} ="39",
(128,32) *+{\bullet} ="40",
(160,32) *+{\bullet} ="41",
(192,32) *+{\bullet} ="42",
(224,0) *+{\bullet} ="43",
(320,128) *+{\bullet} ="44",
(320,160) *+{\bullet} ="45",
(320,192) *+{\bullet} ="46",
(352,224) *+{\bullet} ="47",
(224,320) *+{\bullet} ="48",
(192,320) *+{\bullet} ="49",
(160,320) *+{\bullet} ="50",
(128,352) *+{\bullet} ="51",
(32,224) *+{\bullet} ="52",
(32,192) *+{\bullet} ="53",
(32,160) *+{\bullet} ="54",
(0,128) *+{\bullet} ="55",
"0", {\ar"1"},
"7", {\ar"0"},
"0", {\ar"29"},
"35", {\ar"0"},
"1", {\ar"2"},
"8", {\ar"1"},
"1", {\ar"14"},
"2", {\ar"3"},
"2", {\ar"8"},
"9", {\ar"2"},
"3", {\ar"4"},
"3", {\ar"15"},
"16", {\ar"3"},
"4", {\ar"5"},
"4", {\ar"16"},
"21", {\ar"4"},
"5", {\ar"6"},
"5", {\ar"22"},
"23", {\ar"5"},
"6", {\ar"7"},
"6", {\ar"23"},
"28", {\ar"6"},
"29", {\ar"7"},
"7", {\ar"30"},
"8", {\ar"9"},
"10", {\ar"8"},
"8", {\ar"11"},
"12", {\ar"8"},
"8", {\ar"13"},
"14", {\ar"8"},
"9", {\ar"10"},
"15", {\ar"9"},
"11", {\ar"10"},
"10", {\ar"41"},
"11", {\ar"12"},
"40", {\ar"11"},
"36", {\ar"12"},
"13", {\ar"14"},
"13", {\ar"36"},
"54", {\ar"13"},
"14", {\ar"35"},
"15", {\ar"16"},
"17", {\ar"15"},
"16", {\ar"17"},
"18", {\ar"16"},
"16", {\ar"19"},
"20", {\ar"16"},
"16", {\ar"21"},
"17", {\ar"37"},
"42", {\ar"17"},
"19", {\ar"18"},
"37", {\ar"18"},
"19", {\ar"20"},
"44", {\ar"19"},
"21", {\ar"20"},
"20", {\ar"45"},
"22", {\ar"21"},
"22", {\ar"23"},
"24", {\ar"22"},
"23", {\ar"24"},
"25", {\ar"23"},
"23", {\ar"26"},
"27", {\ar"23"},
"23", {\ar"28"},
"24", {\ar"38"},
"46", {\ar"24"},
"26", {\ar"25"},
"38", {\ar"25"},
"26", {\ar"27"},
"48", {\ar"26"},
"28", {\ar"27"},
"27", {\ar"49"},
"30", {\ar"28"},
"30", {\ar"29"},
"29", {\ar"31"},
"32", {\ar"29"},
"29", {\ar"33"},
"34", {\ar"29"},
"29", {\ar"35"},
"31", {\ar"30"},
"31", {\ar"39"},
"50", {\ar"31"},
"33", {\ar"32"},
"39", {\ar"32"},
"33", {\ar"34"},
"52", {\ar"33"},
"35", {\ar"34"},
"34", {\ar"53"},
"36", {\ar"55"},
"37", {\ar"43"},
"38", {\ar"47"},
"39", {\ar"51"},
"41", {\ar"40"},
"41", {\ar"42"},
"43", {\ar"42"},
"45", {\ar"44"},
"45", {\ar"46"},
"47", {\ar"46"},
"49", {\ar"48"},
"49", {\ar"50"},
"51", {\ar"50"},
"53", {\ar"52"},
"53", {\ar"54"},
"55", {\ar"54"},
\end{xy}
}
\]
\caption{A self-injective quiver with potential for $(k,n) = (5,20)$.}
\label{fig:quiv4b}
\end{figure}
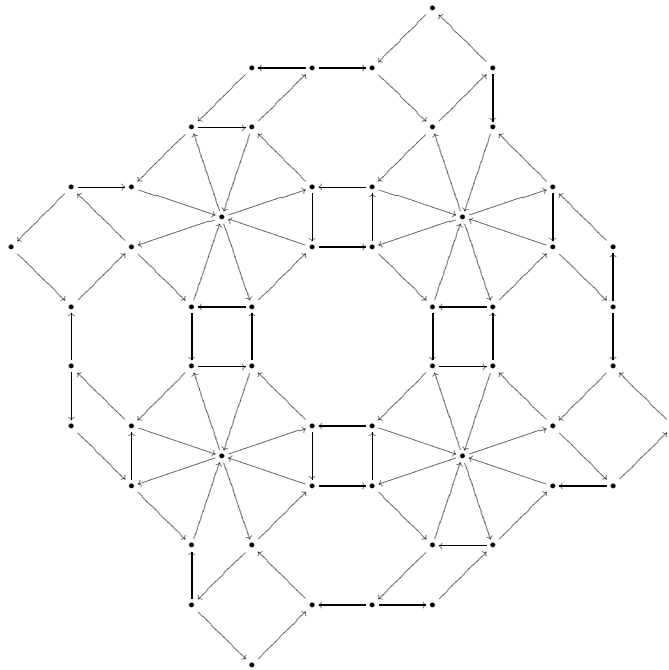

\newpage
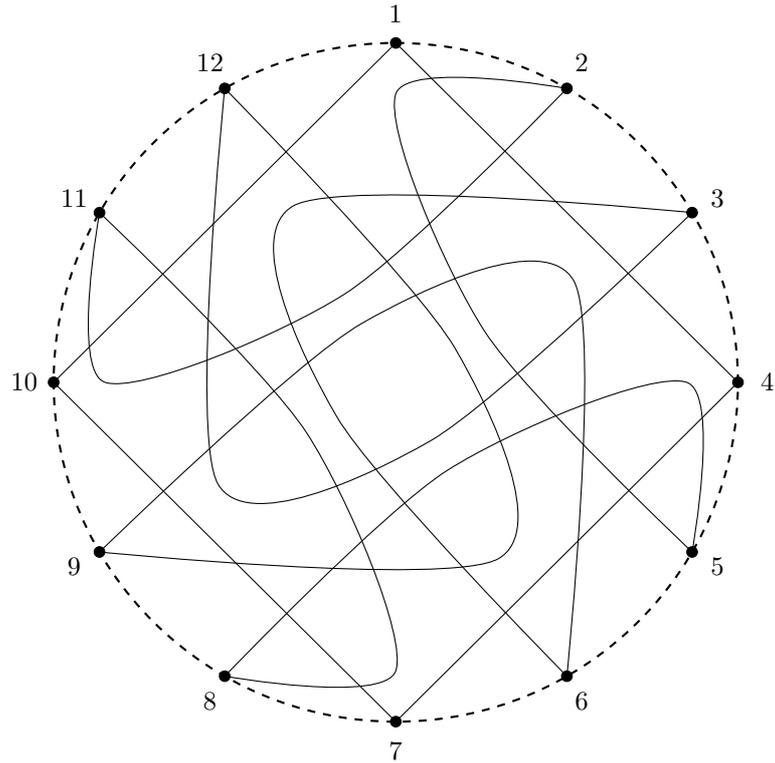
\begin{figure}[H]
  \[
    \begin{tikzpicture}[baseline=(bb.base)]

\newcommand{\strarrow}{\arrow{angle 60}}
\newcommand{\bstart}{120} %
\newcommand{\nth}{30} 
\newcommand{\radius}{4.5cm} 
\newcommand{\eps}{11pt} 
\newcommand{\dotrad}{0.07cm} 

\path (0,0) node (bb){};

\draw (0,0) circle(\radius) [thick, dashed];
\foreach \n in {1,...,12}
{ \coordinate (b\n) at (\bstart-\nth*\n:\radius);
  \draw (\bstart-\nth*\n:\radius+\eps) node {$\n$}; }
  

\foreach \n in {1,...,12} {\draw (b\n) circle(\dotrad) [fill=black];}

\foreach\n in {1, 4, 7, 10}
{ 
  \coordinate (c\n) at (\bstart-\nth*\n-3*\nth:\radius);
  \draw plot coordinates {(b\n) (c\n)};
}

\foreach\n in {2, 5, 8, 11}
{
  \coordinate (c\n) at (\bstart-\nth*\n-3*\nth:\radius);
  \draw plot [smooth]
  coordinates {(b\n) (\bstart-\nth*\n+\nth:\radius*.85)(\bstart-\nth*\n-\nth:\radius*.3)(c\n)};
}
\foreach\n in {3, 6, 9, 12}
{
  \coordinate (c\n) at (\bstart-\nth*\n-3*\nth:\radius);
  \draw plot [smooth]
  coordinates {(b\n) (\bstart-\nth*\n+3*\nth:\radius*.6)(\bstart-\nth*\n-6*\nth:\radius*.2)(c\n)};
}

    \end{tikzpicture}
  \]
  \caption{The symmetric $(3, 12)$-Postnikov diagram corresponding to the left quiver of Figure \ref{fig:quiv4a}.}
  \label{fig:quad1}
\end{figure}

\begin{figure}[H]
  \[
    \begin{tikzpicture}[baseline=(bb.base)]

\newcommand{\strarrow}{\arrow{angle 60}}
\newcommand{\bstart}{112.5} %
\newcommand{\nth}{22.5} 
\newcommand{\radius}{4.5cm} 
\newcommand{\eps}{11pt} 
\newcommand{\dotrad}{0.07cm} 

\path (0,0) node (bb){};

\draw (0,0) circle(\radius) [thick, dashed];
\foreach \n in {1,...,16}
{ \coordinate (b\n) at (\bstart-\nth*\n:\radius);
  \draw (\bstart-\nth*\n:\radius+\eps) node {$\n$}; }
  

\foreach \n in {1,...,16} {\draw (b\n) circle(\dotrad) [fill=black];}

\foreach\n in {3,7,11,15}
{ 
  \coordinate (c\n) at (\bstart-\nth*\n-4*\nth:\radius);
  \draw plot coordinates {(b\n) (c\n)};
}

\foreach\n in {4, 8, 12, 16}
{
  \coordinate (c\n) at (\bstart-\nth*\n-4*\nth:\radius);
  \draw plot [smooth]
  coordinates {(b\n) (\bstart-\nth*\n+1*\nth:\radius*.7) (\bstart-\nth*\n-1*\nth:.4*\radius)(c\n)};
}

\foreach\n in {1, 5, 9, 13}
{
  \coordinate (c\n) at (\bstart-\nth*\n-4*\nth:\radius);
  \draw plot [smooth]
  coordinates {(b\n)(\bstart-\nth*\n+2*\nth:\radius*.65)(\bstart-\nth*\n+4*\nth:\radius*.5) (\bstart-\nth*\n+8*\nth:\radius*.1)(c\n)
  };
}
\foreach\n in {2, 6, 10, 14}
{
  \coordinate (c\n) at (\bstart-\nth*\n-4*\nth:\radius);
  \draw plot [smooth]
  coordinates {(b\n)(\bstart-\nth*\n+.5*\nth:\radius*.7)(\bstart-\nth*\n+2*\nth:\radius*.65)(\bstart-\nth*\n+4.5*\nth:\radius*.3)
    (\bstart-\nth*\n+3*\nth:\radius*.1)(\bstart-\nth*\n+1*\nth:\radius*.2)(\bstart-\nth*\n-.5*\nth:\radius*.28)(\bstart-\nth*\n-2*\nth:\radius*.5)
    (\bstart-\nth*\n-3*\nth:\radius*.55)(\bstart-\nth*\n-4*\nth:\radius*.7)(c\n)
  };
}

    \end{tikzpicture}
  \] 
  \caption{The symmetric $(4, 16)$-Postnikov diagram corresponding to the right quiver of Figure \ref{fig:quiv4a}.}
  \label{fig:quad2}
\end{figure}
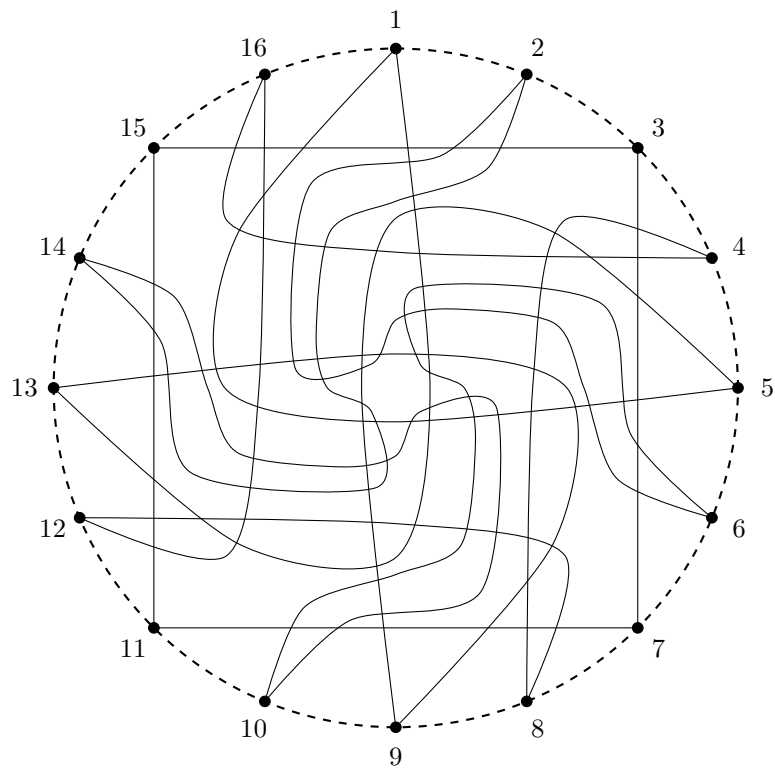

\newpage
\begin{figure}[H]
  \[
    \begin{tikzpicture}[baseline=(bb.base)]

\newcommand{\strarrow}{\arrow{angle 60}}
\newcommand{\bstart}{108} %
\newcommand{\nth}{18} 
\newcommand{\radius}{4.5cm} 
\newcommand{\eps}{11pt} 
\newcommand{\dotrad}{0.07cm} 

\path (0,0) node (bb){};

\draw (0,0) circle(\radius) [thick, dashed];
\foreach \n in {1,...,20}
{ \coordinate (b\n) at (\bstart-\nth*\n:\radius);
  \draw (\bstart-\nth*\n:\radius+\eps) node {$\n$}; }
  

\foreach \n in {1,...,20} {\draw (b\n) circle(\dotrad) [fill=black];}

\foreach\n in {1, 6, 11, 16, 3, 8, 13, 18}
{ 
  \coordinate (c\n) at (\bstart-\nth*\n-5*\nth:\radius);
  \draw plot coordinates {(b\n) (c\n)};
}

\foreach\n in {2, 7, 12, 17}
{
  \coordinate (c\n) at (\bstart-\nth*\n-5*\nth:\radius);
  \draw plot [smooth]
  coordinates {(b\n) (\bstart-\nth*\n+\nth:\radius*.85) (\bstart-\nth*\n-\nth:\radius*.2) (\bstart-\nth*\n-4.5*\nth:.8*\radius)(c\n)};
}
\foreach\n in {5, 10, 15, 20}
{
  \coordinate (c\n) at (\bstart-\nth*\n-5*\nth:\radius);
  \draw plot [smooth]
  coordinates {(b\n) (\bstart-\nth*\n+2*\nth:\radius*.8) (\bstart-\nth*\n-\nth:\radius*.2) (\bstart-\nth*\n-3.5*\nth:.75*\radius)(c\n)};
}

\foreach\n in {4, 9, 14, 19}
{
  \coordinate (c\n) at (\bstart-\nth*\n-5*\nth:\radius);
  \draw plot [smooth]
  coordinates {(b\n)(\bstart-\nth*\n+1*\nth:\radius*.95) (\bstart-\nth*\n+2*\nth:\radius*.85) (\bstart-\nth*\n+3*\nth:\radius*.78) (\bstart-\nth*\n+4*\nth:\radius*.7) (\bstart-\nth*\n+7*\nth:\radius*.55)
  (\bstart-\nth*\n+8*\nth:\radius*.3) (\bstart-\nth*\n+12*\nth:\radius*.4)(\bstart-\nth*\n+13*\nth:\radius*.6) (\bstart-\nth*\n+14*\nth:\radius*.85)
  (\bstart-\nth*\n+15*\nth:\radius*.75) (\bstart-\nth*\n+16*\nth:\radius*.9) (c\n)
  };
}

    \end{tikzpicture}
  \]
  \caption{The symmetric $(5, 20)$-Postnikov diagram corresponding to the quiver of Figure \ref{fig:quiv4b}.}
  \label{fig:quad3}

\end{figure}
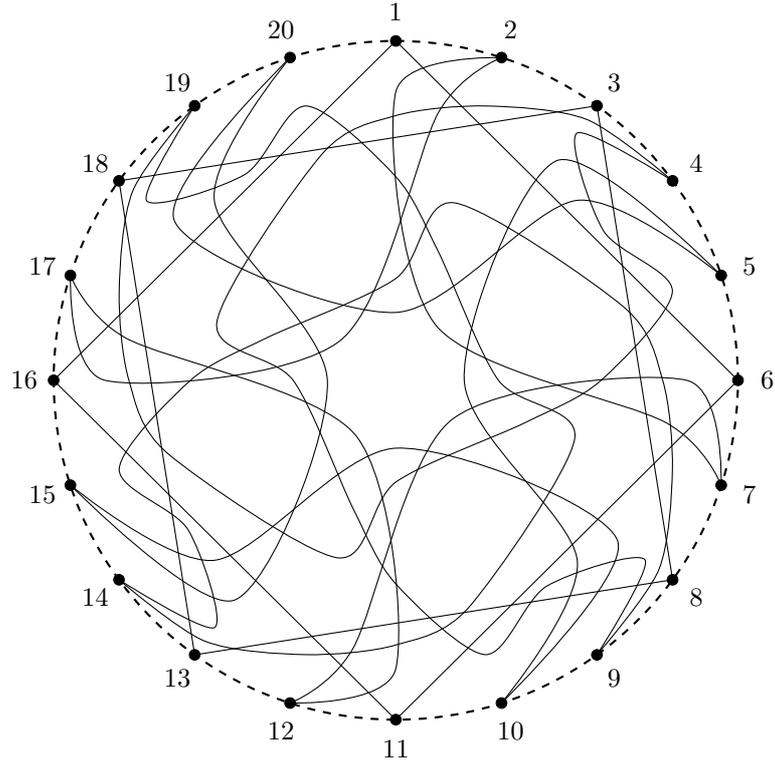

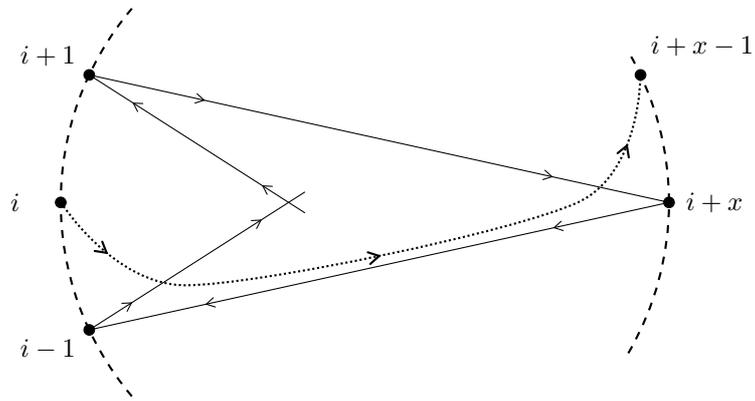
\begin{figure}[H]
  \vspace{1cm}
  \[
    \begin{tikzpicture}[baseline=(bb.base)]
 \newcommand{\goodarrow}{\arrow{angle 60}}
 \newcommand{\radius}{4cm} 
\newcommand{\eps}{17pt} 
\newcommand{\dotrad}{0.07cm} 

\path (0,0) node (bb) {};


\draw (220:\radius) arc (220:140:\radius) [thick, dashed];
\draw (-30:\radius) arc (-30:30:\radius) [thick, dashed];


\coordinate (a1) at (205:\radius);
\coordinate (a2) at (180:\radius);
\coordinate (a3) at (155:\radius);
\coordinate (a4) at (0:\radius);
\coordinate (a5) at (25:\radius);
\coordinate (b1) at (170:.2*\radius);

\foreach \n in {1,...,5} {\draw (a\n) circle(\dotrad) [fill=black];}
\draw (205:\radius+\eps) node {$i-1$};
  \draw (180:\radius+\eps) node {$i$};
  \draw (155:\radius+\eps) node {$i+1$};
  \draw (0:\radius+\eps) node {$i+x$};
  \draw (25:\radius+1.5*\eps) node {$i+x-1$};


  \draw plot [smooth]
  coordinates {(a4) (a1)}
  [ postaction=decorate, decoration={markings,
  mark= at position 0.2 with \goodarrow, mark = at position 0.8 with \goodarrow}];
  \draw plot [smooth]
  coordinates {(a3) (a4)}
  [ postaction=decorate, decoration={markings,
  mark= at position 0.2 with \goodarrow, mark = at position 0.8 with \goodarrow}];

  \draw plot 
  coordinates {(a1) (b1)}
  [ postaction=decorate, decoration={markings,
  mark= at position 0.2 with \goodarrow, mark = at position 0.8 with \goodarrow}];
  \draw plot 
  coordinates {(190:.2*\radius) (a3)}
  [ postaction=decorate, decoration={markings,
  mark= at position 0.2 with \goodarrow, mark = at position 0.8 with \goodarrow}]
  ;

  \draw[thick, densely dotted] plot [smooth]
  coordinates {(a2) (205:.65*\radius) (0: .7*\radius)(a5)}
[ postaction=decorate, decoration={markings,
mark= at position 0.1 with \goodarrow, mark = at position .5 with \goodarrow , mark = at position 0.9 with \goodarrow}]
  ;

 \end{tikzpicture}
\]
\caption{How to draw strand $i$ for $i$ even in $(x-1, 2x)$-Postnikov diagrams.}
\label{fig:cobconst}
\end{figure}

\newpage
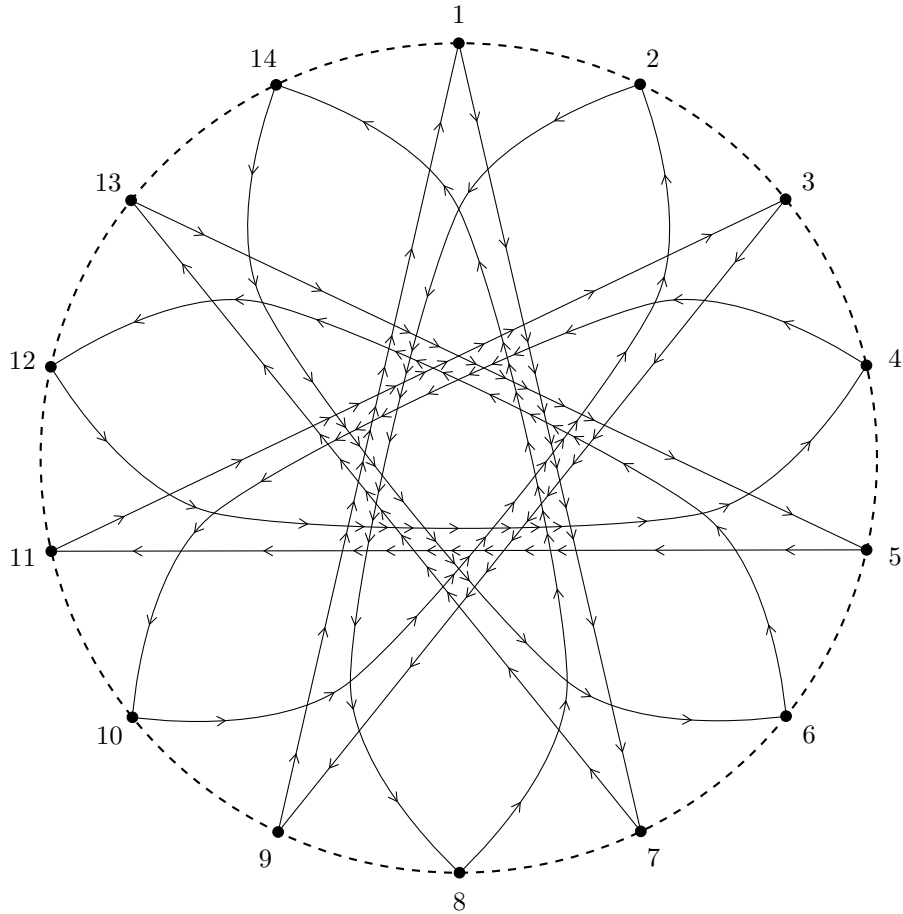
\begin{figure}[H]
\[
  \begin{tikzpicture}[baseline=(bb.base)]

\newcommand{\goodarrow}{\arrow{angle 60}}
\newcommand{\bstart}{115.7} 
\newcommand{\nth}{25.7} 
\newcommand{\radius}{5.5cm} 
\newcommand{\eps}{11pt} 
\newcommand{\dotrad}{0.07cm} 

\path (0,0) node (bb) {};


\draw (0,0) circle(\radius) [thick,dashed];

\foreach \n in {1,...,14}
{ \coordinate (b\n) at (\bstart-\nth*\n:\radius);
  \draw (\bstart-\nth*\n:\radius+\eps) node {$\n$}; }
  

\foreach \n in {1,...,14} {\draw (b\n) circle(\dotrad) [fill=black];}


\draw  plot [smooth]
coordinates {(b1) (b7)}
[ postaction=decorate, decoration={markings,
  mark= at position 0.1 with \goodarrow, mark= at position 0.26 with \goodarrow,
  mark= at position 0.38 with \goodarrow, mark= at position 0.42 with \goodarrow,
  mark= at position 0.47 with \goodarrow, mark= at position 0.51 with \goodarrow,
  mark= at position 0.54 with \goodarrow, mark= at position 0.59 with \goodarrow,
  mark= at position 0.63 with \goodarrow, mark= at position 0.74 with \goodarrow,
 mark= at position 0.9 with \goodarrow}];

\draw  plot [smooth]
coordinates {(b3) (b9)}[ postaction=decorate, decoration={markings,
  mark= at position 0.1 with \goodarrow, mark= at position 0.26 with \goodarrow,
  mark= at position 0.38 with \goodarrow, mark= at position 0.42 with \goodarrow,
  mark= at position 0.47 with \goodarrow, mark= at position 0.51 with \goodarrow,
  mark= at position 0.54 with \goodarrow, mark= at position 0.59 with \goodarrow,
  mark= at position 0.63 with \goodarrow, mark= at position 0.74 with \goodarrow,
 mark= at position 0.9 with \goodarrow}];
\draw  plot [smooth]
coordinates {(b5) (b11)}[ postaction=decorate, decoration={markings,
  mark= at position 0.1 with \goodarrow, mark= at position 0.26 with \goodarrow,
  mark= at position 0.38 with \goodarrow, mark= at position 0.42 with \goodarrow,
  mark= at position 0.47 with \goodarrow, mark= at position 0.51 with \goodarrow,
  mark= at position 0.54 with \goodarrow, mark= at position 0.59 with \goodarrow,
  mark= at position 0.63 with \goodarrow, mark= at position 0.74 with \goodarrow,
 mark= at position 0.9 with \goodarrow}];
\draw  plot [smooth]
coordinates {(b7) (b13)}[ postaction=decorate, decoration={markings,
  mark= at position 0.1 with \goodarrow, mark= at position 0.26 with \goodarrow,
  mark= at position 0.38 with \goodarrow, mark= at position 0.42 with \goodarrow,
  mark= at position 0.47 with \goodarrow, mark= at position 0.51 with \goodarrow,
  mark= at position 0.54 with \goodarrow, mark= at position 0.59 with \goodarrow,
  mark= at position 0.63 with \goodarrow, mark= at position 0.74 with \goodarrow,
 mark= at position 0.9 with \goodarrow}];
\draw  plot [smooth]
coordinates {(b9) (b1)}[ postaction=decorate, decoration={markings,
  mark= at position 0.1 with \goodarrow, mark= at position 0.26 with \goodarrow,
  mark= at position 0.38 with \goodarrow, mark= at position 0.42 with \goodarrow,
  mark= at position 0.47 with \goodarrow, mark= at position 0.51 with \goodarrow,
  mark= at position 0.54 with \goodarrow, mark= at position 0.59 with \goodarrow,
  mark= at position 0.63 with \goodarrow, mark= at position 0.74 with \goodarrow,
 mark= at position 0.9 with \goodarrow}];
\draw  plot [smooth]
coordinates {(b11) (b3)}[ postaction=decorate, decoration={markings,
  mark= at position 0.1 with \goodarrow, mark= at position 0.26 with \goodarrow,
  mark= at position 0.38 with \goodarrow, mark= at position 0.42 with \goodarrow,
  mark= at position 0.47 with \goodarrow, mark= at position 0.51 with \goodarrow,
  mark= at position 0.54 with \goodarrow, mark= at position 0.59 with \goodarrow,
  mark= at position 0.63 with \goodarrow, mark= at position 0.74 with \goodarrow,
 mark= at position 0.9 with \goodarrow}];
\draw  plot [smooth]
coordinates {(b13) (b5)}[ postaction=decorate, decoration={markings,
  mark= at position 0.1 with \goodarrow, mark= at position 0.26 with \goodarrow,
  mark= at position 0.38 with \goodarrow, mark= at position 0.42 with \goodarrow,
  mark= at position 0.47 with \goodarrow, mark= at position 0.51 with \goodarrow,
  mark= at position 0.54 with \goodarrow, mark= at position 0.59 with \goodarrow,
  mark= at position 0.63 with \goodarrow, mark= at position 0.74 with \goodarrow,
 mark= at position 0.9 with \goodarrow}];

\draw  plot [smooth]
coordinates {(b10) (90+6*\nth:.6*\radius) (90-2*\nth:.6*\radius)(b2)}[ postaction=decorate, decoration={markings,
  mark= at position 0.1 with \goodarrow, mark= at position 0.22 with \goodarrow,
  mark= at position 0.34 with \goodarrow, mark= at position 0.4 with \goodarrow,
  mark= at position 0.425 with \goodarrow, mark= at position 0.452 with \goodarrow,
  mark= at position 0.5 with \goodarrow, mark= at position 0.555 with \goodarrow,
  mark= at position 0.586 with \goodarrow, mark= at position 0.61 with \goodarrow,
  mark= at position 0.7 with \goodarrow, mark= at position 0.79 with \goodarrow,
 mark= at position 0.9 with \goodarrow}];
\draw  plot [smooth]
coordinates {(b2) (90:.6*\radius) (90-8*\nth:.6*\radius)(b8)}[ postaction=decorate, decoration={markings,
  mark= at position 0.1 with \goodarrow, mark= at position 0.22 with \goodarrow,
  mark= at position 0.34 with \goodarrow, mark= at position 0.4 with \goodarrow,
  mark= at position 0.425 with \goodarrow, mark= at position 0.452 with \goodarrow,
  mark= at position 0.5 with \goodarrow, mark= at position 0.555 with \goodarrow,
  mark= at position 0.586 with \goodarrow, mark= at position 0.61 with \goodarrow,
  mark= at position 0.7 with \goodarrow, mark= at position 0.79 with \goodarrow,
 mark= at position 0.9 with \goodarrow}];
\draw  plot [smooth]
coordinates {(b8) (90+8*\nth:.6*\radius) (90:.6*\radius)(b14)}[ postaction=decorate, decoration={markings,
  mark= at position 0.1 with \goodarrow, mark= at position 0.22 with \goodarrow,
  mark= at position 0.34 with \goodarrow, mark= at position 0.4 with \goodarrow,
  mark= at position 0.425 with \goodarrow, mark= at position 0.452 with \goodarrow,
  mark= at position 0.5 with \goodarrow, mark= at position 0.555 with \goodarrow,
  mark= at position 0.586 with \goodarrow, mark= at position 0.61 with \goodarrow,
  mark= at position 0.7 with \goodarrow, mark= at position 0.79 with \goodarrow,
 mark= at position 0.9 with \goodarrow}];
\draw  plot [smooth]
coordinates {(b14) (90+2*\nth:.6*\radius) (90-6*\nth:.6*\radius)(b6)}[ postaction=decorate, decoration={markings,
  mark= at position 0.1 with \goodarrow, mark= at position 0.22 with \goodarrow,
  mark= at position 0.34 with \goodarrow, mark= at position 0.4 with \goodarrow,
  mark= at position 0.425 with \goodarrow, mark= at position 0.452 with \goodarrow,
  mark= at position 0.5 with \goodarrow, mark= at position 0.555 with \goodarrow,
  mark= at position 0.586 with \goodarrow, mark= at position 0.61 with \goodarrow,
  mark= at position 0.7 with \goodarrow, mark= at position 0.79 with \goodarrow,
 mark= at position 0.9 with \goodarrow}];
\draw  plot [smooth]
coordinates {(b6) (90+10*\nth:.6*\radius) (90+2*\nth:.6*\radius)(b12)}[ postaction=decorate, decoration={markings,
  mark= at position 0.1 with \goodarrow, mark= at position 0.22 with \goodarrow,
  mark= at position 0.34 with \goodarrow, mark= at position 0.4 with \goodarrow,
  mark= at position 0.425 with \goodarrow, mark= at position 0.452 with \goodarrow,
  mark= at position 0.5 with \goodarrow, mark= at position 0.555 with \goodarrow,
  mark= at position 0.586 with \goodarrow, mark= at position 0.61 with \goodarrow,
  mark= at position 0.7 with \goodarrow, mark= at position 0.79 with \goodarrow,
 mark= at position 0.9 with \goodarrow}];
\draw  plot [smooth]
coordinates {(b12) (90+4*\nth:.6*\radius) (90-4*\nth:.6*\radius)(b4)}[ postaction=decorate, decoration={markings,
  mark= at position 0.1 with \goodarrow, mark= at position 0.22 with \goodarrow,
  mark= at position 0.34 with \goodarrow, mark= at position 0.4 with \goodarrow,
  mark= at position 0.425 with \goodarrow, mark= at position 0.452 with \goodarrow,
  mark= at position 0.5 with \goodarrow, mark= at position 0.555 with \goodarrow,
  mark= at position 0.586 with \goodarrow, mark= at position 0.61 with \goodarrow,
  mark= at position 0.7 with \goodarrow, mark= at position 0.79 with \goodarrow,
 mark= at position 0.9 with \goodarrow}];
\draw  plot [smooth]
coordinates {(b4) (90-2*\nth:.6*\radius) (90+4*\nth:.6*\radius)(b10)}[ postaction=decorate, decoration={markings,
  mark= at position 0.1 with \goodarrow, mark= at position 0.22 with \goodarrow,
  mark= at position 0.34 with \goodarrow, mark= at position 0.4 with \goodarrow,
  mark= at position 0.425 with \goodarrow, mark= at position 0.452 with \goodarrow,
  mark= at position 0.5 with \goodarrow, mark= at position 0.555 with \goodarrow,
  mark= at position 0.586 with \goodarrow, mark= at position 0.61 with \goodarrow,
  mark= at position 0.7 with \goodarrow, mark= at position 0.79 with \goodarrow,
 mark= at position 0.9 with \goodarrow}];

 \end{tikzpicture}
\]
\caption{A  symmetric $(6,14)$-Postnikov diagram.}
\label{fig:post614}
\end{figure}

\begin{figure}[H]
  \vspace{1cm}
\[  
  \begin{tikzpicture}[baseline=(bb.base), 
	xyarrow/.style = {black,very thin, -to}
      ]

\newcommand{\bstartb}{90} 
\newcommand{\bstarta}{90-18} 
\newcommand{\nth}{36} 
\newcommand{\radius}{3.5cm} 
\newcommand{\dotrad}{.9pt} 

\path (0,0) node (bb){};

\foreach \n in {1,...,10}
{ \node (b\n) at (\bstarta-\nth*\n:\radius) {};
  \draw (b\n) circle(\dotrad) [fill=black];
}

\foreach\n in {1, ..., 5 }
{ 
  \node (e\n) at (\bstartb-\nth*\n*2:\radius*.45){};
  \draw (e\n) circle(\dotrad) [fill=black];
}

 \foreach\n in {1, ..., 4 }{
   \draw [xyarrow] (e\n) edge (e\the\numexpr\n+1\relax);
}

\foreach \n in {1,...,4}{
  \draw [xyarrow] (b\the\numexpr\n*2+1\relax) edge (b\the\numexpr\n*2\relax);
  \draw [xyarrow] (b\the\numexpr\n*2+1\relax) edge (b\the\numexpr\n*2+2\relax);
  \draw [xyarrow] (e\the\numexpr\n\relax) edge (b\the\numexpr\n*2-1\relax);
  \draw [xyarrow] (b\the\numexpr\n*2\relax) edge (e\the\numexpr\n\relax);
}

  \draw [xyarrow] (e5) edge (b9);
  \draw [xyarrow] (e5) edge (e1);
  \draw [xyarrow] (b10) edge (e5);
  \draw [xyarrow] (b1) edge (b10);
  \draw [xyarrow] (b1) edge (b2);

\end{tikzpicture}  
\hspace{.5cm}
\begin{tikzpicture}[baseline=(bb.base), 
	xyarrow/.style = {black, very thin, -to}
      ]

\newcommand{\bstartb}{90} 
\newcommand{\bstarta}{90-25.71/2} 
\newcommand{\nth}{25.71} 
\newcommand{\radius}{3.5cm} 
\newcommand{\dotrad}{.9pt} 

\path (0,0) node (bb){};

\foreach \n in {1,...,14}
{   \node (c\n) at (\bstarta-\nth*\n:\radius){};
 \draw (c\n) circle(\dotrad) [fill=black];
 \node (d\n) at (\bstarta-\nth*\n:\radius*.65){};
\draw (d\n) circle(\dotrad) [fill=black];
}

\foreach\n in {1, ..., 7 }
{ 
  \node (e\n) at (\bstartb-\nth*\n*2:\radius*.3){};
  \draw (e\n) circle(\dotrad) [fill=black];
}

 \foreach\n in {1, ..., 6 }{
   \draw [xyarrow] (e\n) edge (e\the\numexpr\n+1\relax);
}

\foreach \n in {1,...,6}{
  \draw [xyarrow] (d\the\numexpr\n*2+1\relax) edge (d\the\numexpr\n*2\relax);
  \draw [xyarrow] (d\the\numexpr\n*2+1\relax) edge (d\the\numexpr\n*2+2\relax);
  \draw [xyarrow] (c\the\numexpr\n*2\relax) edge (c\the\numexpr\n*2+1\relax);
  \draw [xyarrow] (c\the\numexpr\n*2\relax) edge (c\the\numexpr\n*2-1\relax);
  \draw [xyarrow] (c\the\numexpr\n*2+1\relax) edge (d\the\numexpr\n*2+1\relax);
  \draw [xyarrow] (d\the\numexpr\n*2\relax) edge (c\the\numexpr\n*2\relax);
  \draw [xyarrow] (e\the\numexpr\n\relax) edge (d\the\numexpr\n*2-1\relax);
  \draw [xyarrow] (d\the\numexpr\n*2\relax) edge (e\the\numexpr\n\relax);

}
\draw [xyarrow] (d1) edge (d2);
\draw [xyarrow] (c1) edge (d1);
\draw [xyarrow] (d1) edge (d14);
\draw [xyarrow] (c14) edge (c13);
\draw [xyarrow] (c14) edge (c1);
\draw [xyarrow] (d14) edge (c14);
\draw [xyarrow] (d14) edge (e7);
\draw [xyarrow] (e7) edge (d13);
\draw [xyarrow] (e7) edge (e1);

    \end{tikzpicture}
\]
\caption{The self-injective quivers with potential $\on{Cob}^-(5)$ and $\on{Cob}^-(7)$ (for $(k,n) = (6,10)$ and $(8,14)$).}
\label{fig:cob1}
\end{figure}
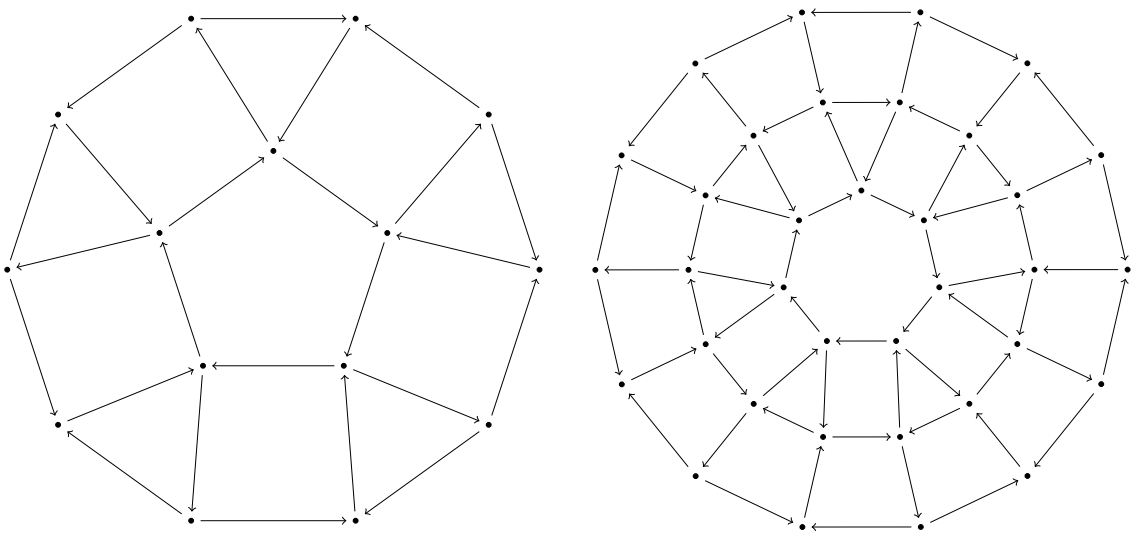

\newpage
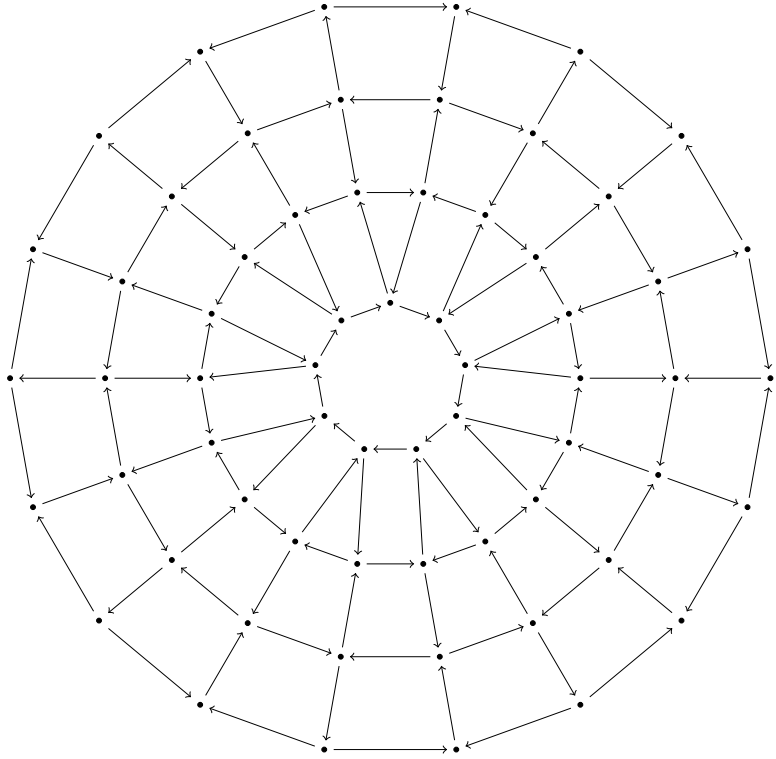
\begin{figure}[H]
  \vspace{1cm}
\[    \begin{tikzpicture}[baseline=(bb.base), 
	xyarrow/.style = {black, very thin, -to}
      ]

\newcommand{\bstartb}{90} 
\newcommand{\bstarta}{80} 
\newcommand{\nth}{20} 
\newcommand{\radius}{5cm} 
\newcommand{\dotrad}{.9pt} 

\path (0,0) node (bb){};

\foreach \n in {1,...,18}
{ \node (b\n) at (\bstarta-\nth*\n:\radius) {};
  \draw (b\n) circle(\dotrad) [fill=black];
  \node (c\n) at (\bstarta-\nth*\n:\radius*.75){};
 \draw (c\n) circle(\dotrad) [fill=black];
 \node (d\n) at (\bstarta-\nth*\n:\radius*.5){};
\draw (d\n) circle(\dotrad) [fill=black];
}

\foreach\n in {1, ..., 9 }
{ 
  \node (e\n) at (\bstartb-\nth*\n*2:\radius*.2){};
  \draw (e\n) circle(\dotrad) [fill=black];
}

 \foreach\n in {1, ..., 8 }{
   \draw [xyarrow] (e\n) edge (e\the\numexpr\n+1\relax);
}

\foreach \n in {1,...,8}{
  \draw [xyarrow] (d\the\numexpr\n*2+1\relax) edge (d\the\numexpr\n*2\relax);
  \draw [xyarrow] (d\the\numexpr\n*2+1\relax) edge (d\the\numexpr\n*2+2\relax);
  \draw [xyarrow] (c\the\numexpr\n*2\relax) edge (c\the\numexpr\n*2+1\relax);
  \draw [xyarrow] (c\the\numexpr\n*2\relax) edge (c\the\numexpr\n*2-1\relax);
  \draw [xyarrow] (b\the\numexpr\n*2+1\relax) edge (b\the\numexpr\n*2\relax);
  \draw [xyarrow] (b\the\numexpr\n*2+1\relax) edge (b\the\numexpr\n*2+2\relax);
  \draw [xyarrow] (c\the\numexpr\n*2+1\relax) edge (d\the\numexpr\n*2+1\relax);
  \draw [xyarrow] (d\the\numexpr\n*2\relax) edge (c\the\numexpr\n*2\relax);
  \draw [xyarrow] (c\the\numexpr\n*2+1\relax) edge (b\the\numexpr\n*2+1\relax);
  \draw [xyarrow] (b\the\numexpr\n*2\relax) edge (c\the\numexpr\n*2\relax);
  \draw [xyarrow] (e\the\numexpr\n\relax) edge (d\the\numexpr\n*2-1\relax);
  \draw [xyarrow] (d\the\numexpr\n*2\relax) edge (e\the\numexpr\n\relax);

}
\draw [xyarrow] (e9) edge (e1);
\draw [xyarrow] (e9) edge (d17);
\draw [xyarrow] (d18) edge (e9);
\draw [xyarrow] (d1) edge (d18);
\draw [xyarrow] (d1) edge (d2);
\draw [xyarrow] (b1) edge (b18);
\draw [xyarrow] (b1) edge (b2);
\draw [xyarrow] (c18) edge (c1);
\draw [xyarrow] (c18) edge (c17);
\draw [xyarrow] (c1) edge (d1);
\draw [xyarrow] (c1) edge (b1);
\draw [xyarrow] (b18) edge (c18);
\draw [xyarrow] (d18) edge (c18);

    \end{tikzpicture}
\]
\caption{The self-injective quiver with potential $ \on{Cob}^-(9)$ (for $(k,n) = (10,18)$).}
\label{fig:cob2}
\end{figure}

\begin{figure}[H]
  \vspace{1cm}
  \[
    \resizebox{9cm}{!}{\begin{xy} 0;<.8pt,0pt>:<0pt,.8pt>::
(336,216) *+{\bullet} ="0",
(360,288) *+{\bullet} ="1",
(300,336) *+{\bullet} ="2",
(240,288) *+{\bullet} ="3",
(264,216) *+{\bullet} ="4",
(336,144) *+{\bullet} ="5",
(408,192) *+{\bullet} ="6",
(432,264) *+{\bullet} ="7",
(408,348) *+{\bullet} ="8",
(348,396) *+{\bullet} ="9",
(252,396) *+{\bullet} ="10",
(192,348) *+{\bullet} ="11",
(168,264) *+{\bullet} ="12",
(192,192) *+{\bullet} ="13",
(264,144) *+{\bullet} ="14",
(456,120) *+{\bullet} ="15",
(504,48) *+{\bullet} ="16",
(540,204) *+{\bullet} ="17",
(492,372) *+{\bullet} ="18",
(576,396) *+{\bullet} ="19",
(444,492) *+{\bullet} ="20",
(252,492) *+{\bullet} ="21",
(252,588) *+{\bullet} ="22",
(132,456) *+{\bullet} ="23",
(84,288) *+{\bullet} ="24",
(0,312) *+{\bullet} ="25",
(72,168) *+{\bullet} ="26",
(216,72) *+{\bullet} ="27",
(168,0) *+{\bullet} ="28",
(312,24) *+{\bullet} ="29",
(432,0) *+{\bullet} ="30",
(384,72) *+{\bullet} ="31",
(600,312) *+{\bullet} ="32",
(516,288) *+{\bullet} ="33",
(348,588) *+{\bullet} ="34",
(348,492) *+{\bullet} ="35",
(24,396) *+{\bullet} ="36",
(108,372) *+{\bullet} ="37",
(144,120) *+{\bullet} ="38",
(96,48) *+{\bullet} ="39",
"0", {\ar"1"},
"4", {\ar"0"},
"0", {\ar"5"},
"6", {\ar"0"},
"1", {\ar"2"},
"1", {\ar"7"},
"8", {\ar"1"},
"2", {\ar"3"},
"2", {\ar"9"},
"10", {\ar"2"},
"3", {\ar"4"},
"3", {\ar"11"},
"12", {\ar"3"},
"4", {\ar"13"},
"14", {\ar"4"},
"5", {\ar"6"},
"5", {\ar"14"},
"15", {\ar"5"},
"29", {\ar"5"},
"5", {\ar"31"},
"7", {\ar"6"},
"6", {\ar"15"},
"7", {\ar"8"},
"17", {\ar"7"},
"18", {\ar"7"},
"7", {\ar"33"},
"9", {\ar"8"},
"8", {\ar"18"},
"9", {\ar"10"},
"20", {\ar"9"},
"21", {\ar"9"},
"9", {\ar"35"},
"11", {\ar"10"},
"10", {\ar"21"},
"11", {\ar"12"},
"23", {\ar"11"},
"24", {\ar"11"},
"11", {\ar"37"},
"13", {\ar"12"},
"12", {\ar"24"},
"13", {\ar"14"},
"26", {\ar"13"},
"27", {\ar"13"},
"13", {\ar"38"},
"14", {\ar"27"},
"15", {\ar"16"},
"31", {\ar"15"},
"16", {\ar"17"},
"16", {\ar"30"},
"33", {\ar"17"},
"18", {\ar"19"},
"33", {\ar"18"},
"19", {\ar"20"},
"19", {\ar"32"},
"35", {\ar"20"},
"21", {\ar"22"},
"35", {\ar"21"},
"22", {\ar"23"},
"22", {\ar"34"},
"37", {\ar"23"},
"24", {\ar"25"},
"37", {\ar"24"},
"25", {\ar"26"},
"25", {\ar"36"},
"38", {\ar"26"},
"27", {\ar"28"},
"38", {\ar"27"},
"28", {\ar"29"},
"28", {\ar"39"},
"31", {\ar"29"},
"30", {\ar"31"},
"32", {\ar"33"},
"34", {\ar"35"},
"36", {\ar"37"},
"39", {\ar"38"},
\end{xy}
}
\]
\caption{A self-injective quiver with potential for $(k,n) = (6,15)$.}
\label{fig:quiv615}
\end{figure}
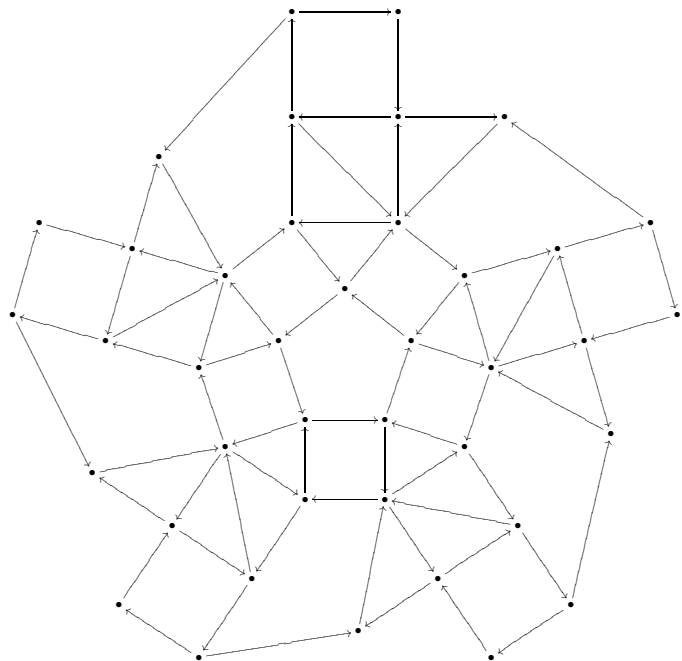

\newpage
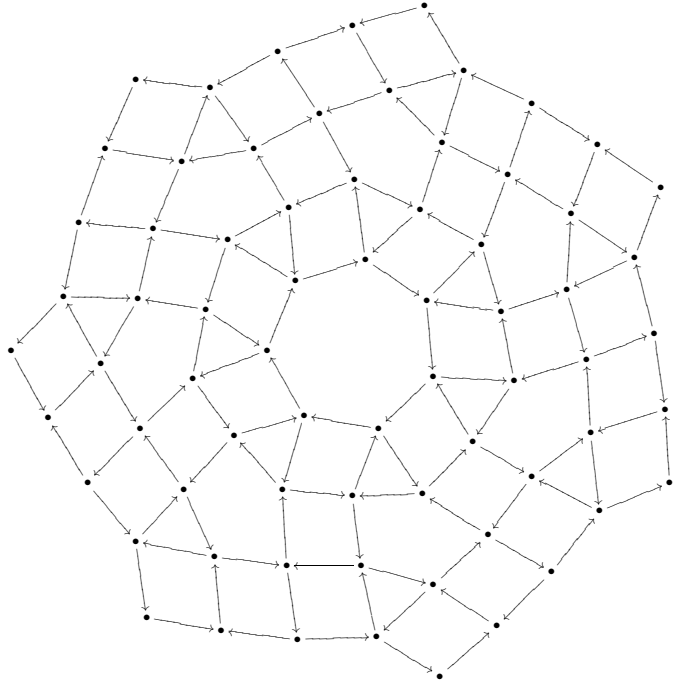
\begin{figure}[H]
  \[
    \resizebox{9cm}{!}{
\begin{xy} 0;<1.5pt,0pt>:<0pt,-1.5pt>:: 
(117,159) *+{\bullet} ="0",
(130,127) *+{\bullet} ="1",
(162,117) *+{\bullet} ="2",
(190,136) *+{\bullet} ="3",
(193,171) *+{\bullet} ="4",
(168,195) *+{\bullet} ="5",
(134,189) *+{\bullet} ="6",
(89,140) *+{\bullet} ="7",
(99,108) *+{\bullet} ="8",
(127,93) *+{\bullet} ="9",
(157,80) *+{\bullet} ="10",
(187,94) *+{\bullet} ="11",
(215,110) *+{\bullet} ="12",
(224,141) *+{\bullet} ="13",
(230,173) *+{\bullet} ="14",
(211,201) *+{\bullet} ="15",
(188,225) *+{\bullet} ="16",
(156,226) *+{\bullet} ="17",
(124,223) *+{\bullet} ="18",
(102,198) *+{\bullet} ="19",
(83,172) *+{\bullet} ="20",
(65,103) *+{\bullet} ="21",
(78,72) *+{\bullet} ="22",
(111,66) *+{\bullet} ="23",
(141,50) *+{\bullet} ="24",
(173,39) *+{\bullet} ="25",
(197,63) *+{\bullet} ="26",
(227,78) *+{\bullet} ="27",
(256,96) *+{\bullet} ="28",
(254,131) *+{\bullet} ="29",
(263,163) *+{\bullet} ="30",
(265,197) *+{\bullet} ="31",
(238,217) *+{\bullet} ="32",
(218,244) *+{\bullet} ="33",
(193,267) *+{\bullet} ="34",
(160,258) *+{\bullet} ="35",
(126,258) *+{\bullet} ="36",
(93,254) *+{\bullet} ="37",
(79,223) *+{\bullet} ="38",
(59,195) *+{\bullet} ="39",
(41,165) *+{\bullet} ="40",
(58,135) *+{\bullet} ="41",
(122,21) *+{\bullet} ="42",
(156,9) *+{\bullet} ="43",
(189,0) *+{\bullet} ="44",
(207,30) *+{\bullet} ="45",
(238,45) *+{\bullet} ="46",
(268,64) *+{\bullet} ="47",
(297,84) *+{\bullet} ="48",
(285,116) *+{\bullet} ="49",
(294,151) *+{\bullet} ="50",
(299,186) *+{\bullet} ="51",
(301,220) *+{\bullet} ="52",
(269,233) *+{\bullet} ="53",
(247,261) *+{\bullet} ="54",
(222,286) *+{\bullet} ="55",
(196,309) *+{\bullet} ="56",
(167,291) *+{\bullet} ="57",
(131,292) *+{\bullet} ="58",
(96,288) *+{\bullet} ="59",
(62,282) *+{\bullet} ="60",
(57,247) *+{\bullet} ="61",
(35,220) *+{\bullet} ="62",
(17,190) *+{\bullet} ="63",
(0,159) *+{\bullet} ="64",
(24,134) *+{\bullet} ="65",
(31,100) *+{\bullet} ="66",
(43,66) *+{\bullet} ="67",
(57,34) *+{\bullet} ="68",
(91,38) *+{\bullet} ="69",
"0", {\ar"1"},
"6", {\ar"0"},
"7", {\ar"0"},
"0", {\ar"20"},
"1", {\ar"2"},
"1", {\ar"8"},
"9", {\ar"1"},
"2", {\ar"3"},
"2", {\ar"10"},
"11", {\ar"2"},
"3", {\ar"4"},
"3", {\ar"12"},
"13", {\ar"3"},
"4", {\ar"5"},
"4", {\ar"14"},
"15", {\ar"4"},
"5", {\ar"6"},
"5", {\ar"16"},
"17", {\ar"5"},
"6", {\ar"18"},
"19", {\ar"6"},
"8", {\ar"7"},
"20", {\ar"7"},
"7", {\ar"41"},
"8", {\ar"9"},
"21", {\ar"8"},
"10", {\ar"9"},
"9", {\ar"23"},
"10", {\ar"11"},
"24", {\ar"10"},
"12", {\ar"11"},
"11", {\ar"26"},
"12", {\ar"13"},
"27", {\ar"12"},
"14", {\ar"13"},
"13", {\ar"29"},
"14", {\ar"15"},
"30", {\ar"14"},
"16", {\ar"15"},
"15", {\ar"32"},
"16", {\ar"17"},
"33", {\ar"16"},
"18", {\ar"17"},
"17", {\ar"35"},
"18", {\ar"19"},
"36", {\ar"18"},
"20", {\ar"19"},
"19", {\ar"38"},
"39", {\ar"20"},
"22", {\ar"21"},
"41", {\ar"21"},
"21", {\ar"66"},
"23", {\ar"22"},
"67", {\ar"22"},
"22", {\ar"69"},
"23", {\ar"24"},
"69", {\ar"23"},
"25", {\ar"24"},
"24", {\ar"42"},
"26", {\ar"25"},
"43", {\ar"25"},
"25", {\ar"45"},
"26", {\ar"27"},
"45", {\ar"26"},
"28", {\ar"27"},
"27", {\ar"46"},
"29", {\ar"28"},
"47", {\ar"28"},
"28", {\ar"49"},
"29", {\ar"30"},
"49", {\ar"29"},
"31", {\ar"30"},
"30", {\ar"50"},
"32", {\ar"31"},
"51", {\ar"31"},
"31", {\ar"53"},
"32", {\ar"33"},
"53", {\ar"32"},
"34", {\ar"33"},
"33", {\ar"54"},
"35", {\ar"34"},
"55", {\ar"34"},
"34", {\ar"57"},
"35", {\ar"36"},
"57", {\ar"35"},
"37", {\ar"36"},
"36", {\ar"58"},
"38", {\ar"37"},
"59", {\ar"37"},
"37", {\ar"61"},
"38", {\ar"39"},
"61", {\ar"38"},
"40", {\ar"39"},
"39", {\ar"62"},
"41", {\ar"40"},
"63", {\ar"40"},
"40", {\ar"65"},
"65", {\ar"41"},
"42", {\ar"43"},
"42", {\ar"69"},
"44", {\ar"43"},
"45", {\ar"44"},
"46", {\ar"45"},
"46", {\ar"47"},
"48", {\ar"47"},
"49", {\ar"48"},
"50", {\ar"49"},
"50", {\ar"51"},
"52", {\ar"51"},
"53", {\ar"52"},
"54", {\ar"53"},
"54", {\ar"55"},
"56", {\ar"55"},
"57", {\ar"56"},
"58", {\ar"57"},
"58", {\ar"59"},
"60", {\ar"59"},
"61", {\ar"60"},
"62", {\ar"61"},
"62", {\ar"63"},
"64", {\ar"63"},
"65", {\ar"64"},
"66", {\ar"65"},
"66", {\ar"67"},
"68", {\ar"67"},
"69", {\ar"68"},
\end{xy}
}
\]
\caption{A self-injective quiver with potential for $(k,n) = (6,21)$.}
\label{fig:quiv621}
\end{figure}

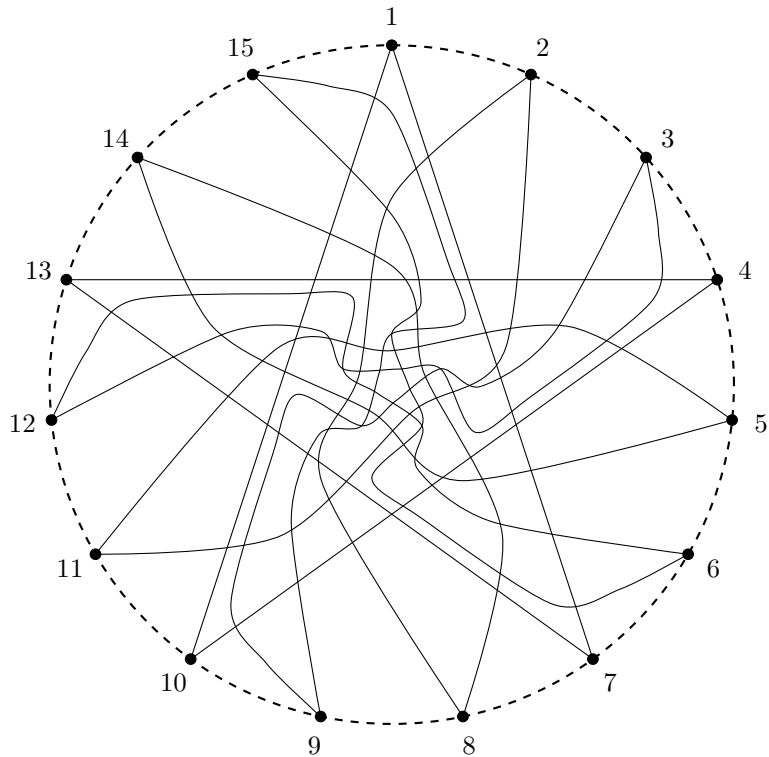
\begin{figure}[H]
  \[
    \begin{tikzpicture}[baseline=(bb.base)]

\newcommand{\strarrow}{\arrow{angle 60}}
\newcommand{\bstart}{114} %
\newcommand{\nth}{24} 
\newcommand{\radius}{4.5cm} 
\newcommand{\eps}{11pt} 
\newcommand{\dotrad}{0.07cm} 

\path (0,0) node (bb){};

\draw (0,0) circle(\radius) [thick, dashed];
\foreach \n in {1,...,15}
{ \coordinate (b\n) at (\bstart-\nth*\n:\radius);
  \draw (\bstart-\nth*\n:\radius+\eps) node {$\n$}; }
  

\foreach \n in {1,...,15} {\draw (b\n) circle(\dotrad) [fill=black];}

\foreach\n in {1, 7, 13, 4, 10}
{ 
  \coordinate (c\n) at (\bstart-\nth*\n-6*\nth:\radius);
  \draw plot coordinates {(b\n) (c\n)};
}

\foreach\n in {2, 8, 14, 5, 11}
{
  \coordinate (c\n) at (\bstart-\nth*\n-6*\nth:\radius);
  \draw plot [smooth]
  coordinates {(b\n)(\bstart-\nth*\n+1*\nth:\radius*.55) (\bstart-\nth*\n+4*\nth:\radius*.1)(\bstart-\nth*\n+7*\nth:\radius*.35)(c\n)};
}
\foreach\n in {3, 9, 15, 6, 12}
{
  \coordinate (c\n) at (\bstart-\nth*\n-6*\nth:\radius);
  \draw plot [smooth]
  coordinates {(b\n)(\bstart-\nth*\n-1*\nth:\radius*.5)(\bstart-\nth*\n-1.8*\nth:\radius*.26)(\bstart-\nth*\n-1*\nth:\radius*.15)(\bstart-\nth*\n-5*\nth:\radius*.05)
  (\bstart-\nth*\n-7*\nth:\radius*.15)(\bstart-\nth*\n+6*\nth:\radius*.28)(\bstart-\nth*\n+7*\nth:\radius*.4)
  (\bstart-\nth*\n+8*\nth:\radius*.8)(\bstart-\nth*\n+8.5*\nth:\radius*.9)(c\n)
  };
}

    \end{tikzpicture}
  \]
  \caption{The Postnikov diagram corresponding to the quiver of Figure \ref{fig:quiv615}.}
  \label{fig:post615}

\end{figure}

\newpage
\begin{figure}[H]
  \[
    \begin{tikzpicture}[baseline=(bb.base)]

\newcommand{\strarrow}{\arrow{angle 60}}
\newcommand{\bstart}{107.14} 
\newcommand{\nth}{17.14} 
\newcommand{\radius}{4.5cm} 
\newcommand{\eps}{11pt} 
\newcommand{\dotrad}{0.07cm} 

\path (0,0) node (bb){};

\draw (0,0) circle(\radius) [thick, dashed];
\foreach \n in {1,...,21}
{ \coordinate (b\n) at (\bstart-\nth*\n:\radius);
  \draw (\bstart-\nth*\n:\radius+\eps) node {$\n$}; }
  

\foreach \n in {1,...,21} {\draw (b\n) circle(\dotrad) [fill=black];}


\draw plot
coordinates {(b1) (b7)}
;
\draw plot
coordinates {(b7) (b13)}
;
\draw plot
coordinates {(b13) (b19)}
;
\draw plot
coordinates {(b19) (b4)}
;
\draw plot
coordinates {(b4) (b10)}
;
\draw plot
coordinates {(b10) (b16)}
;
\draw plot
coordinates {(b16) (b1)}
;

\draw plot[smooth]
coordinates{(b3) (90-3*\nth:.80*\radius) (90-6*\nth:.3*\radius) (90-8.5*\nth:.7*\radius) (b9)}
;
\draw plot[smooth]
coordinates{(b9) (90-9*\nth:.80*\radius) (90-12*\nth:.3*\radius) (90-14.5*\nth:.7*\radius)  (b15)}
;
\draw plot[smooth]
coordinates{(b15) (90-15*\nth:.80*\radius) (90-18*\nth:.3*\radius) (90-20.5*\nth:.7*\radius) (b21)}
;
\draw plot[smooth]
coordinates{(b21) (90-21*\nth:.80*\radius) (90-3*\nth:.3*\radius) (90-5.5*\nth:.7*\radius)  (b6)}
;
\draw plot[smooth]
coordinates{(b6) (90-6*\nth:.80*\radius) (90-9*\nth:.3*\radius) (90-11.5*\nth:.7*\radius) (b12)}
;
\draw plot[smooth]
coordinates{(b12) (90-12*\nth:.80*\radius) (90-15*\nth:.3*\radius) (90-17.5*\nth:.7*\radius)  (b18)}
;
\draw plot[smooth]
coordinates{(b18) (90-18*\nth:.80*\radius) (90-21*\nth:.3*\radius) (90-2.5*\nth:.7*\radius) (b3)}
;

\draw plot[smooth]
coordinates {(b2) (90+\nth:.93*\radius)  (90+\nth:.7*\radius)
(90+3*\nth:.5*\radius)(90+6*\nth:.2*\radius) (90+9*\nth:.5*\radius)
(90+11*\nth:.68*\radius)(90+13.5*\nth:.78*\radius)(b8)}
;
\draw plot[smooth]
coordinates {(b8) (90-5*\nth:.93*\radius)  (90-5*\nth:.7*\radius)
(90-3*\nth:.5*\radius)(90-0*\nth:.2*\radius) (90+3*\nth:.5*\radius)
(90+5*\nth:.68*\radius)(90+7.5*\nth:.78*\radius)(b14)}
;
\draw plot[smooth]
coordinates {(b14) (90-11*\nth:.93*\radius)  (90-11*\nth:.7*\radius)
(90-9*\nth:.5*\radius)(90-6*\nth:.2*\radius) (90-3*\nth:.5*\radius)
(90-1*\nth:.68*\radius) (90+1.5*\nth:.78*\radius)(b20)}
;
\draw plot[smooth]
coordinates {(b20) (90-17*\nth:.93*\radius)  (90-17*\nth:.7*\radius)
(90-15*\nth:.5*\radius)(90-12*\nth:.2*\radius) (90-9*\nth:.5*\radius)
(90-7*\nth:.68*\radius) (90-4.5*\nth:.78*\radius)(b5)}
;
\draw plot[smooth]
coordinates {(b5) (90-23*\nth:.93*\radius)  (90-23*\nth:.7*\radius)
(90-21*\nth:.5*\radius)(90-18*\nth:.2*\radius) (90-15*\nth:.5*\radius)
(90-13*\nth:.68*\radius) (90-10.5*\nth:.78*\radius)(b11)}
;
\draw plot[smooth]
coordinates {(b11) (90-29*\nth:.93*\radius)  (90-29*\nth:.7*\radius)
(90-27*\nth:.5*\radius)(90-24*\nth:.2*\radius) (90-21*\nth:.5*\radius)
(90-19*\nth:.68*\radius) (90-16.5*\nth:.78*\radius)(b17)}
;
\draw plot[smooth]
coordinates {(b17) (90-35*\nth:.93*\radius)  (90-35*\nth:.7*\radius)
(90-33*\nth:.5*\radius)(90-30*\nth:.2*\radius) (90-27*\nth:.5*\radius)
(90-25*\nth:.68*\radius) (90-22.5*\nth:.78*\radius)(b2)}
;

    \end{tikzpicture}
  \]
  \caption{The Postnikov diagram corresponding to the quiver of Figure \ref{fig:quiv621}.}
  \label{fig:post621}
\end{figure}
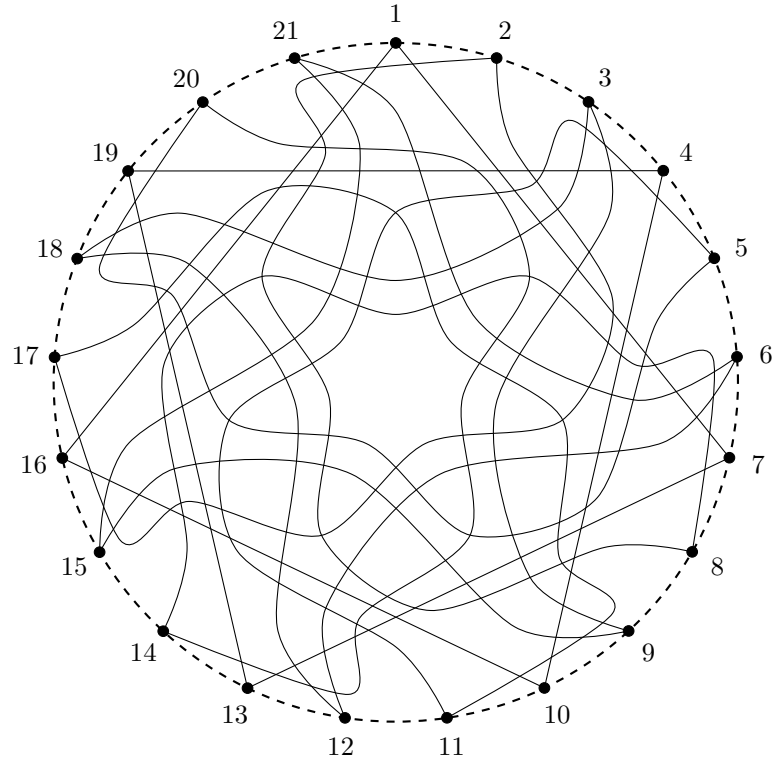

\newpage

\section*{}
\bibliographystyle{alpha}

\end{document}